\documentclass[10pt,a4paper,english]{article}
\usepackage[latin1]{inputenc}
\usepackage{babel}
\usepackage{amsmath,amsfonts,amssymb,amsthm}
\usepackage{bbm}
\usepackage{hyperref}
\usepackage{verbatim}
\usepackage{tikz}

\usetikzlibrary{positioning}

\theoremstyle{plain}
\newtheorem{teo}{Theorem}[section]
\newtheorem{cor}[teo]{Corollary}
\newtheorem{prop}[teo]{Proposition}
\newtheorem{lem}[teo]{Lemma}
\newtheorem{cla}[teo]{Claim}
\newtheorem{rem}[teo]{Remark}

\theoremstyle{defin}
\newtheorem{defin}[teo]{Definition}

\newtheorem{Q}[teo]{Question}

\newcommand{\system}[1]{\mbox{\fontfamily{cmss}\fontshape{n}\fontseries{m}\selectfont#1}}
\newcommand{\ZF}{\system{ZF}}
\newcommand{\ZFC}{\system{ZFC}}
\newcommand{\AC}{\system{AC}}
\newcommand{\GCH}{\system{GCH}}

\newcommand{\DC}{\system{DC}}
\newcommand{\AD}{\system{AD}}
\newcommand{\CH}{\system{CH}}

\DeclareMathOperator{\cof}{cof}
\DeclareMathOperator{\crt}{crt}
\DeclareMathOperator{\OD}{OD}
\DeclareMathOperator{\HOD}{HOD}
\DeclareMathOperator{\Ord}{Ord}
\DeclareMathOperator{\Ult}{Ult}
\DeclareMathOperator{\ot}{ot}
\DeclareMathOperator{\dom}{dom}
\DeclareMathOperator{\ran}{ran}
\DeclareMathOperator{\LST}{LST}
\DeclareMathOperator{\id}{id}
\DeclareMathOperator{\Def}{Def}
\DeclareMathOperator{\Reg}{Reg}
\DeclareMathOperator{\Card}{Card}
\DeclareMathOperator{\lh}{lh}

\title{I0 and rank-into-rank axioms}
\author{Vincenzo Dimonte\footnote{Universit\`{a} degli Studi di Udine, via delle Scienze, 206 33100 Udine (UD) \emph{E-mail address:} \texttt{vincenzo.dimonte@gmail.com}}}

\begin{document}

\maketitle

\begin{abstract}

Just a survey on I0.



\end{abstract}

\section{Informal Introduction to the Introduction}
 Ok, that's it. People who know me also know that it is years that I am ranting about a book about I0, how it is important to write it and publish it, etc. I realized that this particular moment is still right for me to write it, for two reasons: time is scarce, and probably there are still not enough results (or anyway not enough different lines of research). This has the potential of being a very nasty vicious circle: there are not enough results about I0 to write a book, but if nobody writes a book the diffusion of I0 is too limited, and so the number of results does not increase much. 

To avoid this deadlock, I decided to divulge a first draft of such a hypothetical book, hoping to start a ``conversation'' with that. It is literally a first draft, so it is full of errors and in a liquid state. For example, I still haven't decided whether to call it I0 or $I_0$, both notations are used in literature. I feel like it still lacks a vision of the future, a map on where the research could and should going about I0. Many proofs are old but never published, and therefore reconstructed by me, so maybe they are wrong. Maybe even well-known proofs are wrong: for example the only published proof of \ref{Delta0} has something missing, and this was discovered decades after, with a question on MathOverflow. So please, consider this paper as very temporary, and if you spot some problem, do not hesitate to write me!  

Finally, if for some reason you need some result on this paper, do note cite it! Unfortunately, given the direction bibliometric is going, citing this paper is completely useless for my career, even harmful. So, as long as the world does not wake up from this folly, please try to locate the original paper where the result is located and cite that (and if you're thankful, cite a random paper of mine, since this is what mathematics is now\footnote{Of course, this is a joke on the current state of events, not to be taken as a serious suggestion}). I submitted a shorter version of this survey to a journal, if it is published please cite that. 

It would be unfair to ignore the fact that there is already an excellent survey on I0, but I feel it is done with a different scope. If you are interested in I0, please read also \cite{Cramer2}, as it nicely complement the paper you are reading now.

\section{Introduction}

\begin{quote}
 In 1984 Woodin showed [$\dots$] that $\AD^{L(\mathbb{R})}$ follows from a proposition stronger than I1 and straining the limits of consistency [$\dots$] Moreover, Woodin's investigations of I0 was to produce a detailed and coherent structural theory for $L(V_{\delta+1})$.
\end{quote}

\begin{quote}
With attention shifted upward to the strongest large cardinal hypotheses, Woodin boldly formulated I0 just at the edge of the Kunen inconsistency and amplified Martin's embeddings proof to establish the following in early 1984:
\end{quote}

\begin{teo}[Woodin]
\label{thebeginning}
  Assume I0. Then $\AD^{L(\mathbb{R})}$.
\end{teo}

This sparse lines in Kanamori's \emph{The Higher Infinite} \cite{Kanamori} were probably for years the only published trace of the mysterious I0. This fact is even more surprising as \cite{Kanamori} was written in 1994, so ten years after the striking result and even more from its first formulation. What happened? Why all this mystery?

I0 was just one of the many large cardinals that were created in a short span of time: strong cardinals, supercompact cardinals, huge cardinals, Woodin cardinals, rank-into-rank axioms (its smaller siblings), $\dots$ were introduced in a rapid fire, shifting the paradigms of set theory. But I0 could not keep up the pace with all of them: its main victory was short-termed, as $\AD^{L(\mathbb{R})}$ was proven to be of much lower consistency strength, there were no other examples of important propositions implied by I0, and it was definitely too large for inner model theory (in fact, the larger cardinal axiom of all times). So, before even a publication, it quietly disappeared, and with it the proof of Theorem \ref{thebeginning}. 

Still, under the surface there was some activity: the bound between I0 and $\AD^{L(\mathbb{R})}$ ran deeper than just Theorem \ref{thebeginning}, it involved the whole structure of a model of the form $L(V_{\lambda+1})$, that under I0 had properties very similar than those in $L(\mathbb{R})$ under \AD. This peculiarity was explored by Woodin, and the first output was Kafkoulis' lecture notes \emph{An AD-like axiom} \cite{Kafkoulis} (at that time I0 was called Axiom A), a transcription of Woodin's lectures on the argument, that started to be copied and distributed among mathematicians at the beginning of the Nineties. This underground distribution was, of course, very limited (Internet was at its infancy).

In the meantime the rank-into-rank axioms were enjoying a period of popularity, thanks to their connections with algebra, and Laver, in 2001, tried to push some of his results on rank-into-rank (in particular reflection) towards I0. In \cite{Laver4} we can find therefore the first officially published fragments of the theory of I0, taken from Kafkoulis notes, and in 2004 Kafkoulis himself published a paper on Coding Lemmata that built on such notes \cite{Kafkoulis2}. After that, still silence for many years.

The rhythm of publications changed abruptly in 2011. The reason is Woodins' \emph{Suitable extender models} magnum opus. The aim of such paper (published in two parts) was to push inner model theory up to all large cardinals, so it was necessary to know as much as possible as the top of the hierarchy. Therefore in the sprawling 561 pages there are whole sections dedicated to I0, the structure of $L(V_{\lambda+1})$ under it (in a much more developed way respect to Kafkoulis notes) and even stronger axioms. There was a new exciting realm to explore, finally mature, and with connections to other fields of set theory. The people that were Woodin's students in that period (the author included) continued the work, expanding the research on I0 in different directions.

Such a rapid expansion, after a long silence, carries the risk of excessive fragmentation. It is then time to take stock of the situation and pose solid bases on such a venture, not based on underground notes, folklore, and oral tradition. The objective of this paper is therefore to collect all the results on rank-into-rank axioms and I0, with a stress on the latest one, both from the I0 folklore and from the published papers, so that anybody can have all the tools for reading the latest papers or doing research on her own. 

Section \ref{KunensT} glances briefly at some important large cardinals, up until the first inconsistency in the large cardinal hierarchy, i.e., the (non-)existence of $j:V\prec V$. 

Starting from this, Section \ref{rank} is about rank-into-rank axioms, i.e., axioms of the form $j:V_\lambda\prec V_\lambda$ (I3) and $j:V_{\lambda+1}\prec V_{\lambda+1}$ (I1), with a particular attention to the application operator.

In Section \ref{LVlambda} I0 is introduced, and with it all the structural properties of $L(V_{\lambda+1})$ that we can have for free. Much of this part is taken from \cite{Kafkoulis}.

Section \ref{strimp} is a detailed analysis of many (all?) axioms from I3 to I0, with all the proofs of strong implications among them. The techniques of inverse limits and square roots are introduced, that are now a focal centre of the investigations around I0 (see \cite{Cramer}). 

Section \ref{simil} uses again \cite{Kafkoulis} as main source, introducing all the results about similarities between I0 and $\AD^{L(\mathbb{R})}$, including the newest developments: measurability of $\lambda^+$, Coding Lemma, large cardinal properties of $\Theta$, perfect set property, an analogous of weakly homogenous Suslin-ness.

In Section \ref{forcing} all these axioms are treated in relation to forcing, with many results of indestructibility and independence results.

In Section \ref{dissimil} it is noticed how such independence result can actually destroy the similarities between I0 and $\AD^{L(\mathbb{R})}$: the examples treated are the Ultrafilter Axiom, Wadge Lemma, a partition property and Turing Determinacy.

Section \ref{Icarus} casts a quick glance on axioms stronger than I0, introduced in \cite{Woodin}: the Icarus sets.

Section \ref{further} describes highly speculatively some possible future lines of inquiry.

\section{Kunen's Theorem}
\label{KunensT}

To reach the top of the large cardinal hierarchy, it is worth to have a glance at the lower levels before. In the bottom, large cardinal axioms are very combinatoric in nature: an inaccessible cardinal is a cardinal that is regular and closed under exponentiation, a Ramsey cardinal is a cardinal that satisfies a particular partition property, a Rowbottom cardinal is a cardinal that satisfies another partition property, and so on\footnote{Admittedly, this is a very rough classification, as for example indescribable cardinals are difficult to insert in this narrative}. The pivotal moment was when Scott proved that a combinatoric property on a cardinal is equivalent to a global property of the universe:

\begin{defin}[Ulam, 1930]
 A cardinal $\kappa$ is \emph{measurable} iff there exists a $\kappa$-complete ultrafilter on $\kappa$.
\end{defin}

\begin{defin}
 Let $M,\ N$ be sets or classes. Then $j:M\to N$ is an \emph{elementary embedding} and we write $j:M\prec N$ iff it is injective and for any formula $\varphi$ and $a_1,\dots,a_n\in M$, $M\vDash\varphi(a_1,\dots,a_n)$ iff $N\vDash\varphi(j(a_1),\dots,j(a_n))$.
\end{defin}

Fom now on, if we write $j:M\prec N$ we mean that $j$ is not the identity.

So an elementary embedding is a homeomorphism of the logical structure. In particular, $M$ and $N$ must satisfy the same sentences, so for example $M\vDash\ZF$ iff $N\vDash\ZF$. It is easy to see that for all ordinals $\alpha$, $j(\alpha)\geq\alpha$. If $M\vDash\AC$ or $N\subseteq M$, than it is possible to prove that there must exist an ordinal $\alpha$ such that $j(\alpha)>\alpha$. 

\begin{defin}
 Let $M,\ N$ be sets or classes such that $M\vDash\AC$ or $N\subseteq M$. Then the \emph{critical point} of $j$, $\crt(j)$, is the least $\alpha$ such that $j(\alpha)>\alpha$.
\end{defin}

\begin{teo}[Scott, 1961]
 A cardinal $\kappa$ is measurable iff there exist an inner model $M$ and a $j:V\prec M$ such that $\crt(j)=\kappa$.
\end{teo}

This was the start of an intense research on new breed of large cardinal hypotheses, defined more requirements for $M$. For example, one can ask how much of the original universe is in $M$:

\begin{defin}[Gaifman, 1974]
 A cardinal $\kappa$ is $\gamma$-strong iff there exists $j:V\prec M$, $\crt(j)=\kappa$, $\gamma<j(\kappa)$ and $V_{\kappa+\gamma}\subseteq M$.
\end{defin}

Or one can ask how close $M$ is. The starting point is to notice that if $j:V\prec M$ witnesses that $\kappa$ is measurable, then $M^\kappa\subseteq M$.

\begin{defin}[Solovay, Reinhardt]
 A cardinal $\kappa$ is $\gamma$-supercompact iff there exists $j:V\prec M$, $\crt(j)=\kappa$, $\gamma<j(\kappa)$ and $M^\gamma\subseteq M$.
\end{defin}

It is immediate to see that if $\kappa$ is $\gamma$-supercompact then it is $\gamma$-strong. All these definitions fix $j(\kappa)$ as the limit of the closeness of $M$. Overcoming this means to define stronger large cardinals:

\begin{defin}[Gaifman, 1974]
  A cardinal $\kappa$ is superstrong iff there exists $j:V\prec M$, $\crt(j)=\kappa$ and $V_{j(\kappa)}\subseteq M$.
\end{defin} 

\begin{defin}[Kunen, 1972]
 A cardinal $\kappa$ is huge iff there exists $j:V\prec M$, $\crt(j)=\kappa$ and $M^{j(\kappa)}\subseteq M$.
\end{defin}

We can go on, defining stronger and stronger large cardinals:

\begin{defin}
 Let $M,\ N$ be sets or classes such that $M\vDash\AC$ or $N\subseteq M$. Then the \emph{critical sequence} of $j:M\prec N$, $\langle\kappa_n:n\in\omega\rangle$, is defined as:
 \begin{itemize}
  \item $\kappa_0=\crt(j)$;
	\item $\kappa_{n+1}=j(\kappa_n)$.
 \end{itemize}
\end{defin}

\begin{defin}
  A cardinal $\kappa$ is $n$-superstrong iff there exists $j:V\prec M$, $\crt(j)=\kappa$ and $V_{j(\kappa_n)}\subseteq M$.
\end{defin} 

\begin{defin}
 A cardinal $\kappa$ is $n$-huge iff there exists $j:V\prec M$, $\crt(j)=\kappa$ and $M^{j(\kappa_n)}\subseteq M$.
\end{defin}

This ascent eventually must stop. In 1969 Reinhardt, in his thesis, asked whether it is possible to have directly $j:V\prec V$. Kunen, few months later, proved that it is inconsistent. Such result must be approached cautiously: if we work in \ZFC, it is not surprising:

\begin{teo}[Naive Kunen's Theorem](\ZFC)
 There is no $j:V\prec V$.
\end{teo}
\begin{proof}[Proof by Suzuki, \cite{Suzuki}]
 As we are working in \ZFC, it means that $j$ must be a class that is definable with some first order formula. Let $\varphi$ be a formula. Let $a$ be a parameter so that $\varphi(a)$ defines an elementary embedding $j$ from $V$ to itself with least critical point $\kappa$ among all the possible choices of the parameter. Then $\kappa$ is definable in $V$, without parameters. So $j(\kappa)$ satisfies the same definition, and it is the least critical point of embeddings definable with $\varphi$. But $j(\kappa)>\kappa$, contradiction. 
\end{proof}

Kunen's Theorem is stronger than this, as it involves also classes that are not definable: in the original paper, \cite{Kunen}, the proof is in MK, Morse-Kelley theory. It is possible to weaken the theory to NBG, or to \ZFC($j$), i.e., \ZFC{} with $j$ as a predicate that can be inserted in replacement formulas. We leave the interested reader to \cite{HGP}, where all such problems are solved in great detail, and we state Kunen's Theorem in vague terms, with the only provisos that $j$ is not necessarily definable in $V$, replacement for $j$ holds and $\AC$ holds in $V$.

\begin{teo}[Kunen's Theorem, \cite{Kunen}](\AC)
\label{kunens}
 There is no $j:V\prec V$.
\end{teo}

 There are many proofs of this result. Instead of the original proof by Kunen, we show the proof by Woodin (\cite{Kanamori}, Theorem 23.12, second proof), as it gives information that will be useful in Section \ref{LVlambda}. It builds on the following classical result by Solovay:

\begin{teo}[Solovay]
 Let $\kappa>\omega$ be a regular cardinal, and $X\subseteq\kappa$ stationary. Then $X$ can be partitioned in $\kappa$ disjoint stationary sets, i.e., there is $\langle X_\eta:\eta<\kappa\rangle$ such that $X=\bigcup_{\eta<\kappa}X_\eta$ and all the $X_\eta$'s are pairwise disjoint.
\end{teo}

It will be useful this technical remark, that will be used constantly in the paper without comment:
\begin{rem}
 Let $j:M\prec N$, $M\vDash\AC$ or $N\subseteq M$. Then for any $\gamma<\crt(j)$ and any sequence $\langle a_\alpha:\alpha<\gamma\rangle\in M$, $j(\langle a_\alpha:\alpha<\gamma\rangle)=\langle j(a_\alpha):\alpha<\gamma\rangle$. So if $A\in M$ is well-orderable (for example $M\vDash\AC$) and $|A|<\crt(j)$, then $j(A)=\{j(a):a\in A\}$.
\end{rem}
\begin{proof}
 Consider the sequence as a function $f\in M$ with domain $\gamma$. As $\gamma<\crt(j)$, $j(f)(\alpha)=j(f)(j(\alpha))=j(f(\alpha))$.
\end{proof}

\begin{proof}[Proof of \ref{kunens} by Woodin, \cite{Kafkoulis}]
 Let $\lambda$ be the supremum of the critical sequence of $j$. Then $\lambda$ must be a fixed point for $j$, because 
 \begin{equation*}
  j(\lambda)=j(\sup\langle\kappa_n:n\in\omega\rangle)=\sup j(\langle\kappa_n:n\in\omega\rangle)=\sup\langle j(\kappa_n):n\in\omega\rangle=\lambda. 
 \end{equation*}
Let $S^{\lambda^+}_\omega=\{\alpha<\lambda^+:\cof(\alpha)=\omega\}$. Note that $S^{\lambda^+}_\omega$ is stationary: if $C$ is a club in $\lambda^+$, than any $\omega$-limit of elements of $C$ must be in $C$, as $\lambda^+$ is regular, and has cofinality $\omega$\footnote{In fact, $S^{\lambda^+}_\omega$ is $\omega$-stationary, that is, it intersects all the $\omega$-clubs, i.e., the sets that are cofinal in $\lambda^+$ and closed under supremum of $\omega$-sequences}.  We prove that $S^{\lambda^+}_\omega$ cannot be partitioned in $\crt(j)$ disjoint stationary sets, in contradiction with Solovay's result.

Suppose that $\langle S_\alpha:\alpha<\crt(j)\rangle$ is a partition in stationary subsets of $S^{\lambda^+}_\omega$. Then 
\begin{equation*}
 j(\langle S_\alpha:\alpha<\crt(j)\rangle)=\langle T_\beta:\beta<j(\crt(j))\rangle
\end{equation*}
must be a partition in stationary subsets of $S^{\lambda^+}_\omega$ (because $j(\lambda^+)=j(\lambda)^+=\lambda^+)$ of length $j(\crt(j))$. Let $C=\{\alpha<\lambda^+:j(\alpha)=\alpha\}$. Then $C$ is an $\omega$-club in $\lambda^+$: If $\langle\alpha_n:n\in\omega\rangle$ is a sequence of ordinals in $C$, then 
\begin{multline*}
 j(\sup\langle\alpha_n:n\in\omega\rangle)=\sup j(\langle\alpha_n:n\in\omega\rangle)=\\
 =\sup\langle j(\alpha_n):n\in\omega\rangle=\sup\langle\alpha_n:n\in\omega\rangle.
\end{multline*}

 To see that it is unbounded, let $\alpha<\lambda^+$, and consider the sequence $\alpha_0=\alpha$, $\alpha_{n+1}=j(\alpha_n)$. Then $\sup\langle\alpha_n:n\in\omega\rangle\in C$.

As $T_{\crt(j)}$ is stationary, it must intersect the closure of $C$ in some $\eta$. But as $\eta\in T_{\crt(j)}$, it has cofinality $\omega$, so it must be in $C$, therefore $j(\eta)=\eta$. Since $\langle S_\alpha:\alpha<\crt(j)\rangle$ was a partition, there must be some $\alpha$ such that $\eta\in S_\alpha$. But then $\eta=j(\eta)\in j(S_\alpha)=T_{j(\alpha)}$. But $\crt(j)$ cannot be in the image of $j$ (it is not a fixed point, and below it they are all fixed points), so $T_{\crt(j)}\neq T_{j(\alpha)}$, so they should have been disjoint, but they have $\eta$ in common. Contradiction.
\end{proof}

We can now analyze such proof, to understand what it really needs of the hypotheses, with the objective of defining new axioms for which the proof does not work, so that there is a hope that they are not inconsistent. The first remark is that Solovay's Theorem uses \AC{} in an essential way, so the proof would not stand without \AC. 

\begin{defin}(\ZF)[Reinhardt]
 A cardinal $\kappa$ is Reinhardt iff there exists $j:V\prec V$, with $\crt(j)=\kappa$.  
\end{defin} 

It is not immediate to see, but a Reinhardt cardinal is consistently stronger than all the hypotheses we will introduce in this paper. It is still open whether a Reinhardt cardinal would imply a contradiction, and the study of Reinhardt and even larger cardinals is still going on.

Another approach would be to weaken the ground theory, for example not permitting replacement with $j$, so that the critical sequence is not in $V$ and $\lambda$ is not definable. This is the ``Wholeness Axiom'', and it was studied by Paul Corazza (\cite{Corazza}).

A third way would be to consider something smaller than $V$ (but large enough so that the eventual large cardinal would sit at the top of the hierarchy). This will be implemented in the next section.

\section{Rank-into-rank embeddings}
\label{rank}

Analyzing the proof of \ref{kunens}, one realizes that the key object is $\langle S_\alpha:\alpha<\crt(j)\rangle$. Now, if $\beta<\lambda^+$, there is a well-order of $\lambda$ that has length $\beta$, and such well-order can be coded as a subset of $\lambda$. Therefore every ordinal less then $\lambda^+$ can be coded as an element of $V_{\lambda+1}$, and a subset of $\lambda^+$ can be coded as an element of $V_{\lambda+2}$. So the whole proof can be carried on for:

\begin{cor}
 For any $\eta$, there is no $j:V_{\eta+2}\prec V_{\eta+2}$.
\end{cor}
\begin{proof}
 Suppose that there is $j:V_{\eta+2}\prec V_{\eta+2}$. Let $\lambda$ be the supremum of its critical sequence. Then $\lambda\leq\eta$, because $\lambda$ is limit. Then $V_{\lambda+2}\subseteq V_{\eta+2}$, and the proof of \ref{kunens} gives the contradiction.
\end{proof}

Therefore the proof of \ref{kunens} does not exclude the following cases:

\begin{description}
 \item[I3] There exists $j:V_\lambda\prec V_\lambda$;
 \item[I1] There exists $j:V_{\lambda+1}\prec V_{\lambda+1}$.
\end{description}

There is a small ambiguity in the definition of I3. On one hand $\lambda$ must be at least the supremum of the critical sequence (because $\ran(j)\subseteq V_\lambda$), on the other hand $\lambda$ must be less then the supremum of the critical sequence $+2$, because of the Corollary, but, as it is written, it can be the supremum of the critical sequence $+1$. In general we will assume, without loss of generality, that $\lambda$ is the supremum of the critical sequence:

\begin{description}
 \item[I3] There exists $j:V_\lambda\prec V_\lambda$, where $\lambda$ is the supremum of its critical sequence.
\end{description}

Admittedly, not being excluded by a proof is not the most solid of foundations for new axioms. In fact, the ``I'' in their name stands for ``Inconsistency'', as this was the first thought of the community when they were introduced. Yet, they survived all attempts of proving their inconsistency, and they are now in the center of a rich and complex theory, that has many interesting results, some of which will be shown in this paper.

Let $j:V_\lambda\prec V_\lambda$. The mental picture is the following:
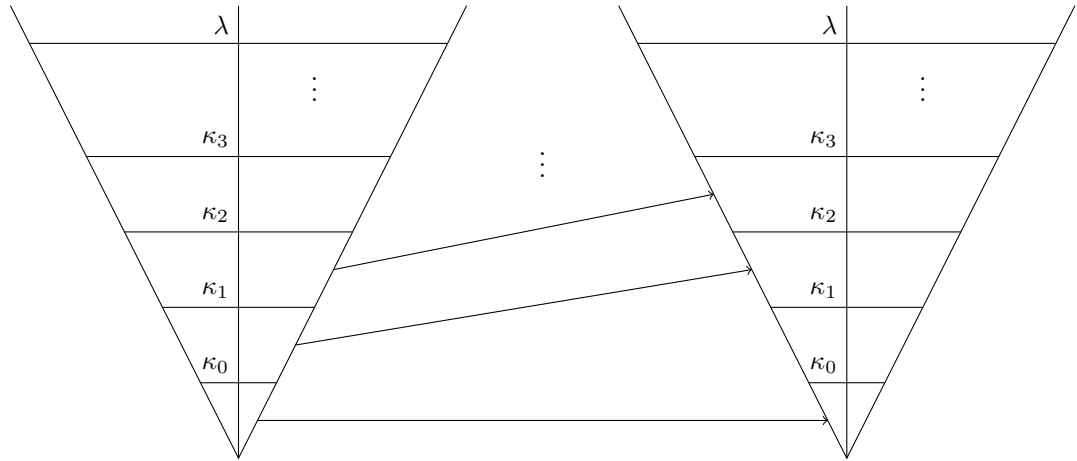
\newline
\newline
\begin{tikzpicture}
 \draw (0,6) -- (3,0) -- (6,6);
 \draw (3,0) -- (3,6);
 \draw (2.5,1) -- (3.5,1);
 \draw (2,2) -- (4,2);
 \draw (1.5,3) -- (4.5,3);
 \draw (1,4) -- (5,4);
 \draw (0.25,5.5) -- (5.75,5.5);
 \draw (3,1) node[anchor=south east]{$\kappa_0$}; 
 \draw (3,2) node[anchor=south east]{$\kappa_1$}; 
 \draw (3,3) node[anchor=south east]{$\kappa_2$};
 \draw (3,4) node[anchor=south east]{$\kappa_3$};  
 \draw (3,5.5) node[anchor=south east]{$\lambda$}; 
 \draw (4,5) node{$\vdots$};

 \draw (8,6) -- (11,0) -- (14,6);
 \draw (11,0) -- (11,6);
 \draw (10.5,1) -- (11.5,1);
 \draw (10,2) -- (12,2);
 \draw (9.5,3) -- (12.5,3);
 \draw (9,4) -- (13,4);
 \draw (8.25,5.5) -- (13.75,5.5);
 \draw (11,1) node[anchor=south east]{$\kappa_0$}; 
 \draw (11,2) node[anchor=south east]{$\kappa_1$}; 
 \draw (11,3) node[anchor=south east]{$\kappa_2$};
 \draw (11,4) node[anchor=south east]{$\kappa_3$};  
 \draw (11,5.5) node[anchor=south east]{$\lambda$}; 
 \draw (12,5) node{$\vdots$};

 \draw [->] (3.25,0.5) -- (10.75,0.5);
 \draw [->] (3.75,1.5) -- (9.75,2.5);
 \draw [->] (4.25,2.5) -- (9.25,3.5);
 \draw (7,4) node{$\vdots$};
\end{tikzpicture}
\newline
\newline
so $V_\lambda$ is divided in two: under a certain rank, all sets are fixed, above a certain rank, all sets are moved. There are $\omega$ ``stripes'', and all the sets on one stripe are moved into the next. There is one stripe, then, that contains no image of $j$. 

The members of the critical sequence are large cardinals:

\begin{rem}
 Let $\langle\kappa_n:n\in\omega\rangle$ be the critical sequence of $j:V_\lambda\prec V_\lambda$. Then all the $\kappa_n$'s are measurable.
\end{rem}
\begin{proof}
 Consider the ultrafilter $U_j=\{X\in{\cal P}(\kappa_0):\kappa_1\in j(X)\}$. As usual, it is a $\kappa_0$-complete measure, so $\kappa_0$ is measurable. Then by elementarity $j(\kappa_0)=\kappa_1$ is measurable, and so on.
\end{proof}

Therefore $\lambda$ is strong limit: if $\alpha<\lambda$, then there is a $\kappa_n$ such that $\alpha<\kappa_n$, and $2^\alpha<\kappa_n<\lambda$. So $|V_\lambda|=\lambda$, and we can code every pair of sets in $V_\lambda$ (in fact every finite sequence of sets) with a single set of $V_\lambda$. Therefore we can code every $\lambda$-sequence of elements in $V_{\lambda+1}$ as a single element in $V_{\lambda+1}$ (coding the sequence $\langle a_\eta:\eta<\lambda\rangle$ with $\{(\eta,x):x\in a_\eta\}$).

The ``stripes'' structure permits to extend $j$ beyond its scope:

\begin{defin}
			Let $j:V_\lambda\prec V_\lambda$. Define $j^+:V_{\lambda+1}\to V_{\lambda+1}$ as
			\begin{equation*}
				\forall A\subset V_{\lambda}\quad j^+(A)=\bigcup_{\beta<\lambda}j(A\cap V_\beta).
			\end{equation*}
		\end{defin}
		
		While it is not clear whether $j^+$ is an elementary embedding, every elementary embedding from $V_{\lambda+1}$ to itself is the `plus' of its restriction to $V_\lambda$:
		
		\begin{lem}
			If $j:V_{\lambda+1}\prec V_{\lambda+1}$, then $(j\upharpoonright V_\lambda)^+=j$. Thus every $j:V_{\lambda+1}\prec V_{\lambda+1}$ is defined by its behaviour on $V_\lambda$, i.e., for every $j,k:V_{\lambda+1}\prec V_{\lambda+1}$,
			\begin{equation*}
				j=k\qquad\text{iff}\qquad j\upharpoonright V_\lambda=k\upharpoonright V_\lambda.
			\end{equation*}
		\end{lem}
		\begin{proof}
			The critical sequence $\langle\kappa_n:n\in\omega\rangle$ is a subset of $V_\lambda$, so it belongs to $V_{\lambda+1}$. But then for every $A\subseteq V_\lambda$, $\{A\cap V_{\kappa_n}:n\in\omega\}\in V_{\lambda+1}$, so
			\begin{equation*}
				\begin{split}
					(j\upharpoonright V_\lambda)^+(A) & = \bigcup_{n\in\omega}(j\upharpoonright V_\lambda)(A\cap V_{\kappa_n})\\
					                                  & = \bigcup_{n\in\omega}j(A\cap V_{\kappa_n})\\
					                                  & = j(\bigcup_{n\in\omega}(A\cap V_{\kappa_n}))\\
					                                  & = j(A).
				\end{split}
			\end{equation*}
		\end{proof}
		
		It is worth noting that there is a strong connection between first-order formulas in $V_{\lambda+1}$ and second-order formulas in $V_\lambda$. In fact, all the elements of $V_{\lambda+1}$ are subsets of $V_\lambda$, so they can be replaced with predicate symbols:
		
		\begin{lem}
		
			Let $A\in V_{\lambda+1}\setminus V_\lambda$ and $\varphi(v_0,v_1,\dots,v_n)$ be a formula. Fix $\hat{A}$ a predicate symbol, and define $\varphi^*(v_1,\dots,v_n)$ in the language of $\LST$ expanded with $\hat{A}$ as following:
			\begin{itemize}
				\item for every occurrence of $v_0$, substitute $\hat{A}$;
				\item for every non-bounded quantified variable $x$, substitute every occurrence of $x$ with $X$, a second-order variable.
			\end{itemize}
			Then for every $a_1,\dots,a_n\in V_\lambda$
			\begin{equation*}
				V_{\lambda+1}\vDash\varphi(A,a_1,\dots,a_n)\quad\text{iff}\quad(V_\lambda,A)\vDash\varphi^*(a_1,\dots,a_n).
			\end{equation*}
		\end{lem}
		\begin{proof}
			The proof is by induction on the complexity of $\varphi$.
		\end{proof}		
		
		The previous Lemma is key in clarifying the relationship between elementary embeddings in $V_\lambda$ and $V_{\lambda+1}$, and to finally prove that $j^+$ is a $\Sigma_0$ elementary embedding from $V_{\lambda+1}$ to itself (or, alternatively that $j$ is a $\Sigma^1_0$ elementary embedding from $V_\lambda$ to itself).
		
		\begin{teo}
		\label{Delta0}
			Let $j:V_\lambda\prec V_\lambda$. Then $j^+:V_{\lambda+1}\to V_{\lambda+1}$ is a $\Delta_0$-elementary embedding.
		\end{teo}
		\begin{proof}
			Let $\hat{A}$ be a symbol for an 1-ary relation, $\hat{a}_1,\dots,\hat{a}_n$ symbols for constants and let $\varphi$ be a formula in $\LST^*$, the language of $\LST$ expanded with $\hat{A}$ and $\hat{a}_1,\dots,\hat{a}_n$. We want to prove that for any $A\subseteq V_\lambda$:
			\begin{equation*}
				(V_\lambda,A,a_1,\dots,a_n)\vDash\varphi\qquad\text{iff}\qquad(V_\lambda,j^+(A),j(a_1),\dots,j(a_n))\vDash\varphi.
			\end{equation*}
			Actually we will prove only one direction for every $\varphi$, the other one following by considering $\neg\varphi$.
			
			First of all, we skolemize $\varphi$, so we find $f_1,\dots,f_m$ functions, $f_i:(V_\lambda)^{m_i}\to V_\lambda$, such that $\varphi$ is equivalent to a formula $\varphi^*$ in $\LST^*$ expanded with $\hat{f}_1,\dots,\hat{f}_m$, where $\varphi^*$ is $\forall x_1\forall x_2\dots\forall x_m\ \varphi'$, with $\varphi'$ a $\Delta_0$ formula.
			
			Let $t_0(x_1,\dots,x_m),\dots,t_p(x_1,\dots,x_m)$ be all the terms that are in $\varphi'$. We can express as a logical formula the phrase $``t_i(x_1,\dots,x_m)$ exists``: we work by induction on the tree of the term, writing a conjuction of formulas in this way
			\begin{itemize}
				\item every time that there is an occurrence of a function, i.e. $\hat{f}_i(t')$, we add to the formula $\exists y\  f_i(t')=y$;
				\item every time that there is an occurrence of a constant, i.e. $\hat{a}_i$, we add to the formula $\hat{a}_i\neq\emptyset$, if $a_i\neq\emptyset$; otherwise we don't write anything.
			\end{itemize}
			
			Then we have that
			\begin{multline*}
				(V_\lambda,A,a_1,\dots,a_n,f_1,\dots,f_m)\vDash\varphi^*\quad\text{iff}\\ 
				\forall\delta<\lambda\ (V_\delta,A\cap V_\delta,a_1',\dots,a_n',f_1\cap V_\delta,\dots,f_m\cap V_\delta)\vDash\\
				\forall x_1,\dots,\forall x_m(\bigwedge_{i<p}t_i(x_1,\dots,x_m)\text{ exists }\rightarrow\varphi'),
			\end{multline*}
			where $a_i'$ is $a_i$ if $a_i\in V_\delta$, otherwise is $\emptyset$. This is true because, with $\delta$ fixed, if only one term does not exist the formula is satisfied, and if all the terms exist, then they are a witness for the satisfaction (or not) of $\varphi'$.
			
			So
			\begin{multline*}
				(V_\lambda,A,a_1,\dots,a_n)\vDash\varphi\\					
				\shoveleft{\quad\leftrightarrow (V_\lambda,A,a_1,\dots,a_n,f_1,\dots,f_m)\vDash\varphi^*}\\
				\shoveleft{\quad\leftrightarrow \forall\delta<\lambda\ (V_\delta,A\cap V_\delta,a_1',\dots,a_n',f_1\cap V_\delta,\dots,f_m\cap V_\delta)\vDash}\\
				\shoveright{\forall x_1,\dots,\forall x_m\ (\bigwedge_{i<p}t_i(x_1,\dots,x_m)\text{ exists }\rightarrow\varphi')}\\
				\shoveleft{\quad\leftrightarrow \forall\delta<\lambda\ (V_\delta,j(A\cap V_\delta),j(a_1'),\dots,j(a_n'),j(f_1\cap V_\delta),\dots,j(f_m\cap V_\delta))\vDash}\\
				\shoveright{\forall x_1,\dots,\forall x_m\ (\bigwedge_{i<p}t_i(x_1,\dots,x_m)\text{ exists }\rightarrow\varphi')}\\
				\shoveleft{\quad\leftarrow (V_\lambda,j^+(A),j(a_1),\dots,j(a_n),j^+(f_1),\dots,j^+(f_m))\vDash\varphi^*}\\
				\shoveleft{\quad\leftrightarrow (V_\lambda,j^+(A),j(a_1),\dots,j(a_n))\vDash\varphi}.\\
			\end{multline*}		
			
			The only thing left to prove is the right direction of the fifth line. The bump is in the fact that if a $j^+(f_i)$ is not total, then it would be possible to have the equation in the fourth line satisfied because the relevant terms do not exist, but the fifth line not satisfied on the elements not in the domain of $j^+(f_i)$. We prove that actually the $j^+(f_i)$ are total, in fact even more:
			
			\begin{cla}
			 If $j:V_\lambda\prec V_\lambda$ and $f:V_\lambda\to V_\lambda$, even partial, then $j^+(\dom(f))=\dom(j^+(f))$.
			\end{cla}
			 \begin{proof}[Proof of claim]
  By definition, 
	\begin{equation*}
	j^+(\dom(f))=\bigcup_{m\in\omega}j(\dom(f)\cap V_{\kappa_m}), 
	\end{equation*}
	while 
	\begin{multline*}
	\dom(j^+(f))=\dom(\bigcup_{m\in\omega}j(f\cap V_{\kappa_m}))=\\
	=\bigcup_{m\in\omega}\dom(j(f\cap V_{\kappa_m}))=\bigcup_{m\in\omega}j(\dom(f\cap V_{\kappa_m})).
  \end{multline*}
		
	It is not immediate to see why the two are the same, because it is possible that $\dom(f)\cap V_{\kappa_m}$ and $\dom(f\cap V_{\kappa_m})$ are different, as on the left we are cutting the dominion, on the right both the dominion and codominion. The solution is to divide everything in even smaller pieces, that would be the same on the left and on the right side.
	
	\begin{equation*}
	 \dom(f)\cap V_{\kappa_m}=\bigcup_{n\in\omega}\dom(f\cap V_{\kappa_n})\cap V_{\kappa_m}, 
	\end{equation*}
	and 
	\begin{equation*}
	 \dom(f\cap V_{\kappa_m})=\bigcup_{n\in\omega}\dom(f\cap V_{\kappa_m})\cap V_{\kappa_n}. 
	\end{equation*}
	Therefore:
	
	\begin{multline*}
	j^+(\dom(f))=\bigcup_{m\in\omega}j(\dom(f)\cap V_{\kappa_m})=\bigcup_{m\in\omega}j(\bigcup_{n\in\omega}\dom(f\cap V_{\kappa_n})\cap V_{\kappa_m})=\\
	=\bigcup_{m\in\omega}\bigcup_{n\in\omega}j(\dom(f\cap V_{\kappa_n})\cap V_{\kappa_m})=\bigcup_{n\in\omega}\bigcup_{m\in\omega}j(\dom(f\cap V_{\kappa_n})\cap V_{\kappa_m})=\\
	=\bigcup_{n\in\omega}j(\dom(f\cap V_{\kappa_n}))=\dom(j^+(f)).
	\end{multline*}
 \end{proof}

The third and fifth equivalence are a delicate point. In the third equivalence, the $\omega$-sequence $\langle\dom(f\cap V_{\kappa_n})\cap V_{\kappa_m}:n\in\omega\rangle\in V_\lambda$, even if it is defined with $\kappa_n$, therefore $j$ behaves properly. In the fifth equivalence, what seems an infinite union is in fact finite.

		\end{proof}
		
		Therefore every $j:V_\lambda\prec V_\lambda$ can be extended to a unique $j^+:V_{\lambda+1}\to V_{\lambda+1}$ that is at least a $\Delta_0$-elementary embedding. Moreover it is possible to prove
		\begin{itemize}
			\item I3 holds iff for some $\lambda$ there exists a $j:V_{\lambda+1}\to V_{\lambda+1}$ that is a $\Delta_0$-elementary embedding;
			\item I1 holds iff for some $\lambda$ there exists a $j:V_{\lambda+1}\prec V_{\lambda+1}$ that is a $\Sigma_n$-elementary embedding for every $n$
		\end{itemize}
		
		Theorem \ref{Delta0} implies that if $j,k:V_\lambda\prec V_\lambda$, then $j^+(k):V_\lambda\prec V_\lambda$. This operation between elementary embeddings is called \emph{application}, and we write $j^+(k)=j\cdot k$. Coupled with composition (not to be mistaken to!) they create an interesting algebra.
		
		\begin{teo}[Laver, 1992 \cite{Laver1}]
			Fix $\lambda$ and let ${\cal E}_\lambda=\{j:V_\lambda\prec V_\lambda\}$. For all $j\in{\cal E}_\lambda$ the closure of $\{j\}$ in $({\cal E}_\lambda,\cdot)$ is the free algebra generated by the law
					\begin{equation*}
						\shoveleft{\text{(Left Distributive Law)}\quad i\cdot(j\cdot k)=(i\cdot j)\cdot(i\cdot k)}.
					\end{equation*}
		\end{teo}
		
		Moreover Laver proved that the free algebra generated by the laws above satisfies the word problem. His proof used extensively I3, but later Dehornoy (\cite{Dehornoy}) managed to prove the same thing in \ZFC. A proof of this can be found in the Handbook of Set Theory \cite{Handbook}, Chapter 11. 
		
		There are still open problems at the I3 level:
		
		\begin{Q}
		 Are there $j,k:V_\lambda\prec V_\lambda$ such that the algebra generated by them is free?
		\end{Q}
		
		Another interesting result in I3 regards a function on the integers. Consider $j:V_\lambda\prec V_\lambda$ with critical point $\kappa$ and let ${\cal A}_j$ be the closure of $j$ in $({\cal E}_j,\cdot)$. Then define
		\begin{equation*}
			f(n)=|\{\crt(k):k\in{\cal A}_j,\ \kappa_n<\crt(k)<\kappa_{n+1}\}|.
		\end{equation*}  
		Then $f(0)=f(1)=0$ and $f(2)=1$, because the simplest element of ${\cal A}_j$ that has a critical point not in the critical sequence of $j$ is $((j\cdot j)\cdot j)\cdot(j\cdot j)$. However $f(3)$ is very large. Laver (\cite{Laver2}) proved that for any $n$, $f(n)$ is finite, but $f$ dominates the Ackermann function, so $f$ cannot be primitive recursive (\cite{Dougherty}).

 A standard way of using application is to consider $j$, $j(j)$, $j(j(j))$\dots, and we write $j^0=j$ and $j^{n+1}=j(j^n)$. If $\langle\kappa_n:n\in\omega\rangle$ is the critical sequence of $j$, then by elementarity, $\crt(j(j))=j(\crt(j))=\kappa_1$ and $\langle\kappa_n:n\geq 1,n\in\omega\rangle$ is the critical sequence of $j^2$. By induction, $\crt(j^n)=\kappa_n$. This is the reason why, when talking about I3 or I1, the critical point is never mentioned, when usually it is the main object of research (see strongness, supercompactness\dots): if there exists a $\lambda$ that witnesses I3, then there are infinite critical points tied to the same $\lambda$, so it is really $\lambda$ the relevant cardinal, while the critical points are interchangeable and unbounded in $\lambda$. Still, they carry a lot of strength. The following result shows that I3 really sits on top of the large cardinal hierarchy:

\begin{prop}
 \label{verylarge}
 Let $\langle\kappa_i:i<\omega\rangle$ be the critical sequence of $j$. Then for every $i$, $\kappa_i$ is $\gamma$-supercompact for any $\gamma<\lambda$ and $n$-huge for every $n$.
\end{prop}

We use the characterization via ultrafilter:

\begin{prop}[Solovay, Reinhardt]
  A cardinal $\kappa$ is $\gamma$-supercompact iff there exists a fine, normal ultrafilter on ${\cal P}_\kappa(\lambda)$. It is $n$-huge iff there is a $\kappa$-complete normal ultrafilter $U$ over some ${\cal P}(\eta)$ and $\kappa=\eta_0<\eta_1<\dots<\eta_n=\eta$ such that for any $i<n$ $\{x\in{\cal P}(\eta):\ot(x\cap\eta_{i+1})=\eta_i\}\in U$. 
\end{prop} 

\begin{proof}[Proof of \ref{verylarge}]
 Let $\kappa_0\leq\gamma<\lambda$, and suppose $\gamma<\kappa_n$; then it is routine to check that $\gamma$-supercompactness of $\kappa_0$ is witnessed by
  \begin{equation*}
   U=\{X\subseteq{\cal P}_{\kappa_0}(\gamma):(j^n)''\gamma\in j^n(X)\}
  \end{equation*}
and that $n$-hugeness of $\kappa_0$ is witnessed by
  \begin{equation*}
   U=\{X\subseteq{\cal P}(\kappa_n):j''\kappa_n\in j(X)\}\quad\text{and}\quad\kappa_0,\dots,\kappa_n.
  \end{equation*}
Then $\kappa_n$ is $\gamma$-supercompact and $n$-huge by elementarity.
\end{proof}

On the other hand $\lambda$ (the ``true'' rank-into-rank cardinal) does not enjoy many large cardinal properties, as it is of cofinality $\omega$, and usually large cardinals are regular. The most one can say is that it is a Rowbottom cardinal. 

We end the section with some interesting reflection properties of $V_\lambda$:
		
		\begin{lem}
		\label{reflection}
		 Suppose $j:V_\lambda\prec V_\lambda$ and let $\langle \kappa_n: n\in\omega\rangle$ be its critical sequence. Then for every $n\in\omega$, $V_{\kappa_n}\prec V_\lambda$.
		\end{lem}
		\begin{proof}
		 Since $j$ is the identity on $V_{\kappa_0}$, it is easy to see that $V_{\kappa_0}\prec V_{\kappa_1}$. Considering $j(j)$, since $\crt(j(j))=\kappa_1$ and 
		 \begin{equation*}
		  j(j)(\kappa_1)=j(j)(j(\kappa_0))=j(j(\kappa_0))=\kappa_2, 
		 \end{equation*}
		we have also that $V_{\kappa_1}\prec V_{\kappa_2}$. We can generalize this to prove that for every $n\in\omega$, $V_{\kappa_n}\prec V_{\kappa_{n+1}}$. But then $\langle (V_{\kappa_n},\id_{V_{\kappa_n}}),n\in\omega\rangle$ forms a direct system, whose direct limit is $V_\lambda$. 
		\end{proof}

  In particular, as $\kappa_0$ is inaccessible, $V_\lambda\vDash\ZFC$.

\section{$L(V_{\lambda+1})$ and I0}
\label{LVlambda}

It would seem that one could not go any further than I1, as already $j:V_{\lambda+2}\prec V_{\lambda+2}$ is inconsistent. But there is a way to enlarge the domain of $j$ without stumbling in Kunen's inconsistency:

\begin{defin}
 Let $A$ be a set. Then the class of sets constructible relative to $A$ is:
 \begin{itemize}
  \item $L_0(A)=\text{ the transitive closure of }A$;
	\item $L_{\alpha+1}(A)=\Def(L_\alpha(A))$;
	\item $L_\gamma(A)=\bigcup_{\alpha<\gamma}L_\alpha(A)$;
	\item $L(A)=\bigcup_{\alpha\in\Ord}L_\alpha(A)$.
 \end{itemize}
\end{defin}

This is the smallest \ZF-model that contains $A$. We can now introduce I0:

  \begin{defin}[Woodin]		
    \begin{description}
      \item[I0] There exists $j:L(V_{\lambda+1})\prec L(V_{\lambda+1})$ with critical point less then $\lambda$.
    \end{description}
  \end{defin}

 As $L(V_{\lambda+1}))\vDash V=L(V_{\lambda+1})$, it must be that $j(V_{\lambda+1})=V_{\lambda+1}$. Thanks to the added assumption that the critical point must be less than $\lambda$, then, I0 is stronger than I1, as $j\upharpoonright V_{\lambda+1}:V_{\lambda+1}\prec V_{\lambda+1}$ (otherwise it could have been the identity). This implies that $\lambda$ must be the supremum of the critical sequence. If one wants to avoid the menace of Kunen's inconsistency, it must be that $\langle S_\alpha:\alpha<\crt(j)\rangle$ is not in $L(V_{\lambda+1})$. This is constructed via the Axiom of Choice on $V_{\lambda+1}$. So, for I0 to hold, it must be that $L(V_{\lambda+1})\nvDash\AC$, and in particular $V_{\lambda+1}$ must not be well-orderable. This is the first affinity with $\AD^{L(\mathbb{R})}$, as in that case we have $L(\mathbb{R})\nvDash\AC$. More affinities will be introduced in section \ref{simil}.

At a superficial glance I0 seems a Reinhardt cardinal, as it is an embedding in a ``universe'' of \ZF. But the situation is completely different: our ground model is $V$, we assume that $V\vDash\ZFC$, and we construct $L(V_{\lambda+1})$ inside $V$. So $j$ is a definable class of $V$, and we will see that $L(V_{\lambda+1})$ does not satisfy replacement for $j$, so such a $j$ is not a witness for Reinhardt, as in that case the theory of the ground model must be richer.

Before delving in I0, we analyze the structure of $L(V_{\lambda+1})$ itself, without further assumptions, that still carry similarities with $L(\mathbb{R})$.

As $L(V_{\lambda+1})$ is $L$ relativized to $V_{\lambda+1}$, we can relativize in the same way $OD$ and $HOD$:

\begin{defin}
 $OD_{V_{\lambda+1}}$, the class of the sets that are ordinal-definable over $V_{\lambda+1}$, is the G\"{o}del closure of $\{V_\alpha:\alpha\in\Ord\}\cup V_{\lambda+1}$.

 $HOD_{V_{\lambda+1}}$, the class of the sets that are hereditarily ordinal-definable over $V_{\lambda+1}$, is the class of sets whose transitive closure is contained in $OD_{V_{\lambda+1}}$.
\end{defin}

  \begin{lem}
			There exists a definable surjection $\Phi:Ord\times V_{\lambda+1}\twoheadrightarrow L(V_{\lambda+1})$, therefore $L(V_{\lambda+1})\vDash V=\HOD_{V_{\lambda+1}}$
		\end{lem}
		\begin{proof}
			This is immediate from the theory of relative constructibility. 
		\end{proof}
		
		The first application of this Lemma comes in form of partial Skolem functions. Since $L(V_{\lambda+1})\nvDash\AC$, we possibly could not have Skolem functions. But with the previous Lemma we can define for every formula $\varphi(x,x_1,\dots,x_n)$, $a\in V_{\lambda+1}$, $a_1,\dots,a_n\in L(V_{\lambda+1})$:
		\begin{multline*}
			h_{\varphi,a}(a_1,\dots,a_n)= y\quad\text{where }y\text{ is the minimum in }(\OD_a)^{L(V_{\lambda+1})}\\
			\text{ such that }L(V_{\lambda+1})\vDash\varphi(y,a_1,\dots,a_n).					
		\end{multline*}
		These are partial Skolem functions, and the Skolem Hull of a set is its closure under such Skolem functions.
		
		We indicate with $\twoheadrightarrow$ the surjectivity of a function.
		
		\begin{defin}
			\begin{equation*}
				\Theta_{V_{\lambda+1}}^{L(V_{\lambda+1})}=\sup\{\gamma:\exists f:V_{\lambda+1}\twoheadrightarrow\gamma, f\in L(V_{\lambda+1})\}.
			\end{equation*}
		\end{defin}
		
		In the following, we will call $\Theta_{V_{\lambda+1}}^{L(V_{\lambda+1})}=\Theta$, because it is a lighter notation and there is no possibility of misinterpretation. 
		
		The role of $\Theta$ in $L(V_{\lambda+1})$ is exactly the same of its correspondent in $L(\mathbb{R})$. In the usual setting, under \AC, to measure the largeness of a set we fix a bijection from this set to a cardinal or, equivalently, the ordertype of a well-ordering of the set. Since there is no Axiom of Choice in $L(V_{\lambda+1})$, it is not always possible to define cardinality for sets that are not in $V_{\lambda+1}$ in the usual way, so to quantify the ``largeness" of a subset of $V_{\lambda+1}$ we will not use bijections, but surjections, or, equivalently, not well-orders, but prewellorderings (pwo for short).

		An order is a pwo if it satisfies antireflexivity, transitivity, and every subset has a least element; in other words, it is a well-order without the antisymmetric property. It is easy to see that the counterimage of a surjective function is a pwo. One can think of a pwo as an order whose equivalence classes are well-ordered, or a well-ordered partition. This creates a strong connection between subsets of $V_{\lambda+1}$ and ordinals in $\Theta$:		
		\begin{lem}
		\label{Thetasubsets}
			\begin{enumerate}
				\item For every $\alpha<\Theta$, there exists in $L(V_{\lambda+1})$ a pwo in $V_{\lambda+1}$ with ordertype $\alpha$, that is codeable as a subset of $V_{\lambda+1}$;
				\item for every $Z\subseteq V_{\lambda+1}$, $Z\in L(V_{\lambda+1})$ there exists $\alpha<\Theta$ such that $Z\in L_\alpha(V_{\lambda+1})$.
			\end{enumerate}
		\end{lem}
		\begin{proof}
			\begin{enumerate}
				\item Let $\rho:V_{\lambda+1}\twoheadrightarrow\alpha$. Then 
					\begin{equation*}
						R_\alpha=\{(a,b)\in V_{\lambda+1}\times V_{\lambda+1}:\rho(a)\leq\rho(b)\} 
					\end{equation*}
					is a pwo in $V_{\lambda+1}$. Moreover, $V_{\lambda+1}\times V_{\lambda+1}$ can be codified as a subset of $V_{\lambda+1}$, so also $R_\alpha$ can.
				\item Let $\gamma$ be such that $Z\in L_\gamma(V_{\lambda+1})$ and consider $H^{L_\gamma(V_{\lambda+1})}(V_{\lambda+1},Z)$ the Skolem Hull in $L_\gamma(V_{\lambda+1})$ of $V_{\lambda+1}$ and $Z$. Then, since $H^{L_\gamma(V_{\lambda+1})}(V_{\lambda+1},Z)\cong L_\gamma(V_{\lambda+1})$, by condensation its collapse ${\cal X}=L_\alpha(V_{\lambda+1})$ for some $\alpha$. But $H^{L_\gamma(V_{\lambda+1})}(V_{\lambda+1},Z)$ is the closure under the Skolem functions, and since there is a surjection from $V_{\lambda+1}$ to the Skolem functions, this surjection transfers to $H^{L_\gamma(V_{\lambda+1})}(V_{\lambda+1},Z)$ and to $L_\alpha(V_{\lambda+1})$, so $\alpha<\Theta$. Since $Z$ and all its elements are in $H^{L_\gamma(V_{\lambda+1})}(V_{\lambda+1},Z)$, $Z$ is not collapsed and then $Z\in L_\alpha(V_{\lambda+1})$.
			\end{enumerate}
		\end{proof}
		
		We can also have a ``local'' version of $\Theta$:
		
		\begin{defin}
			\begin{multline*}
				\Theta_{V_{\lambda+1}}^{L_\alpha(V_{\lambda+1})}=\sup\{\gamma:\exists f:V_{\lambda+1}\twoheadrightarrow\gamma, f\in L(V_{\lambda+1}),\\
				\{(a,b):f(a)<f(b)\}\in L_\alpha(V_{\lambda+1})\}.
			\end{multline*}
		\end{defin}
		
		The reason why we do not just relativize $\Theta$ to $L_\alpha(V_{\lambda+1})$ is because there can be pwos in $L(V_{\lambda+1})$ that are not in $L_\alpha(V_{\lambda+1})$, but can be coded in $L_\alpha(V_{\lambda+1})$, so in this way $\Theta^{L_\alpha(V_{\lambda+1})}$ is more ``truthful'' to indicate the ordinals that are ``hidden'' inside $V_{\lambda+1}$.
		
		The notion of ``stable'' also neatly generalizes. Let $\delta$ be the minimum ordinal such that $L_\delta(V_{\lambda+1})\prec_1 L_\Theta(V_{\lambda+1})$. Then $L_\delta(V_{\lambda+1})$ is the collapse of the closure of $V_\lambda+1$ under the Skolem functions $h_{\varphi,a}$ such that $\varphi$ is $\Sigma_1$. The map that associates to any $(\varphi,a,x)$ the collapse of $h_{\varphi,a}(x)$ is therefore a $\Sigma_1$ partial surjection from $V_{\lambda+1}$ to $L_\delta(V_{\lambda+1})$. A study of the fine structure of $L_\delta(V_{\lambda+1})$ has the following results:
		
		\begin{lem}
		\label{finestructure}
		 \begin{itemize}
		  \item For any $\beta<\delta$, there exists a total surjection from $V_{\lambda+1}$ to $L_\beta(V_{\lambda+1})$ that is in $L_\delta(V_{\lambda+1})$. Therefore $\Theta^{L_\delta(V_{\lambda+1})}=\delta$.
			\item $\delta$ is the supremum of the $\mathbf{\Delta}^2_1$ pwos of $V_{\lambda+1}$ and $(\mathbf{\Delta}^2_1)^{L(V_{\lambda+1})}=L_\delta(V_{\lambda+1})\cap V_{\lambda+2}$.
			\item Let $\rho:V_{\lambda+1}\to L_\delta(V_{\lambda+1})$ be a partial $\Sigma_1$ map. Then for any $Z\subseteq\dom(\rho)$, $Z\in L_\delta(V_{\lambda+1})$, $\rho\upharpoonright Z$ is bounded.
		 \end{itemize}
		\end{lem}
		
		The proofs are the same as in \cite{Steel}.
		
		\begin{defin}
			\begin{equation*}
				\DC_\lambda:\quad \forall X\ \forall F:(X)^{<\lambda}\to{\cal P}(X)\setminus\emptyset\ \exists g:\lambda\to X\ \forall\gamma<\lambda\  g(\gamma)\in F(g\upharpoonright\gamma).
			\end{equation*}
		\end{defin}
		
		Note that this is a generalization of \DC, since $\DC=\DC_\omega$: we use directly a function on $<\lambda$-sequences instead of considering a binary relation because binary relations cannot handle the limit stages. 
			
		\begin{lem}
		\label{lambdastructure}
			In $L(V_{\lambda+1})$ the following hold:
			\begin{enumerate}
				\item $\Theta$ is regular (\cite{Kafkoulis});
				\item $\DC_\lambda$.
			\end{enumerate}
		\end{lem}
		\begin{proof}
			\begin{enumerate}
				\item We fix a definable surjection $\Phi:Ord\times V_{\lambda+1}\twoheadrightarrow L(V_{\lambda+1})$. For every $\xi<\Theta$ there is a surjection $h:V_{\lambda+1}\twoheadrightarrow\xi$. First of all, we choose one surjection for each $\xi$: we define $t:\Theta\setminus\emptyset\to Ord$, where for every $\xi<\Theta$, $t(\xi)$ is the least $\gamma$ such that there exists $x\in V_{\lambda+1}$ such that $\Phi(\gamma,x)$ is a surjection from $V_{\lambda+1}$ to $\xi$. Then we define
				\begin{equation*}
					h_\xi(\langle x,y\rangle)=\begin{cases}
						\Phi(t(\xi),x)(y) & \text{ if }\Phi(t(\xi),x)\text{ is a map in }\xi;\\
						\emptyset         & \text{ else}.
						\end{cases}
				\end{equation*}
				We have that $h_\xi:V_{\lambda+1}\to \xi$ is well defined, because $\langle x,y\rangle$ is codeable in $V_{\lambda+1}$, and it is indeed a surjection: by definition there exists $x\in V_{\lambda+1}$ such that $\Phi(t(\xi),x)$ is a surjection from $V_{\lambda+1}$ to $\xi$, so for every $\beta<\xi$ there exists $y\in V_{\lambda+1}$ such that $\Phi(t(\xi),x)(y)=h_\xi(\langle x,y\rangle)=\beta$. Moreover $h_\xi$ is definable in $L(V_{\lambda+1})$.
				
				Now, suppose that $\Theta$ is not regular, i.e., there exists $\pi:\alpha\to\Theta$ cofinal in $\Theta$ with $\alpha<\Theta$. Then we claim that $H(\langle x,y\rangle)=h_{\pi\circ h_\alpha(x)}(y)$ is a surjection from $V_{\lambda+1}$ to $\Theta$: let $\beta<\Theta$; then there exists $\gamma<\alpha$ such that $\pi(\gamma)>\beta$ and there must exist $x\in V_{\lambda+1}$ such that $h_\alpha(x)=\gamma$; so $\pi(h_\alpha(x))>\beta$, and there exists $y\in V_{\lambda+1}$ such that $H(\langle x,y\rangle)=h_{\pi\circ h_\alpha(x)}(y)=\beta$. Contradiction.
				
				\item We have to prove that
					\begin{equation*}
						\forall X\ \forall F:(X)^{<\lambda}\to{\cal P}(X)\setminus\emptyset\ \exists g:\lambda\to X\ \forall\gamma<\lambda\  g(\gamma)\in F(g\upharpoonright\gamma).
					\end{equation*}
				  The proof is through several steps. First of all, $\DC_\lambda(V_{\lambda+1})$, that is $\DC_\lambda$ only for $X=V_{\lambda+1}$, is quite obvious, because for every $F$ as above by \AC{} there exists a $g$ as above in $V$, but since $g$ is a $\lambda$-sequence of elements in $V_{\lambda+1}$ we have that $g$ is codeable in $V_{\lambda+1}$, so $g\in L(V_{\lambda+1})$. 
				
				Then we prove $\DC_\lambda(\alpha\times V_{\lambda+1})$ for every ordinal $\alpha$. The idea is roughly to divide $F$ in two parts, and to define $g$ using the minimum operator for the ordinal part, and $\DC_\lambda(V_{\lambda+1})$ for the other part. For every $s\in(\alpha\times V_{\lambda+1})^{<\lambda}$, we define $m(s)$ as the minimum $\gamma$ such that there exists $x\in V_{\lambda+1}$ such that $(\gamma,x)\in F(s)$. We call $\pi_2:(\alpha\times V_{\lambda+1})^{<\lambda}\to(V_{\lambda+1})^{<\lambda}$ the projection. For every $t\in (V_{\lambda+1})^{<\lambda}$, say $t=\langle x_\xi:\xi<\nu\rangle$, we define $c(t)$ by induction as a sequence in $(\alpha\times V_{\lambda+1})^{<\lambda}$ such that $\pi_2(c(t))\subseteq t$: $c(t)=\langle(\gamma_\xi,x_\xi):\xi<\bar{\nu}\rangle$, where 
				\begin{equation*}
				 \gamma_\xi=\min\{\gamma:(\gamma,x_\xi)\in F(c(t\upharpoonright\xi))\}, 
				\end{equation*}
				 so that $\bar{\nu}$ is the smallest one such that there is no $\gamma$ such that $(\gamma,x_{\bar{\nu}})\in F(c(t)\upharpoonright\bar{\nu})$. Let $G:(V_{\lambda+1})^{<\lambda}\to{\cal P}(V_{\lambda+1})\setminus\emptyset$ defined as 
				\begin{equation*}
					G(t)=\{x\in V_{\lambda+1}:(m(c(t)),x)\in F(c(t))\}, 
				\end{equation*}
				then by $\DC_\lambda(V_{\lambda+1})$ there exists $g:\lambda\to V_{\lambda+1}$ such that for every $\beta<\lambda$ $g(\beta)\in G(g\upharpoonright\beta)$. Now let $f:\lambda\to(\alpha\times V_{\lambda+1})$ be defined by induction, $f(\beta)=(m(f\upharpoonright\beta),g(\beta))$. We want to prove that $f(\beta)\in F(f\upharpoonright\beta)$ for every $\beta<\lambda.$
				
				We prove by induction that $f\upharpoonright\beta=c(g\upharpoonright\beta)$. Suppose that for every $\xi<\beta$, $f\upharpoonright\xi=c(g\upharpoonright\xi)$. By definition $c(g\upharpoonright\beta)=\langle(\gamma_\xi,g(\xi)):\xi<\bar{\beta}\rangle$, with 
				\begin{equation*}
					\gamma_\xi=\min\{\gamma:(\gamma,g(\xi))\in F(c(g\upharpoonright\xi))\}, 
				\end{equation*}
				so $(\gamma_\xi,g(\xi))\in F(c(g\upharpoonright\xi))$. But $g(\xi)\in G(g\upharpoonright\xi)$, so by definition of $G$, 
				\begin{equation*}
				    (m(c(g\upharpoonright\xi)),\ g(\xi))\in F(c(g\upharpoonright\xi)),
				\end{equation*}
				therefore $\gamma_\xi=m(c(g\upharpoonright\xi))$ and
				\begin{equation*}
					f(\xi)=(m(f\upharpoonright\xi),g(\xi))=(m(c(g\upharpoonright\xi)),g(\xi))=(\gamma_\xi,g(\xi))=c(g\upharpoonright\beta)(\xi).
				\end{equation*}
				So $f\upharpoonright\beta=c(g\upharpoonright\beta)$ and, since for every $\xi$, $(\gamma_\xi,g(\xi))\in F(c(g\upharpoonright\xi))$, $f(\beta)\in F(f\upharpoonright\beta)$.
				
				Finally, let $X\in L(V_{\lambda+1})$. Let $\alpha$ be such that $\Phi''(\alpha\times V_{\lambda+1})\supseteq X$, and let $F:(X)^{<\lambda}\to{\cal P}(X)\setminus\emptyset$. For every $t=\langle(\gamma_\xi,x_\xi):\xi<\nu\rangle\in (\alpha\times V_{\lambda+1})^{<\lambda}$, we call $c(t)=\langle\Phi(\gamma_\xi,x_\xi):\xi<\bar{\nu}\rangle\in X^{<\lambda}$, where $\bar{\nu}$ is the minimum such that $\Phi(\gamma_{\bar{\nu}},x_{\bar{\nu}})\notin X$. Then we define $G:(\alpha\times V_{\lambda+1})^{<\lambda}\to{\cal P}(\alpha\times V_{\lambda+1})\setminus\emptyset$, 
				\begin{equation*}
					G(t)=\{(\gamma,x):\Phi(\gamma,x)\in F(c(t))\}, 
				\end{equation*}
				and by $\DC_\lambda(\alpha\times V_{\lambda+1})$ we find $g:\lambda\to(\alpha\times V_{\lambda+1})$ such that $g(\beta)\in G(g\upharpoonright\beta)$ for every $\beta<\lambda$. Then $f=\Phi\circ g$ is as we wanted, because for every $\beta<\lambda$, $\Phi(g(\beta))\in F(c(g\upharpoonright\beta))$, and as above we can prove that $c(g\upharpoonright\beta)=f\upharpoonright\beta$. 
				
				So we have $\DC_\lambda(X)$ for every $X\in L(V_{\lambda+1})$, that is exactly $\DC_\lambda$.
								
				\end{enumerate}
		\end{proof}

Suppose now for the rest of the section that $j:L(V_{\lambda+1})\prec L(V_{\lambda+1})$, with $\crt(j)<\lambda$. Define
		\begin{equation*}
			U_j=\{X\subseteq V_{\lambda+1}:X\in L(V_{\lambda+1}), j\upharpoonright V_\lambda\in j(X)\}. 
		\end{equation*}
		Then $U_j$ is a normal non-principal $L(V_{\lambda+1})$-ultrafilter in $V_{\lambda+1}$, and we can construct the ultraproduct $Ult(L(V_{\lambda+1}),U_j)$. Note that for every $f,g:V_{\lambda+1}\to L(V_{\lambda+1})$ 
		\begin{equation*}
			\begin{split}
				[f]=[g]\ & \text{iff }\{x\in V_{\lambda+1}:f(x)=g(x)\}\in U_j\\ 
				         & \text{iff }j\upharpoonright V_\lambda\in j(\{x\in V_{\lambda+1}:f(x)=g(x)\})=\\
				         & \qquad=\{x\in V_{\lambda+1}:j(f)(x)=j(g)(x)\}\\
				         & \text{iff }j(f)(j\upharpoonright V_\lambda)=j(g)(j\upharpoonright V_\lambda), 
			\end{split}
		\end{equation*}
		and in the same way $[f]\in[g]$ iff $j(f)(j\upharpoonright V_\lambda)\in j(g)(j\upharpoonright V_\lambda)$, so 
		\begin{equation*}
			Ult(L(V_{\lambda+1}),U_j)\cong\{j(f)(j\upharpoonright V_\lambda):\dom f=V_{\lambda+1}\}. 
		\end{equation*}
		Let $i:L(V_{\lambda+1})\to Ult(L(V_{\lambda+1}),U_j)$ be the natural embedding of the ultraproduct, then for every $a\in L(V_{\lambda+1})$, $i(a)=[c_a]$ corresponds in the equivalence to $j(c_a)(j\upharpoonright V_\lambda)=j(a)$, so we can suppose $i=j$. Is $j$ an elementary embedding from $L(V_{\lambda+1})$ to $Ult(L(V_{\lambda+1}),U_j)$? Since we do not have \AC, the answer is not immediate because we possibly do not have \L os' Theorem.
		
		We will prove \L os' Theorem for this case, and this will imply that $j$ is an elementary embedding. It is clear that the only real obstacle is to prove that for every formula $\varphi$ and $f_1,\dots,f_n\in L(V_{\lambda+1})$ such that $\dom f_i=V_{\lambda+1}$
		\begin{multline*}
			Ult(L(V_{\lambda+1}),U_j)\vDash\exists x\ \varphi([f_1],\dots,[f_n])\\
			\text{ iff }\{x\in V_{\lambda+1}:L(V_{\lambda+1})\vDash\exists y\ \varphi(f_1(x),\dots,f_n(x))\}\in U_j.
		\end{multline*}
		The direction from left to right is immediate: if $[g]$ witness the left side, then $g(x)$ witness the right side. For the opposite direction, we need a sort of $U_j$-choice, i.e., we need to find a function $g$ such that 
		\begin{equation*}
			\{x\in V_{\lambda+1}:L(V_{\lambda+1})\vDash\varphi(g(x),f_1(x),\dots,f_n(x))\}\in U_j. 
		\end{equation*}
		We re-formulate this considering $f:V_{\lambda+1}\to L(V_{\lambda+1})\setminus\emptyset$, 
		\begin{equation*}
			f(x)=\{y\in L(V_{\lambda+1}):L(V_{\lambda+1})\vDash\varphi(y,f_1(x),\dots,f_n(x))\}.
		\end{equation*}
		
		\begin{lem}[\L os' Theorem for I0, Woodin, \cite{Kafkoulis}]
			Let $j:L(V_{\lambda+1})\prec L(V_{\lambda+1})$ and $U_j$ as above. Then for every $F:V_{\lambda+1}\to L(V_{\lambda+1})\setminus\{\emptyset\}$ there exists $H:V_{\lambda+1}\to L(V_{\lambda+1})\setminus\{\emptyset\}$ such that $\{x\in V_{\lambda+1}:H(x)\in F(x)\}\in U_j$.
		\end{lem}
		\begin{proof}
			First we consider the case $\forall a\in V_{\lambda}$, $F(a)\subseteq V_{\lambda+1}$. We have to define $H$ such that $j(H)(j\upharpoonright V_\lambda)\in j(F)(j\upharpoonright V_\lambda)$. Fix a $b\in j(F)(j\upharpoonright V_\lambda)$, and define
			\begin{equation*}
				H(k)=\begin{cases}
					c & \text{if $k:V_\lambda\prec V_\lambda$, $k(k)=j$ and $k(c)=b$}\\
					\emptyset & \text{otherwise}
					\end{cases}
			\end{equation*}
			
			Note that 
			\begin{equation*}
				K_b:=\{k\in V_{\lambda+1}\ |\ k:V_\lambda\prec V_\lambda, k(k)=j,\exists c\  k(c)=b\}\in U_j, 
			\end{equation*}
			because 
			\begin{equation*}
				j(K_b)=\{k\in V_{\lambda+1}\ |\ k:V_\lambda\prec V_\lambda, k(k)=j(j),\exists c\  k(c)=j(b)\} 
			\end{equation*}
			and $j\upharpoonright V_\lambda\in j(K_b)$ (with $c=b$), so $\{x\in V_{\lambda+1}:H(x)\neq\emptyset\}\in U_j$. Then 
			\begin{equation*}
				j(H)(k)=\begin{cases}
					c & \text{if $k:V_\lambda\prec V_\lambda$, $k(k)=j(j)$ and $k(c)=j(b)$}\\
					0 & \text{otherwise}
					\end{cases}
			\end{equation*}
			so $j(H)(j\upharpoonright V_\lambda)=b\in j(F)(j\upharpoonright V_\lambda)$.
			
			For the more general case $\forall a\in V_{\lambda}$ $F(a)\subseteq L(V_{\lambda+1})$ fix $\Phi:Ord\times V_{\lambda+1}\twoheadrightarrow L(V_{\lambda+1})$ definable and define 
			\begin{equation*}
				\hat{F}(a)=\{x\in V_{\lambda+1}:\exists\gamma\ \Phi(\gamma,x)\in F(a)\}. 
			\end{equation*}
			Then there exists $\hat{H}$ such that $\{a\in V_{\lambda+1}:\hat{H}(a)\in\hat{F}(a)\}\in U_j$. Let $\gamma_a=\min\{\gamma:\Phi(\gamma,\hat{H}(a))\in F(a)\}$. Therefore $H(a)=\Phi(\gamma_a,\hat{H}(a))$ is as desired.
		\end{proof}
		
		Therefore, calling 
		\begin{equation*}
		 {\cal Z}=\{j(f)(j\upharpoonright V_\lambda): f\in L(V_{\lambda+1}),\dom(F)=V_{\lambda+1}\},
		\end{equation*}
		$j:L(V_{\lambda+1})\to{\cal Z}$ is an elementary embedding, and ${\cal Z}\cong L(V_{\lambda+1})$. Let $k_U$ be the inverse of the collapse of ${\cal Z}$. We've seen in the proof of the previous Lemma that for every $b\in V_{\lambda+1}$ there exists $h$ such that $j(h)(j\upharpoonright V_\lambda)$, so $V_{\lambda+1}\subseteq{\cal Z}$ and $k_U:L(V_{\lambda+1})\prec{\cal Z}$. Moreover, if $R$ is a pwo in $V_{\lambda+1}$, then $R=\{a\in V_{\lambda+1}:j(a)\in j(R)\}$, and since $j(a),j(R)\in{\cal Z}$ and $V_{\lambda+1}\subseteq{\cal Z}$ we have that $R$ is not collapsed, so $\Theta\subseteq{\cal Z}$ and $\crt(k_U)>\Theta$.
		
		\begin{teo}[Woodin, \cite{Kafkoulis}]
		\label{allproper}
			For every $j:L(V_{\lambda+1})\prec L(V_{\lambda+1})$ there exist a $L(V_{\lambda+1})$-ultrafilter $U$ in $V_{\lambda+1}$ and $j_U,k_U:L(V_{\lambda+1})\prec L(V_{\lambda+1})$ such that $j_U$ is the elementary embedding from $U$, $j=j_U\circ k_U$ and $j\upharpoonright L_\Theta(V_{\lambda+1})=j_U\upharpoonright L_\Theta(V_{\lambda+1})$.
		\end{teo}
		
		\begin{defin}
			Let $j:L(V_{\lambda+1})\prec L(V_{\lambda+1})$. We say that $j$ is \emph{weakly proper} if $j=j_U$.
		\end{defin}
		
		Note that $U\notin L(V_{\lambda+1})$, otherwise we could run the naive Kunen's Theorem argument, so $L(V_{\lambda+1})$ does not satisfy replacement for $j$, and this proves that I0 is not a ``$L(V_{\lambda+1})$-Reinhardt''. In fact, one cannot even have $j\upharpoonright L_\Theta(V_{\lambda+1})\in L(V_{\lambda+1})$, otherwise one could construct $U$ from that, so $j$ is not amenable. 
		
		In a certain sense weakly proper embeddings are the ``right'' embeddings for I0, and the most natural extension of I1, if one considers that $L_\Theta(V_{\lambda+1})$ is, loosely speaking, the $V_{\lambda+2}$ of $L(V_{\lambda+1})$, i.e., the set of subsets of $V_{\lambda+1}$. Every embedding has therefore a core part, that comes from an ultraproduct, it is unique and depends solely on its behaviour on $L_\Theta(V_{\lambda+1})$, and a spurious part (note that $k_U$ is not an I0-embedding, as it has critical point $>\lambda$) that moves things up above without really changing the important part of the embedding. For example, if there are indiscernibles, it could be a simple shift of indiscernibles, so nothing that carries real information. 
		
		In fact, the behaviour of an embedding that is weakly proper depends only on its $V_\lambda$ part:
		
		\begin{lem}[Woodin, \cite{Kafkoulis}]
		\label{onlyfromVlambda}
			For every $j_1,j_2:L(V_{\lambda+1})\prec L(V_{\lambda+1})$, if $j_1\upharpoonright V_\lambda=j_2\upharpoonright V_\lambda$ then $j_1\upharpoonright L_\Theta(V_{\lambda+1})=j_2\upharpoonright L_\Theta(V_{\lambda+1})$.
		\end{lem}
		\begin{proof}
			We can suppose that $j_1$ and $j_2$ are weakly proper. By the usual analysis of the ultraproduct, we have that every strong limit cardinal with cofinality bigger than $\Theta$ is a fixed point for $j_1$ and $j_2$, so $I=\{\eta:j_1(\eta)=j_2(\eta)=\eta\}$ is a proper class. Let $M=H^{L(V_{\lambda+1})}(I\cup V_{\lambda+1})$. Since $V_{\lambda+1}\subseteq M$ $k^*$, the inverse of the transitive collapse of $M$, has domain $L(V_{\lambda+1})$. If $k^*$ is not the identity then, as $\Theta\subseteq M$ by Lemma \ref{Thetasubsets}, $\crt(k^*)>\Theta$. But in that case $\crt(k^*)$ is a strong limit cardinal with cofinality bigger than $\Theta$, so $\crt(k^*)\in I$, and this is a contradiction, because $I\subseteq\ran(k^*)$ and $\crt(k^*)\notin\ran(k^*)$. So $k^*$ is the identity and $L(V_{\lambda+1})=H^{L(V_{\lambda+1})}(I\cup V_{\lambda+1})$.
			
			Therefore every element of $L_\Theta(V_{\lambda+1})$ is definable with parameters in $I\cup V_{\lambda+1}$. Let $A\in L(V_{\lambda+1})\cap V_{\lambda+2}$, $A=\{x\in V_{\lambda+1}:L(V_{\lambda+1})\vDash\varphi(\eta,a)\}$ with $\eta\in I$ and $a\in V_{\lambda+1}$. Since $j_1\upharpoonright V_\lambda=j_2\upharpoonright V_\lambda$, we have that $j_1(a)=j_2(a)$. Then 
			\begin{multline*}
			 j_1(A)=\{x\in V_{\lambda+1}:L(V_{\lambda+1})\vDash\varphi(\eta,j_1(a))\}=\\
			 =\{x\in V_{\lambda+1}:L(V_{\lambda+1})\vDash\varphi(\eta,j_2(a))\}=j_2(A). 
			\end{multline*}
			
			But every element of $L_\Theta(V_{\lambda+1})$ is definable from an ordinal $\alpha<\Theta$ and an element of $V_{\lambda+1}$, $\alpha$ is definable from some pwo in $L(V_{\lambda+1})\cap V_{\lambda+2}$, so $j_1\upharpoonright L_\Theta(V_{\lambda+1})=j_2\upharpoonright L_\Theta(V_{\lambda+1})$.
		\end{proof}
		
		\begin{cor}[Woodin, \cite{Kafkoulis}]
		\label{FixedpointsbelowTheta}
		 Let $j:L(V_{\lambda+1})\prec L(V_{\lambda+1})$ with $\crt(j)$. Then $C=\{\alpha<\Theta:j(\alpha)=\alpha\}$ is unbounded in $\Theta$.
		\end{cor}
		\begin{proof}
		 We can suppose that $j$ is weakly proper. Then every ordinal $\beta<\Theta$ is definable from some elements of $I\cup V_{\lambda+1}$, say $I_\beta=\{i_0,\dots,i_n\}\in I$. 
		
		 Let $\alpha_0<\Theta$ and consider the transitive collapse of the Skolem closure $Z$ of $\bigcup_{\beta<\alpha_0}I_\beta\cup V_{\lambda+1}$. Then by condensation it must be some $L_\gamma(V_{\lambda+1})$. As $\alpha_0<\Theta$, there exists $\pi:V_{\lambda+1}\twoheadrightarrow Z$, therefore $\gamma<\Theta$. As $\alpha_0\subseteq Z$, we have $\alpha_0<\gamma$. As $j(Z)=Z$, we have $j(\gamma)=\gamma$. Therefore for any $\alpha_0<\Theta$ there exists a $\gamma\in C$ such that $\alpha_0<\gamma<\Theta$.
		\end{proof}

\section{Strong implications and inverse limits}
\label{strimp}

When introducing new axioms, it is good practice to prove that they really are adding something, i.e., that their consistency strength is strictly larger (or smaller) than the consistency of the already known ones. 

\begin{defin}
 Let $A(x)$ and $B(x)$ two large cardinal properties. We say that $A(x)$ \emph{strongly implies} $B(x)$ iff for any $\kappa$, $A(\kappa)$ implies $B(\kappa)$ and the smallest $\kappa$ such that $A(\kappa)$ holds is bigger than the smallest $\eta$ such that $B(\eta)$ holds.
\end{defin}

This is even stronger then just the strict implications, cfr the note after Lemma \ref{change}.

\begin{prop}
 If $j:V_\lambda\prec V_\lambda$ with $\crt(j)=\kappa$, then there is a normal ultrafilter $U$ on $\kappa$ such that $\{\alpha<\kappa:\alpha$ is $n$-huge for every $n\}\in U$. Therefore I3 strongly implies being $n$-huge for every $n$.
\end{prop}  
\begin{proof}
 By \ref{verylarge} $\kappa$ is $n$-huge for every $n\in\omega$. Let $U=\{X\subseteq\kappa:\kappa\in j(X)\}$ and let 
 \begin{equation*}
  X_0=\{\alpha<\kappa:\alpha\text{ is $n$-huge for every }n\}. 
 \end{equation*}

 Then 
 \begin{equation*}
  j(X_0)=\{\alpha<j(\kappa):\alpha\text{ is $n$-huge for every }n\}, 
 \end{equation*}
 
 and $\kappa\in j(X_0)$, so $X_0\in U$.
\end{proof}

Note that ``$\omega$-huge'', i.e, $j:V\prec M$ with $M^\lambda\subseteq M$\footnote{This is not a standard definition. For different authors, ``$\omega$-huge'' can mean different things.}, is inconsistent by Kunen's original proof of Kunen's Theorem, so this is the best we can do.

\begin{defin}
 Let $j:M_0\prec M_1$. Then we say that $(M_\alpha,j^\alpha)$ is the $\alpha$-th iterate if:
 \begin{itemize}
  \item $j^0=j$;
  \item $j^{\alpha+1}=j^\alpha(j^\alpha)$;
	\item $M_{\alpha+1}=j^\alpha(M_\alpha)$ is well-founded;
	\item $j_{\gamma,\alpha+1}=j^{\alpha}\circ j^{\alpha-1}\circ\dots\circ j^\gamma$. So $j_{\alpha,\alpha+1}=j^\alpha$;
	\item if $\alpha$ is limit, then $(M_\alpha,j^\alpha,j_{n,\alpha})$ is the direct limit of the system $(M_\beta,j_{\beta,\gamma})$ with $\beta,\gamma<\alpha$, if it is well-founded.
 \end{itemize} 

 We say that $j$ is $\alpha$-iterable if $(M_\alpha,j^\alpha)$ exists. We say that $j$ is iterable if it is $\alpha$-iterable for any $\alpha$ ordinal.
\end{defin}

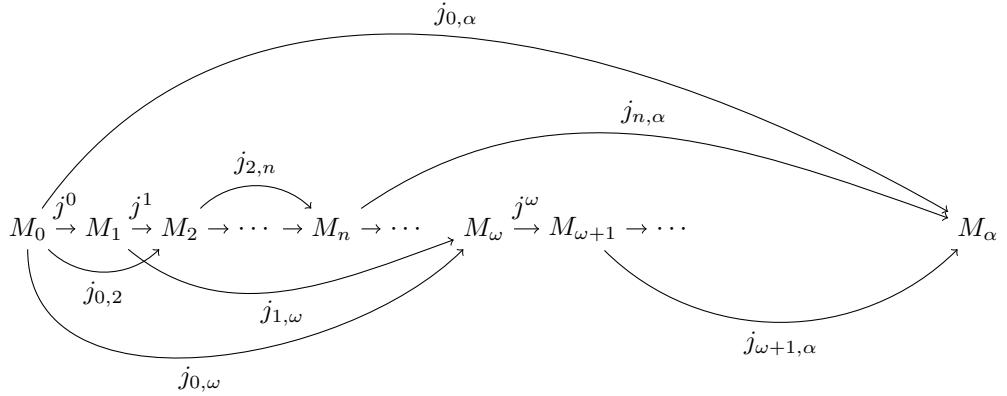
\begin{tikzpicture}
 \node (M0) {$M_0$};;
 \node (M1)[right of = M0] {$M_1$}
     edge [<-] node[above]{$j^0$}(M0);
 \node (M2)[right of = M1] {$M_2$}
     edge [<-] node[above]{$j^1$}(M1)
		 edge [<-,out=-135,in=-45] node[below]{$j_{0,2}$}(M0);
 \node (dots1)[right of = M2] {$\cdots$}
     edge [<-] (M2);
 \node (Mn)[right of =dots1] {$M_n$}
     edge [<-] (dots1)
		 edge [<-,out=135,in=45] node[above]{$j_{2,n}$}(M2);
 \node (dots2)[right of = Mn] {$\cdots$}
     edge [<-] (Mn);
 \node (Momega)[right of = dots2] {$M_\omega$}
     edge [<-,out=-135,in=-90] node[below]{$j_{0,\omega}$}(M0) 
	   edge [<-,out=-160,in=-40] node[below]{$j_{1,\omega}$}(M1);
 \node (Momega1)[right of = Momega,xshift=3mm] {$M_{\omega+1}$}
     edge [<-] node[above]{$j^\omega$}(Momega);
 \node (dots3)[right of = Momega1,xshift=2mm] {$\cdots$}
     edge [<-] (Momega1);
 \node (Malpha)[right of = dots3,xshift=3cm] {$M_\alpha$}
     edge [<-,out=-135,in=-45] node[below]{$j_{\omega+1,\alpha}$}(Momega1)
		 edge [<-,out=150,in=55] node[above]{$j_{0,\alpha}$}(M0)
		 edge [<-,out=160,in=35] node[above]{$j_{n,\alpha}$}(Mn);
\end{tikzpicture}
\newline
\newline
It is possible for the iterate to not exists, whether because $j^\alpha(j^\alpha)$ or $j^\alpha(M_\alpha)$ are not defined, or because the model $M_\alpha$ is not well-founded. The direct limit of the system $(M_\beta,j_{\beta,\gamma})$ is defined as an equivalence relation on the set of $(\beta,a)$ such that $a\in M_\beta$: if $\beta<\gamma$ then $(\beta,a)$ is equivalent to $(\gamma,b)$ iff $j_{\beta,\gamma}(a)=b$, and $[(\beta,a)]$ is in $[(\gamma,b)]$ iff $j_{\beta,\gamma}(a)\in b$. If $\gamma<\beta$, viceversa. Then $j_{\beta,\alpha}(a)$ is defined as $[(\beta,a)]$ and $j^\alpha=j_{0,\alpha}(j)$. Ambiguously, we call $M_\alpha$ both the domain of the model defined via the equivalence class and the $\in$-model equivalent to it. So if $M_1\subseteq M_0$, then for all $\alpha$ such that $M_\alpha$ is defined we have $M_\alpha\subseteq M_0$, but it is possible that $M_0\subseteq M_\alpha$ if seen as $\in$-models. It will be clear from context which of the two will be used. Note that if $a\in M_\alpha$, with $\alpha$ limit, then it must be in the range of some $j_{\beta,\alpha}$, with $\beta<\alpha$. 

So let $j:V_\lambda\prec V_\lambda$. We have already defined $j^n$, but 
\begin{equation*}
 j^n(j^n)=j(j^{n-1})(j(j^{n-1}))=j(j^{n-1}(j^{n-1}))=j(j^n)=j^{n+1}, 
\end{equation*}
 by left distributive law and induction, therefore the $j^n$'s are iterates of $j$. Moreover, $j(V_\lambda)=V_\lambda$, so the first $\omega$-iterates of $j$ are $(V_\lambda,j^n)$ and $j$ is finitely iterable. Could we push it any further? As all the $M_\alpha\subseteq V_\lambda$, technically there is no problem in defining it, the only decisive factor being whether $M_\alpha$ is well-founded.

\begin{prop}
\label{iterable}
 ``There exists $j:V_\lambda\prec V_\lambda$ $\omega$-iterable'' strongly implies I3($\lambda)$.
\end{prop}
\begin{proof}
  Suppose $j:V_\lambda\prec V_\lambda$ is $\omega$-iterable, let $\langle\kappa_n:n\in\omega\rangle$ be its critical sequence. Then $M_\omega$ is the set of equivalence classes of $(n,a)$ such that $a\in V_\lambda$, where if $n<m$ then $(n,a)$ is equivalent to $(m,b)$ iff $j_{n,m}(a)=b$. Also, $j_{n,\omega}(a)$ is the class of $[(n,a)]$. 
	
	We prove by induction that $j_{n,\omega}(\alpha)=\alpha$ for $\alpha\in\kappa_n$. Note that if $\beta<\kappa_n$, then for every $m>n$, $j_{n,m}(\beta)=\beta$, therefore $[(n,\beta)]=[(m,\beta)]$. So for any $m\in\omega,\beta<\lambda$, $[(m,\beta)]<[(n,\alpha)]$ iff $[(n,\beta)]<[(n,\alpha)]$: If $m\leq n$, then $j_{m,n}(\beta)<\alpha<\kappa_n=j_{m,n}(\kappa_m)$ therefore by elementarity $\beta<\kappa_m$ and $[(m,\beta)]=[(n,\beta)]$. If $m>n$, then $\beta<j_{n,m}(\alpha)=\alpha$, therefore also $[(n,\beta)]<[(n,\alpha)]$. So $j_{n,\omega}(\alpha)=[(n,\alpha)]=\{[(n,\beta)]:\beta<\alpha\}$, that is by induction $\alpha$.  
	
	We want to calculate now $j_{0,\omega}(\kappa_0)$. If $\alpha<j_{0,\omega}(\kappa_n)$, then there exists $n\in\omega$ such that $\alpha=j_{n,\omega}(\beta)$ for some $\beta<\lambda$. But then $j_{n,\omega}(\beta)<j_{0,\omega}(\kappa_0)=j_{n,\omega}(j_{0,n}(\kappa_0))=\kappa_n$, therefore $\alpha=j_{n,\omega}(\beta)=\beta<\lambda$. On the other hand, if $\alpha<\lambda$ then there exists $n\in\omega$ such that $\alpha<\kappa_n$, so $j_{n,\omega}(\alpha)=\alpha<j_{n,\omega}(\kappa_n)=j_{0,\omega}(\kappa_0)$. Therefore $j_{0,\omega}(\kappa_0)=\lambda$.
	
If $a\in V_\lambda$, then there exists $n\in\omega$ such that $a\in V_{\kappa_n}$. So $j_{n,\omega}(a)=a\in M_\omega$. Therefore $V_\lambda\subseteq M_\omega$. Let 
\begin{multline*}
 T=\{\langle e_0,\dots,e_n\rangle:\exists\langle\alpha_i:i\leq n+1\rangle,\forall i\leq n+1\ \alpha_i<\lambda,\\
\forall i\leq n\exists e_i:V_{\alpha_i}\prec V_{\alpha_{i+1}},\forall i<n\ e_{i+1}\upharpoonright V_{\alpha_i}=e_i\}, 
\end{multline*}
the tree of approximations of an I3 elementary embedding. Note that $T=T^{M_\omega}$. As $j$ is a branch of $T$, $T$ is ill-founded, and by absoluteness of well-foundedness it is ill-founded also in $M_\omega$, with the difference that a branch cannot go up to $\lambda$, as $\lambda$ is regular in $M_\omega$ by elementarity. So $M_\omega\vDash\exists\gamma<\lambda\ \exists k:V_\gamma\prec V_\gamma$. So by elementarity $\exists\gamma<\kappa_0\ \exists k:V_\gamma\prec V_\gamma$.
\end{proof}

\begin{cor}
 Let $\alpha$ be a countable ordinal. Then ``There exists $j:V_\lambda\prec V_\lambda$ $\omega\cdot(\alpha+1)$-iterable'' strongly implies ``There exists $j:V_\lambda\prec V_\lambda$ $\omega\cdot\alpha$-iterable''.
\end{cor}
\begin{proof}
  By induction, $V_\lambda\subseteq M_{\omega\cdot\alpha}\subseteq M_{\omega\cdot(\alpha+1)}$. We build again a tree of approximations, but this time we want to add rank functions, so that the $\omega\cdot\alpha$-th iterate of $\bigcup_{n\in\omega}e_n$ will be well-founded. So a branch of $T$ is $(\langle e_0,\dots,e_n\rangle,\langle f_0,\dots,f_n\rangle)$, where $f_i$ is a rank function on $\{(\beta,a):\beta<\omega\cdot\alpha,a\in V_{\alpha_i}\}$ that is coherent with the equivalence on $M_{\omega\cdot\alpha}$. A branch of $T$ will give an embedding that is $\omega\cdot\alpha$-iterable, and the proof is as before.
\end{proof}

\begin{prop}
 \label{iterableI3}
 If $j:V_\lambda\prec V_\lambda$ is $\omega_1$-iterable, then it is iterable.
\end{prop}
\begin{proof}
 This is a standard result in the theory of the iterations of elementary embeddings. Suppose that $j$ is not iterable, and let $\nu$ be the least such that $M_\nu$ is ill-founded. Let $N$ be a countable substructure of $V_\lambda$, let $h:N\prec V_\lambda$ indicate the embedding resulting form the collapse. Let $h(\bar{M}_0)=V_\lambda$, $h(\bar{\jmath})=j$, and build the iterate sequence in $N$. Then $h(\bar{M}_\beta)=M_{h(\beta)}$ and let $\bar{\nu}=h^{-1}(\nu)$. In particular $\bar{\nu}$ is countable. For any $\beta\leq\nu$ define $h'_\beta$ so that the diagram commutes:
\newline
\newline
\begin{tikzpicture}
 \node (M0) {$V_\lambda$};
 \node (dots1)[right of = M0] {$\cdots$};
 \node (Malpha)[right of = dots1] {$M_\alpha$};
 \node (Mbeta)[right of = Malpha,xshift=5mm] {$M_\beta$}
  edge[<-] node[above] {$j_{\alpha,\beta}$} (Malpha);  
 \node (Mbarnu)[right of = Mbeta,xshift=1cm] {$M_{\bar{\nu}}$}
  edge[<-] node[above] {$j_{\beta,\bar{\nu}}$} (Mbeta);
 \node (dots2)[right of = Mbarnu,xshift=1cm] {$\cdots$};
 \node (Mnu)[right of = dots2,xshift=1cm] {$M_\nu$};
 \node (barM0)[below of = M0, yshift=-1cm] {$\bar{M}_0$}
  edge[->] node[right]{$h'_0$} node[left]{$h$}(M0);
 \node (dots3)[right of = barM0] {$\cdots$};
 \node (barMalpha)[below of = Malpha, yshift=-1cm] {$\bar{M}_\alpha$}
  edge[->] node[right]{$h'_\alpha$}(Malpha);
 \node (barMbeta)[below of = Mbeta, yshift=-1cm] {$\bar{M}_\beta$}
  edge[<-] node[above] {$\bar{\jmath}_{\alpha,\beta}$} (barMalpha)
 	edge[->] node[right] {$h'_\beta$} (Mbeta);
 \node (barMbarnu)[below of = Mbarnu, yshift=-1cm] {$\bar{M}_{\bar{\nu}}$}
  edge[<-] node[above] {$\bar{\jmath}_{\beta,\bar{\nu}}$} (barMbeta)
	edge[->] node[right] {$h'_{\bar{\nu}}$} (Mbarnu)
	edge[->] node[below] {$h$} (Mnu);
\end{tikzpicture}
\newline
\newline
Let $h'_0=h\upharpoonright\bar{M}_0$. If $\beta$ is a limit ordinal, then every $x\in\bar{M}_\beta$ is some $\bar{j}_{\alpha,\beta}(y)$, with $\alpha<\beta$ and $y\in\bar{M}_\alpha$. Then $h'_\beta(x)=j_{\alpha,\beta}(h'_\alpha(y))$. If $\beta=\gamma+1$, then by elementarity, as $j:V_\lambda\prec V_\lambda$, $\bar{M}_\beta=\bar{M}_\gamma$ and $M_\beta=M_\gamma$, so $h'_\beta=h'_\gamma$. 

But then $h'_{\bar{\nu}}:\bar{M}_{\bar{\nu}}\prec M_{\bar{\nu}}$. As $M_\nu$ is illfounded, then by elementarity also $M_{\bar{\nu}}$ is illfounded, and since $\nu$ was the least and $\bar{\nu}$ is countable, $\nu=\bar{\nu}$ is countable and $j$ is not $\omega_1$-iterable.
\end{proof} 

While the ``$\omega$-hugeness'' is inconsistent, it is not immediately clear what happens with ``$\omega$-superstrongness'':

  \begin{defin}		
    \begin{description}
      \item[I2] There exists $j:V\prec M$, with $M\subseteq V$, such that $V_\lambda\subseteq M$ for some $\lambda=j(\lambda)>\crt(j)$.
    \end{description}
  \end{defin}
	
	\begin{prop}
	  Suppose $j$ witnesses I2. Then $\lambda$ is the supremum of its critical sequence, and $j\upharpoonright V_\lambda:V_\lambda\prec V_\lambda$.
	\end{prop}
	\begin{proof}
	 Let $\eta$ be the supremum of the critical sequence of $j$. Then it must be $\eta\leq\lambda$, as $\eta$ is the first fixed point of $j$ above $\crt(j)$. If $\eta+2\leq\lambda$ then there is a contradiction with Kunen's Theorem, so we only have to check that $\eta+1\neq\lambda$, i.e., not all subsets of $V_\eta$ are in $M$.
	
	Suppose that $j''\eta\in M$, then define $U=\{X\subseteq{\cal P}(\eta):j''\eta\in j(X)\}$ and let $i:V\prec N$ be the ultrapower via $U$. Then $i(\eta)=\eta$, and
	 \begin{equation*}
	  i''{\cal P}(\eta)=\{\bigcup_{\alpha<\eta}i(X\cap\alpha):X\in{\cal P}(\eta)\}\in N. 
	 \end{equation*}
	
	 Now, coding $\lambda^+$ as sets in ${\cal P}(\lambda)$, $i''\lambda^+\in N$, and Woodin's proof of Kunen's theorem goes smoothly (because $C$ is definable from $i''\lambda^+$). Contradiction.
	
	So $j''\eta\notin M$, and $\eta$ must be $\lambda$. In particular $j\upharpoonright V_\lambda:V_\lambda\prec V_\lambda$.
	\end{proof}
	
 The previous Proposition says that I2 implies I3. We can do more:

\begin{prop}
 If $j:V\prec M$ witnesses I2$(\lambda)$ then $j\upharpoonright V_\lambda$ is iterable.
\end{prop}
\begin{proof}
 Consider an iteration of $j:V\prec M$ with $V_\lambda\subseteq M$. Suppose that $(V_\lambda)^{M_\nu}$ is ill-founded, with $\nu$ the least. Let $\alpha_0$ be the minimum such that $j_{0,\nu}(\alpha_0)$ is ill-founded. In other words, $V_\lambda\vDash``\alpha_0$ is the least $\alpha$ such that $\exists\mu$ such that $j_{0,\mu}(\alpha)$ is ill-founded''. Let $\beta<j_{0,\nu}(\alpha_0)$ ill-founded. Let $\eta<\nu$ and $\alpha_1$ such that $j_{\eta,\nu}(\alpha_1)=\beta$. Consider $j_{0,\eta}(\alpha_0)$. As $j_{\eta,\nu}(\alpha_1)<j_{0,\nu}(\alpha_0)$, we must have $\alpha_1<j_{0,\eta}(\alpha_0)$. But by elementarity 
\begin{multline*}
 (V_\lambda)^{M_\eta}\vDash ``j_{0,\eta}(\alpha_0)\text{ is the least }\alpha\text{ such that }\\
 \exists\mu\text{ such that }[j_{0,\eta}(j)]_{0,\mu}(\alpha)\text{ is ill-founded}''. 
\end{multline*}

Note that $[j_{0,\eta}(j)]_{0,\mu}=[(j_{\eta,\eta+1})]_{0,\mu}=j_{\eta,\eta+\mu}$. But then $\alpha_1$ is such that, contradiction because $j_{0,\eta}(\alpha_0)$ should have been the minimum.
\end{proof}

Another way to extend I3 comes from the way we characterized I3 and I1, respectively as $\Sigma_0$- and $\Sigma_\omega$-elementary embeddings from $V_{\lambda+1}$ to itself. We can consider all the intermediate steps:

\begin{defin}[$\Sigma^1_n$ Elementary Embedding]
		 Let $j:V_\lambda\prec V_\lambda$. Then $j$ is $\Sigma^1_n$ iff $j^+$ is a $\Sigma_n$-elementary embedding, i.e., iff for every $\Sigma^1_n$-formula $\varphi(X)$, for every $A\subseteq V_\lambda$, 
		  \begin{equation*}
		   V_\lambda\vDash\varphi(A)\leftrightarrow\varphi(j^+(A)).
		  \end{equation*}
		\end{defin}
		
		It is not immediate, however, to see if these hypothesis are really different, and what is their relationship with I2. In fact, they are not always different: the following theorems will prove that if $n$ is odd, then $j$ is $\Sigma^1_n$ iff it is $\Sigma^1_{n+1}$. During the proof it will be clear that I2 is the same as $j$ is $\Sigma^1_1$, therefore proving that I1 implies I2.
		
		To prove this we need what is often called the `descriptive set theory' on $V_\lambda$. The fact that $\lambda$ has cofinality $\omega$ and is a strong limit, makes possible to develop a description of the closed subsets of $V_\lambda$ as projection of trees similar to the classic one in descriptive set theory. The fact that both $\Sigma^1_1$ and $\Sigma^1_2$ sets can be represented with trees will give the first equivalence, in a way similar to Schoenfield Absoluteness Lemma. Fix $\lambda$ and a critical sequence $\langle\kappa_0,\kappa_1,\dots\rangle$.
		
		Let $\varphi(X,Y)$ a $\Sigma^1_0$-formula, namely 
		  \begin{equation*}
		   \forall x_0\exists y_0\forall x_1\exists y_1\dots\cdot_n\psi(X,Y,x_0,y_0,\dots).
		  \end{equation*}
		We define the tree $T_{\varphi(X,Y)}$. The $m$-th level of $T_{\varphi(X,Y)}$ is the set of $(a,b,F,P)$ such that:
		\begin{itemize}
			\item $a,b\subseteq V_{\kappa_m}$;
			\item $F$ is a partial function $:(V_{\kappa_m})^{\leq n+1}\to V_{\kappa_m}$ such that for all $d_0,\dots,d_n$ where $F$ is defined then 
			  \begin{equation*}
			   \psi(a,b,d_0,F(d_0),d_1,F(d_0,d_1),\dots);
			  \end{equation*}
			\item $P:((V_{\kappa_m})^{\leq n+1}\setminus\dom(F))\to(\omega\setminus(m+1))$.
		\end{itemize} 
		We say that $(a,b,F,P)<(a',b',F',P')$, with the first term in the $m$-th level and the second in the $m'$-th level, when $a\subseteq a'$, $b\subseteq b'$, $a'\cap V_{\kappa_m}=a$, $b'\cap V_{\kappa_m}=b$, $F\subseteq F'$, and if $P(\vec{d})<m$ then $\vec{d}\in\dom(F)$, otherwise $P'(\vec{d})=P(\vec{d})$.
		
		Informally, $a$ and $b$ are just the attempts to construct $X$ and $Y$, $F$ is an approximation of the Skolem function that would witness the first order part of $\varphi$, and $P$ is booking the level in which the elements of $V_{\kappa_m}$ that are not yet in the dominion of $F$ \emph{will} be in its extension.
		
		It is clear that if $T_{\varphi(X,Y)}$ has an infinite branch, the union of the $a$'s and $b$'s will give suitable $X$ and $Y$, and the $F$'s will construct a Skolem function, that is total thanks to the $P$'s.
		
		\begin{lem}
			Let $\varphi(X,Y)$ a $\Sigma^1_0$-formula and $B\subseteq V_\lambda$. Then $V_\lambda\vDash\exists X\ \varphi(X,B)$ iff $T_{\varphi(X,B)}$ has an infinite branch.
		\end{lem}	
		
		The $\Sigma^1_2$ case is a bit more complex. Fix a $B\subseteq V_\lambda$. For every $a\subseteq V_{\kappa_m}$ let 
		  \begin{equation*}
		   G_m(a)=\{(c,F,P):(a,c,F,P)\in m\text{-th level of }T_{\neg\varphi(X,B,Y)}\}. 
		  \end{equation*}
		Then the $m$-th level of ${\cal T}_{\varphi(X,B)}$ is 
		\begin{equation*}
		 \{(a,H):a\subseteq V_{\kappa_m},H:G_m(a)\to\lambda^+\}. 
		\end{equation*}
		We say that $(a,H)<(a',H')$ when $a'\cap V_{\kappa_m}$ and if $(a,c,F,P)<(a',c',F',P')$, then $H(c,F,P)>H'(c',F',P')$.
		
		\begin{lem}
		\label{SatisfactionTree1}
		  Let $\varphi(X,Y)$ a $\Sigma^1_0$-formula and $B\subseteq V_\lambda$. Then $V_\lambda\vDash\exists X\forall Y\ \varphi(X,B)$ iff ${\cal T}_{\varphi(X,B)}$ has an infinite branch.
		\end{lem} 
			
		If ${\cal T}_{\varphi(X,B)}$ has an infinite branch, with $A$ the union of the $a$'s in the branch, then the $H$'s assure that there are no possible infinite branches in $T_{\neg\varphi(A,B,Y)}$ for every $Y\subset V_\lambda$, because otherwise it would be possible to build a descending chain in $\lambda^+$.
		
		\begin{lem}
			Let $j:V_\lambda\prec V_\lambda$. Then $j$ is $\Sigma^1_1$ iff $j^+$ preserves the well-founded relations.
		\end{lem}	
		\begin{proof}
			`Being well-founded' is a $\Delta^1_1$ relation, so if $j$ is $\Sigma^1_1$ it preserves well-founded relations. Vice versa, let $\varphi(X,Y)$ a $\Sigma^1_0$ formula, $B\subseteq V_\lambda$ and suppose that $V_\lambda\vDash\exists X\varphi(X,B)$. Fix an $X\subseteq V_\lambda$ such that $V_\lambda\vDash\varphi(X,B)$. Therefore by Theorem \ref{Delta0} $V_\lambda\vDash\varphi(j^+(X),j^+(B))$ and $V_\lambda\vDash\exists X\ \varphi(X,j^+(B))$. If $V_\lambda\vDash\forall X\varphi(X,B)$, then $T_{\neg\varphi(X,B)}$ is without infinite branches. Fix a well-ordering $R$ of $V_\lambda$, and define the relative Kleene-Brouwer order. Therefore the Kleene-Brouwer order of $T_{\neg\varphi(X,B)}$ is well-founded. But then the Kleene-Brouwer order relative to $j^+(R)$ on $T_{\neg\varphi(X,j^+(B))}$ is well-founded. So $V_\lambda\vDash\forall X\varphi(X,j^+(B))$.
		\end{proof}
		
		\begin{cor}
		 If $j:V\prec M$ witnesses $I2(\lambda)$, then $j\upharpoonright V_\lambda$ is $\Sigma^1_1$.
		\end{cor}
		\begin{proof}
		 If $R$ is well-founded, then $(j\upharpoonright V_\lambda)^+(R)=j(R)$ is well founded in $M$. But by absoluteness of well-foundedness, $j(R)$ is well-founded in $V$.
		\end{proof}
		\begin{defin}[$\lambda$-Ultrafilters Tower]
		\label{ultratower}
			Fix $\lambda$ and $\langle\kappa_n:n\in\omega\rangle$ a cofinal sequence of regular cardinals in $\lambda$. Then $\vec{{\cal U}}=\langle{\cal U}_n:n\in\omega\rangle$ is a \emph{$\lambda$-ultrafilters tower} iff for every $n\in\omega$, ${\cal U}_n$ is a normal, fine, $\kappa$-complete ultrafilter on $[\kappa_{n+1}]^{\kappa_n}$, and if $m<n$ then ${\cal U}_n$ projects to ${\cal U}_m$; i.e., for every $A\in{\cal U}_m$ 
			\begin{equation*}
			 \{x\in[\kappa_{n+1}]^{\kappa_n}:x\cap\kappa_{m+1}\in A\}\in{\cal U}_n.
			\end{equation*}
			
			A $\lambda$-ultrafilters tower is \emph{complete} if for every sequence $\langle A_i\in i\in\omega\rangle$ with $A_i\in{\cal U}_i$ there exists $X\subseteq\lambda$ such that for every $i\in\omega$, $X\cap\kappa_{i+1}\in A_i$.
		\end{defin}
		
		\begin{lem}[Gaifman (\cite{Gaifman}), Powell (\cite{Powell}), 1974]
		\label{gaifpow}
			If $\vec{{\cal U}}$ is a complete $\lambda$-tower, then the direct limit of the ultrapowers $\Ult(V,{\cal U}_n)$ is well-founded.
		\end{lem}
		
		When $j:V_\lambda\prec V_\lambda$, it is natural to consider the sequence $\vec{{\cal U}}_j=\langle{\cal U}_n:n\in\omega\rangle$ defined as 
		\begin{equation*}
		 {\cal U}_n=\{A\subseteq[\kappa_{n+1}]^{\kappa_n}:j"\kappa_{n+1}\in j(A)\}.
		\end{equation*}
		
		\begin{lem}
		 \label{completetower}
		  If $j$ is $\Sigma^1_1$, then $\vec{{\cal U}}_j$ is a complete $\lambda$-tower.
		\end{lem}
		\begin{proof}
			Let $\langle A_i:i\in\omega\rangle$ such that $A_i\in{\cal U}_i$. Let $X=j"\lambda$. Then for every $i\in\omega$, $X\cap\kappa_{i+2}=j"\kappa_{i+1}\in j(A_i)$, so
			\begin{equation*}
			  \exists X\subset\lambda\forall i\in\omega\ X\cap\kappa_{i+2}\in j(A_i). 
			\end{equation*}
			This is a $\Sigma^1_1$ statement, so by $\Sigma^1_1$-elementarity $\exists X\subset\lambda\forall i\in\omega\ X\cap\kappa_{i+1}\in A_i$.
		\end{proof}

		\begin{defin}
			Let $M$ be a transitive class, with $V_\lambda\subset M$. Then $M$ is \emph{$\Sigma^1_n$-correct in $\lambda$} if for every $\Sigma^1_n$-formula $\varphi(X)$, for every $A\subseteq V_\lambda$ and $A\in M$, $(V_\lambda\vDash\varphi(A))^M$ iff $V_\lambda\vDash\varphi(A)$. 
		\end{defin}
		
		\begin{teo}[Martin (\cite{Martin})]
		\label{martin}
			Let $j:V_\lambda\prec V_\lambda$. If $j$ is $\Sigma^1_1$ then 
			\begin{itemize}
			 \item $V_\lambda\subseteq \Ult(V,\vec{{\cal U}}_j)$, so we have I2;
			 \item $\Ult(V,\vec{{\cal U}}_j)$ is $\Sigma^1_2$-correct, therefore $j$ is $\Sigma^1_2$.
			\end{itemize}
		\end{teo}
		\begin{proof}
			Consider the direct limit $\Ult(V,\vec{{\cal U}}_j)$ of the ultrapowers $\Ult(V,{\cal U}_n)$. By Lemma \ref{gaifpow} and Lemma \ref{completetower} $\Ult(V,\vec{{\cal U}}_j)$ is well founded, and it is possible to collapse it on a model $M$.
		 \newline
\newline
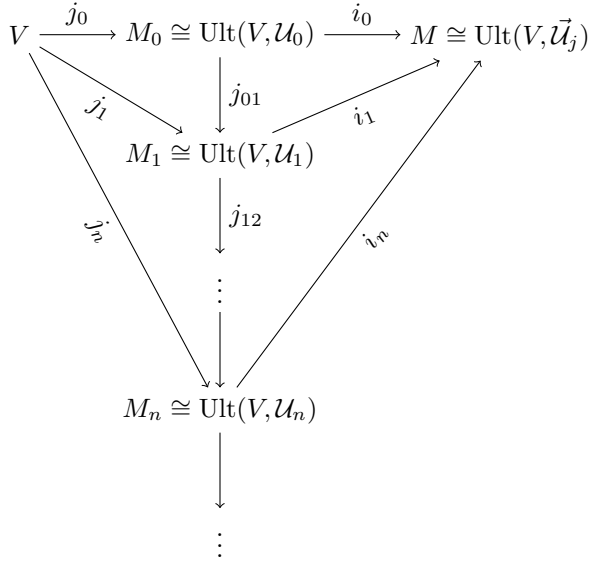
			\begin{tikzpicture}
				\node (V) {$V$};
				\node (M0)[right=of V] {$M_0\cong\Ult(V,{\cal U}_0)$}
					edge [<-] node[above]{$j_0$}(V);
				\node (M1)[below=of M0] {$M_1\cong\Ult(V,{\cal U}_1)$}
					edge [<-] node[sloped,below]{$j_1$}(V)
					edge [<-] node[right]{$j_{01}$}(M0);
				\node (dots)[below=of M1] {$\vdots$}
					edge [<-] node[right]{$j_{12}$}(M1);
				\node (Mn)[below=of dots] {$M_n\cong\Ult(V,{\cal U}_n)$}
					edge [<-] node[sloped,below]{$j_n$}(V)
					edge [<-](dots);
				\node (dots2)[below=of Mn] {$\vdots$}
					edge [<-](Mn);
				\node (M)[right=of M0] {$M\cong\Ult(V,\vec{{\cal U}}_j)$}
					edge [<-] node[above]{$i_0$}(M0)
					edge [<-] node[sloped,below]{$i_1$}(M1)
					edge [<-] node[sloped,below]{$i_n$}(Mn);				
			\end{tikzpicture}
			\newline
\newline
			By the usual theory of normal ultrafilters, it is possible to prove that the above diagram commute, that $\crt(j_n)=\kappa$, $\crt(j_{n,n+1})=\kappa_n$, so $V_{\kappa_n}\subset M_n$, $\crt(i_n)=\kappa_n$ and therefore $V_\lambda\subset M$. Let $i=i_n\circ j_n$ for every $n$. Then $i\upharpoonright V_\lambda=j$, $i(\lambda)=\lambda$ and $i((\lambda^+)^V)=(\lambda^+)^V$.
			
			Observe that for every $B\subseteq V_\lambda$, $B\in M$ we have that $(T_{\varphi(B)})^M=T_{\varphi(B)}$, so $M$ is $\Sigma^1_1$-correct in $\lambda$. Unfortunately, the same cannot be said to the $\Sigma^1_2$ case, since it is possible that $({\cal T}_{\varphi(X,B,Y)})^M\neq{\cal T}_{\varphi(X,B,Y)}$. The reason for this is that while it is true that $(G_m(a))^M=G_m(a)$ for every $m\in\omega$ and $a\in V_{\kappa_m}$, it is possible that there exists $H:G_m(a)\to\lambda^+$ with $H\notin M$.
			
			\begin{cla}
				If $c\in V_\lambda$ and $F:c\to\Ord$, then $i\circ F\in M$.
			\end{cla}
			\begin{proof}
				Let $c\in V_\alpha$, with $\alpha<\lambda$, and pick $n\in\omega$ such that $\crt(i_n)>\alpha$. Then $i\circ F=i_n\circ j_n\circ F=i_n(j_n\circ F)\circ i_n\upharpoonright c=i_n(j_n\circ F)\in M$.
			\end{proof}
			
			By $\Sigma^1_1$ correctness in $\lambda$, 
			\begin{equation*}
			 (V_\lambda\vDash\exists X\forall Y\varphi(X,B,Y))^M\rightarrow V_\lambda\vDash\exists X\forall Y\varphi(X,B,Y),
			\end{equation*}
			so it suffices to prove the other direction.
			
			Suppose that $(V_\lambda\vDash\forall X\exists Y\neg\varphi(X,B,Y))^M$. Then $({\cal T}_{\varphi(X,B,Y)})^M$ is well-founded. Let $L:({\cal T}_{\varphi(X,B,Y)})^M\to\Ord$ that witnesses it. Define $\tilde{L}:{\cal T}_{\varphi(X,B,Y)}\to\Ord$ as $\tilde{L}((a,H))=L((a,i\circ H))$. By the previous claim $i\circ H\in M$, and moreover $i\circ H$ is a function from $G_m(a)$ to $i(\lambda^+)=\lambda^+$, so $(a,i\circ H)\in({\cal T}_{\varphi(X,B,Y)})^M$ and $\tilde{L}$ is well-defined. It remains to prove that $\tilde{L}$ witnesses that ${\cal T}_{\varphi(X,B,Y)}$ is well-founded. Suppose that $(a,H)<(a',H')$. If in $T_{\varphi(X,B,Y)}$ we have that $(a,c,F,P)<(a',c',F',P')$ then $H(c,F,P)>H'(c',F',P')$, so
			\begin{equation*}
			 i\circ H(c,F,P)>i\circ H'(c',F',P'),
			\end{equation*}
			therefore $(a,i\circ H)<(a',i\circ H')$. It follows that 
			\begin{equation*}
			  L((a,i\circ H))>L((a',i\circ H')),
			\end{equation*}
			so $\tilde{L}((a,H))>\tilde{L}((a',H')$.
			
			Then ${\cal T}_{\varphi(X,B,Y)}$ cannot have an infinite branch and $V_\lambda\vDash\forall X\exists Y\neg\varphi(X,B,Y)$, i.e., $M$ is $\Sigma^1_2$-correct in $\lambda$.
			
			This proves that $i\upharpoonright V_\lambda$ is $\Sigma^1_2$: if $\varphi$ is $\Sigma^1_2$ in $V_\lambda$ and $V\vDash\varphi(B)$, then by elementarity $\Ult(V,\vec{{\cal U}}_j)\vDash\varphi(i(B))$, therefore by $\Sigma^1_2$-correctness also $V\vDash\varphi(i(B))$. Since $i\upharpoonright V_\lambda=j$ we're done.			
		\end{proof}
		
		This theorem proves that $\Sigma^1_1$ and $\Sigma^1_2$ elementary embeddings in $V_\lambda$ are the same, and they are equivalent to I2($\lambda$). It is possible, after a pair of technical lemmas, to generalize this for $\Sigma^1_n$ and $\Sigma^1_{n+1}$.
		
		\begin{lem}[Laver, \cite{Laver3}]
			\label{Sigmavalue}
			Let $n$ be odd. Then ``$j$ is $\Sigma^1_n$'' is a $\Pi^1_{n+1}$ formula in $V_\lambda$, with $j$ as a parameter.
		\end{lem}
		\begin{proof}
			By definition $j$ is $\Sigma^1_n$ iff
			\begin{multline*}
			 \forall B\subseteq V_\lambda\ \forall\Sigma^1_0\text{ formula }\varphi(X_1,\dots,X_n,Y)\\
			  \exists A_1\ \forall A_2\dots\exists A_n\ V_\lambda\vDash\varphi(\vec{A},j^+(B))\rightarrow \exists A_1\ \forall A_2\dots\exists A_n\ V_\lambda\vDash\varphi(\vec{A},B).
			\end{multline*}
			But for every $D\subset V_\lambda$, $V_\lambda\vDash\varphi(\vec{A},D)$ iff there exists a branch in $T_{\varphi(\vec{A})}$ whose projection is $D$. Using this is easy to check the Lemma.
		\end{proof}
		
		\begin{lem}[Square root Lemma for $\Sigma^1_n$, Laver, \cite{Laver3}]
			\label{change}
			Let $n$ be odd, $n>1$. Let $j$ be $\Sigma^1_n$, $\crt(j)=\kappa$. Let $\beta<\kappa$, $A,B\subseteq V_\lambda$. Then there exists a $k:V_\lambda\prec V_\lambda$ that is $\Sigma^1_{n-1}$ such that $k(k)=j$, $k(B)=j(B)$, $k(A')=A$ for some $A'\subseteq V_\lambda$ and $\beta<\crt(k)<\kappa$.
		\end{lem}
		\begin{proof}
			The formula ``$\exists k,Y,\ k\text{ is }\Sigma^1_{n-2}$, $k(k)=j$, $k(B)=j(B), k(Y)=A,\beta<\crt(k)<\kappa$'' is $\Sigma^1_n$ by Lemma \ref{Sigmavalue}. The following formula is clearly true:
			\begin{multline*}
				j\text{ is }\Sigma^1_{n-2},\ j(j)=j(j),\ j(j(B))=j(j(B)),\\
				j(A)=j(A),\ j(\beta)=\beta<\crt(j)<j(\kappa)
			\end{multline*}
			but then, with a smart quantification of some of the parameters of the formula above, we have
			\begin{multline*}
				\exists k,Y\ k\text{ is }\Sigma^1_{n-2},\ k(k)=j(j),\ k(j(B))=j(j(B)),\\
				k(Y)=j(A),\ j(\beta)<\crt(k)<j(\kappa)
			\end{multline*}
			By elementarity the lemma is proved.
		\end{proof}
		
		The lemma therefore provides a ``square root'' for $j$, with certain desirable characteristics, but with the catch that for a square root of a $\Sigma^1_n$-elementary embedding we can only assure that it is $\Sigma^1_{n-2}$. It is not possible to do better: suppose that $\kappa$ is the minimum critical point of $j:V_\lambda\prec V_\lambda$ that is $\Sigma^1_n$. Then the critical point of the square root $k$ of $j$ is less then $\kappa$, therefore $k$ cannot be $\Sigma^1_n$. Also note that the square root of an iterable embedding is iterable, as a square root of $j$ is, in a certain sense, $j^{-1}$. 
		
		Note that this would be enough to prove strong implication if we considered the existence of the embeddings as a property of the critical point, and we would have also that I2($\kappa$) strongly implies iterable I3($\kappa$). But we already argued that this seems not the right way to look at this properties. Moreover, it is possible to prove the strong implications of such axioms respect to $\lambda$:
		
		 \begin{cor}
  I2 strongly implies iterable I3.
 \end{cor}
 \begin{proof}
   Note that ``$j$ is $\omega_1$-iterable'', and therefore ``$j$ is iterable'', is $\Sigma^1_1$ (there exist $\alpha<\omega_1$, $M_0,\dots,M_\alpha)$, $j_0,\dots,j_\alpha$ such that\dots). Therefore one can put ``$j$ is iterable'' in \ref{change} and the proof works.
 \end{proof}
		
		\begin{teo}[Laver, \cite{Laver3}]
			\label{sameconsistency}
			Let $n$ be odd. Then if $j$ is $\Sigma^1_n$, it is also $\Sigma^1_{n+1}$.
		\end{teo}
		\begin{proof}
			Theorem \ref{martin} is the case $n=1$. We proceed by induction on $n$.
			
			Let $B\subseteq V_\lambda$. We have to prove that for every $\psi$ $\Sigma_{n+1}$ formula and $B\subseteq V_{\lambda+1}$, 
			\begin{equation*}
			 V_\lambda\vDash\psi(B)\text{ iff }V_\lambda\vDash\psi(j^+(B)).
			\end{equation*}
			By induction, the direction from left to right is immediate, so it suffices to show the other direction. Suppose 
			\begin{equation*}
			 V_\lambda\vDash\exists X_1\ \forall X_2\ \exists X_3\dots\forall X_{n+1}\ \varphi(X_1,\dots,X_{n+1},j^+(B)). 
			\end{equation*}
			Let $X_1=A$ a witness. By Lemma \ref{change}, we can pick a $k$ that is $\Sigma^1_{n-2}$ such that $k(B)=j(B)$ and there exists $A'\subseteq V_{\lambda+1}$ such that $k(A')=A$. By induction, $k$ is also $\Sigma^1_{n-1}$, and since 
			\begin{equation*}
			 V_\lambda\vDash\forall X_2\ \exists X_3\dots\forall X_{n+1}\ \varphi(k(A'),X_2,\dots,X_{n+1},k(B)),
			\end{equation*}
			then 
			\begin{equation*}
			 V_\lambda\vDash\forall X_2\ \exists X_3\dots\forall X_{n+1}\ \varphi(A',X_2,\dots,X_{n+1},B).
			\end{equation*}
			Therefore $A'$ is a witness for $\psi(B)$.
		\end{proof}
		
		Theorem \ref{sameconsistency} shows a peculiar asimmetry. What problems arise when $n$ is even? The key is in the proof of Lemma \ref{Sigmavalue}. When the number of the quantifiers of a $\Sigma_n$ formula is odd, then the last one is an existential quantifier. Since by Lemma \ref{SatisfactionTree1} the satisfaction relation for $\Sigma^1_1$ formulae is $\Sigma^1_1$, this quantifier is absorbed by the satisfaction formula. When $n$ is even, however, the last quantifier is a universal one, and ``being $\Sigma^1_n$`` is still $\Pi^1_{n+2}$. This seems an ineffectual detail, but in fact Laver proved that it is an essential one, since he showed that being a $\Sigma^1_{n+1}$-embedding, for $n$ even, is strictly stronger than being a $\Sigma^1_n$ one.
		
		\begin{teo}[Laver, \cite{Laver3}]
			\label{Application}
			Let $h$ and $k$ be $\Sigma^1_n$ embeddings. Then $h\cdot k$ is a $\Sigma^1_n$-embedding.
		\end{teo}
		\begin{proof}
			By Theorem \ref{sameconsistency} we can suppose $n$ odd. Let $\varphi$ be a $\Sigma^1_0$ formula with $n+1$ variables. Then
			\begin{multline*}
				V_\lambda\vDash\forall Y\ \bigl(\exists X_1\ \forall X_2\dots\exists X_n\ \varphi(X_1,\dots,X_n, k(Y))\\
				\rightarrow\exists X_1\ \forall X_2\dots\exists X_n\ \varphi(X_1,\dots,X_n,Y)\bigr).
			\end{multline*}
			This formula is $\Pi^1_{n+1}$. Since by Theorem \ref{sameconsistency} $h$ is $\Sigma^1_{n+1}$, we have
			\begin{multline*}
				V_\lambda\vDash\forall Y\ \bigl(\exists X_1\ \forall X_2\dots\exists X_n\ \varphi(X_1,\dots,X_n, h(k)(Y))\\
				\rightarrow\exists X_1\ \forall X_2\dots\exists X_n\ \varphi(X_1,\dots,X_n,Y)\bigr).
			\end{multline*}
			so $h(k)$ is $\Sigma^1_n$.
		\end{proof}
		
		Note that the converse is not true: minimizing the critical point, we can find a $k$ that is $\Sigma^1_n$ such that any $h$ so that $h(h)=k$ cannot be $\Sigma^1_n$ (see comment after Lemma \ref{change}).
		
		However, if we switch application with composition, then also the converse is true.
		
		\begin{lem}
			\label{forComposition}
			If $h$ and $k$ are $\Sigma^1_m$ and $h\circ k$ is $\Sigma^1_{m+1}$, then $k$ is $\Sigma^1_{m+1}$.
		\end{lem}
		\begin{proof}
			Suppose that 
			\begin{equation*}
				V_\lambda\vDash\forall X_1\ \exists X_2\dots\forall X_{m+1}\ \varphi(X_1,\dots,X_{m+1},B) 
			\end{equation*}
			with $B\subseteq V_\lambda$. Then 
			\begin{equation*}
				V_\lambda\vDash\forall X_1\ \exists X_2\dots\forall X_{m+1}\ \varphi(X_1,\dots,X_{m+1},h\circ k(B)). 
			\end{equation*}
			Let $A\subseteq V_\lambda$. Then in particular 
			\begin{equation*}
				V_\lambda\vDash\exists X_2\dots\forall X_{m+1}\ \varphi(h(A),X_2,\dots,X_{m+1},h\circ k(B)). 
			\end{equation*}
			By elementarity 
			\begin{equation*}
				V_\lambda\vDash\exists X_2\dots\forall X_{m+1}\ \varphi(A,X_2,\dots,X_{m+1},k(B)). 
			\end{equation*}
			Since this is true for every $A\subseteq V_\lambda$, we have 
			\begin{equation*}
				V_\lambda\vDash\forall X_1\ \exists X_2\dots\forall X_{m+1}\ \varphi(X_1,\dots,X_{m+1},k(B)).
			\end{equation*}
		\end{proof}
		
		\begin{teo}[Laver, \cite{Laver3}]
		\label{Composition}
			Let $h,k\in{\cal E}_\lambda$. Then $h,k$ are $\Sigma^1_n$ iff $h\circ k$ is $\Sigma^1_n$. 
		\end{teo}
		\begin{proof}
			If $h$ and $k$ are $\Sigma^1_n$, then obviously $h\circ k$ is $\Sigma^1_n$. We prove by induction on $m\leq n$ that $h$ and $k$ are $\Sigma^1_m$.
			
			The case $m=0$ is by hypothesis. Suppose it is true for $m$. Then by Theorem \ref{Application} $h(k)$ is $\Sigma^1_m$. It is easy to calculate that $h\circ k=h(k)\circ h$, and this is $\Sigma^1_{m+1}$ by hypothesis. By using Lemma \ref{forComposition} in the left side, we have that $k$ is $\Sigma^1_{m+1}$, and using it on the right side we have that $h$ is $\Sigma^1_{m+1}$.
		\end{proof}
		
		Theorem \ref{Composition} is promising for our objective, that is proving that being $\Sigma^1_{n+2}$ is strictly stronger than being $\Sigma^1_n$ for an elementary embedding. The most natural idea for doing this is using some sort of reflection, to prove that if there is a $j\in{\cal E}_\lambda$ $\Sigma^1_{n+2}$, then there is a $k\in{\cal E}_{\lambda'}$ that is $\Sigma^1_n$ for some $\lambda'<\lambda$. It would be tempting to use, as in the case of iterability, direct limit of elementary embeddings, but unfortunately in this setting there are counterexamples:
		
		\begin{teo}[Laver, \cite{Laver2}]
			There exists a $j$ that is $\Sigma^1_n$ that has a stabilizing direct limit of members of ${\cal A}_j$ that is not $\Sigma^1_1$,
		\end{teo}
		
		So we will consider inverse limits instead.
		
		Let $\langle j_0,j_1,\dots\rangle$ be a sequence of elements of ${\cal E}_\lambda$, and let $J=j_0\circ j_1\circ\dots$ be the inverse limit of the sequence. By definition the dominion of $J$ is $\{x\in V_\lambda:\exists n_x\ \forall i>n_x\  j_i(x)=x\}$. But we know that $j_i(x)=x$ iff $x\in V_{\crt(j_i)}$, so this is $\{x\in V_\lambda:\exists n_x\ \forall i>n_x\ x\in V_{\crt(j_i)}\}$. That is, $x\in\dom J$ depends only on the rank of $x$, and this implies that $\dom J=V_\alpha$ for some $\alpha$. It is also possible to calculate $\alpha$, since $\beta<\alpha$ iff $\exists n\ \forall m\geq n\ \beta<\crt(j_m)$:
		  \begin{equation*}
		   \alpha=\sup_{n\geq 0}\inf_{m\geq n}\crt(j_m)=\liminf_{n\in\omega}\crt(j_n).
		  \end{equation*}
		
		With some cosmetic change, we can also suppose that $\alpha$ as the supremum of the critical points, not only the limit inferior. This will also simplify the following proofs and notations.

		So let $\lambda_n=\inf_{m\geq n}\crt(j_m)$. Then $\alpha=\sup_{n\in\omega}\lambda_n$, and $\lambda_n$ is increasing in $n$. If the supremum is also a maximum, we incur in the trivial case, where $J$ is in fact just a finite composition of elementary embeddings: if $n$ is the first one such that $\lambda_n=\alpha$, then all $\crt(j_m)$ with $m>n$ are bigger than $\alpha$, so they are constant in the domain of $J$ and they don't change anything.
		
		Suppose then that $\alpha$ is a proper supremum of the $\lambda_n$ sequence. We can suppose $\crt(j_n)=\lambda_n$ by aggregating multiple elementary embeddings in just one: consider the largest $n$ such that $\crt(j_n)=\lambda_0$ (there will be a largest one because $\alpha$ is a proper supremum of the $\lambda_n$ sequence), define the new $k_0$ as the old $j_0\circ\dots\circ j_n$, and repeat this for every $\lambda_n$. The following is a graphical example:
		\newline
		\newline
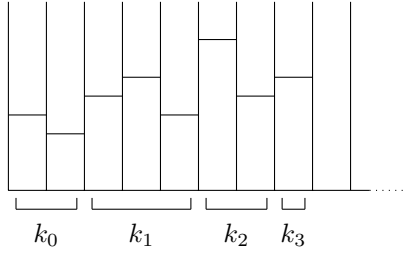
		\begin{tikzpicture}
			\draw (0,0) grid [xstep=0.5, ystep=3] (4.5,2.5);
			\draw (0,1) -- (0.5,1);
			\draw (0.5,0.75) -- (1,0.75);
			\draw (1,1.25) -- (1.5,1.25);
			\draw (1.5,1.5) -- (2,1.5);
			\draw (2,1) -- (2.5,1);
			\draw (2.5,2) -- (3,2);
			\draw (3,1.25) -- (3.5,1.25);
			\draw (3.5,1.5) -- (4,1.5);
			\draw (4.5,0) -- (4.75,0);
			\draw [dotted] (4.75,0) -- (5.25,0);
			\draw (0.1,-0.25) -- node [below= 2pt] {$k_0$} (0.9,-0.25);
			\draw (0.1,-0.25) -- (0.1,-0.1);
			\draw (0.9,-0.25) -- (0.9,-0.1);
			\draw (1.1,-0.25) -- node [below= 2pt] {$k_1$} (2.4,-0.25);
			\draw (1.1,-0.25) -- (1.1,-0.1);
			\draw (2.4,-0.25) -- (2.4,-0.1);
			\draw (2.6,-0.25) -- node [below= 2pt] {$k_2$} (3.4,-0.25);
			\draw (2.6,-0.25) -- (2.6,-0.1);
			\draw (3.4,-0.25) -- (3.4,-0.1);
			\draw (3.6,-0.25) -- node [below= 2pt] {$k_3$} (3.9,-0.25);
			\draw (3.6,-0.25) -- (3.6,-0.1);
			\draw (3.9,-0.25) -- (3.9,-0.1);
		\end{tikzpicture} 
		\newline
\newline
		The columns represent the behaviours of each $j_n$ on $\lambda$, where the column on the left represent $j_0$, and the horizontal lines indicate the critical point.
		
		\begin{defin}
		 Let $J=j_0\circ j_1\circ\dots$. Then we define $J_n=j_n\circ j_{n+1}\circ\dots$ and $J_{0(n-1)}=j_0\circ j_1\circ\dots\circ j_{n-1}$.
		\end{defin}
		
		\begin{lem}
		\label{InverseLimit}
		 Let $j_m\in {\cal E}_\lambda$ for every $m\in\omega$, define $\alpha_m=\crt(j_m)$ and suppose that for every $m\in\omega$, $\alpha_m<\alpha_{m+1}$. Let $\alpha=\sup_{m\in\omega}\alpha_m$ and $J=j_0\circ j_1\circ\dots$. Then
		 \begin{itemize}
		 	\item $J''\alpha$ is unbounded in $\lambda$;
		 	\item $J: V_\alpha\prec V_\lambda$ is elementary.
		 \end{itemize}
		\end{lem}
		\begin{proof}
			\begin{itemize}
				\item Let $\delta_m=\sup J_m``\alpha$. We prove that when $\delta_m<\lambda$, then $\delta_m\neq\delta_{m+1}$. Obviously $\delta_m\geq\alpha$, because otherwise $J_m(\delta_m)\in J_m``\alpha$ and $J_m(\delta_m)<\delta_m$. In particular $\delta_m$ is above $\alpha_m$, i.e., the critical point of $j_m$, so $\delta_m$ is moved by $j_m$ and there exists a $\mu<\delta_m$ such that $j_m(\mu)>\delta_m$. By contradiction, suppose that $\delta_m=\delta_{m+1}$. Then $\mu$ is also less than $\delta_{m+1}$, so by definition there exists an $i\in\omega$ such that $J_{m+1}(\alpha_i)\geq\mu$. Therefore 
					\begin{equation*}
						J_m(\alpha_i)=j_m\circ J_{m+1}(\alpha_i)\geq j_m(\mu)\geq\delta_m, 
					\end{equation*}
					contradiction.
				
				Suppose then that $\delta_0<\lambda$. Therefore for any $m\in\omega$ $\delta_{m+1}<\delta_m$, but this creates a strictly descending sequence of ordinals, contradiction.
				
				\item Fix $n\in\omega$ and let $k=j_0(j_1(\dots j_n(j_n)\dots))$. Then 
				\begin{multline*}
					\crt(k)=\crt(j_0(j_1(\dots j_n(j_n)\dots)))=\\
					=j_0(j_1(\dots j_n(\crt(j_n))\dots))=j_0\circ j_1\circ\dots\circ j_n(\alpha_n)=J(\alpha_n). 
				\end{multline*}
				By Lemma \ref{reflection}, then, $V_{J(\alpha_n)}\prec V_\lambda$. But $J\upharpoonright V_{\alpha_n}:V_{\alpha_n}\to V_{J(\alpha_n)}$ is an elementary embedding, because $J\upharpoonright V_{\alpha_n}=J_{0,n}\upharpoonright V_{\alpha_n}$, so for every $n\in\omega$, $J\upharpoonright V_{\alpha_n}\prec V_\lambda$. With methods similar to those in the proof of Theorem \ref{Delta0}, it is possible to prove that this implies $J:V_\alpha\prec V_\lambda$.
			\end{itemize}
		\end{proof}
		
		Like in the $V_\lambda$ case, we can extend $J$ to $V_{\alpha+1}$ in the expected way: when $A\subseteq V_\alpha$, $J(A)=\bigcup_{\beta<\alpha}J(A\cap V_\beta)$. Now we want to prove an equivalent of Lemma \ref{change}, but for inverse limits. 

		\begin{lem}[Square root for inverse limits of $\Sigma^1_n$]
		\label{ChangeInverse}
		 Let $J:V_\alpha\prec V_\lambda$ an inverse limit of $\Sigma^1_{n+1}$ elementary embeddings. Then for all $A,B\subseteq V_\alpha$ there exist $K:V_\alpha\prec V_\lambda$ inverse limit of $\Sigma^1_n$ elementary embeddings and $A'\subseteq V_\alpha$ such that $k(A')=A$ and $k(B)=J(B)$.
		\end{lem}
		\begin{proof}
		 We define $k_m$ and $A_m$ by induction, with repeated uses of Lemma \ref{change}. At the end, $K$ will be the inverse limit of the $k_m$'s, and the $A_m$'s will be the images of $A'$ through the inverse limit. 
		 
		 Let $A_0=A$, $k_0$ and $A_1$ such that $k_0$ is a $\Sigma^1_n$ elementary embedding, $A_1\subseteq V_\lambda$, $k_0(A_1)=A=A_0$, 
		 \begin{equation*}
		  k_0(J_1(B))=j_0(J_1(B))=J(B)
		 \end{equation*}
		 and $\crt(k_0)<\crt(j_0)$. 
		 
		 More generally, $k_{m+1}$ and $A_{m+2}$ are such that $k_{m+1}$ is a $\Sigma^1_n$ elementary embedding, $A_{m+2}\subseteq V_\lambda$, $k_{m+1}(A_{m+2})=A_{m+1}$, 
		 \begin{equation*}
		  k_{m+1}(J_{m+2}(B))=j_{m+1}(J_{m+2}(B))
		 \end{equation*}
		 and $\crt(j_m)<\crt (k_{m+1})<\crt (j_{m+1})$.
		 
		 So $k_m$ and $A_m$ satisfy:
		 \begin{itemize}
		 	\item $k_m$ is $\Sigma^1_n$;
		 	\item $k_m(A_{m+1})=A_m$;
		 	\item $k_m(J_{m+1}(B))=J_m(B)$;
		 	\item $\crt(k_0)<\crt(j_0)<\crt(k_1)<\dots<\crt(j_m)<\crt(k_m)<\crt(j_{m+1})<\dots$
		 \end{itemize}
		 
		 Let $K$ be the inverse limit of the $k$'s. Then $\sup_{m\in\omega}\crt (k_m)=\alpha$, and by Lemma \ref{InverseLimit} $K:V_\alpha\prec V_\lambda$ is an elementary embedding. Note that
		 \begin{multline*}
		 	K(B\cap V_{\crt(k_m)})=k_0\circ k_1\circ\dots\circ k_m(K_{m+1}(B\cap V_{\crt(k_m)})=\\
			=k_0\circ\dots\circ k_m(B\cap V_{\crt(k_m)})=k_0\circ\dots\circ k_m(J_{m+1}(B\cap V_{\crt(k_m)})=\\
			=J(B\cap V_{\crt(k_m)}).
		 \end{multline*}
		 So $K(B)=J(B)$.
		 
		 Finally, consider $A_{m+1}\cap V_{\crt(k_m)}$:
		 \begin{equation*}
		 	A_{m+2}\cap V_{\crt(k_m)}=k_{m+1}(A_{m+2}\cap V_{\crt(k_m)})=A_{m+1}\cap V_{\crt(k_m)},
		 \end{equation*}
		 so
		 \begin{equation*}
		 	K(A_{m+1}\cap V_{\crt(k_m)})=k_0\circ\dots\circ k_m(A_{m+1}\cap V_{\crt(k_m)})=A\cap V_{K(\crt(k_m))}.
		 \end{equation*}
		 Define $A'=\bigcup_{m\in\omega}(A_{m+1}\cap V_{\crt(k_m)})$. Then $K(A')=A$.
		\end{proof}
		
		We use Lemma \ref{ChangeInverse} to calculate the strength of an inverse limit:
		
		\begin{teo}[Laver, \cite{Laver3}]
		\label{StrengthInvLimit}
			If $J:V_\alpha\prec V_\lambda$ is an inverse limit of $\Sigma^1_n$ elementary embeddings, then $J$ is $\Sigma^1_n$.
		\end{teo}
		\begin{proof}
			The case $n=0$ is Lemma \ref{InverseLimit} combined with an obvious generalization of Theorem \ref{Delta0}, so we proceed by induction on $n$.
			
			Suppose that $J$ is $\Sigma^1_{n-1}$, we need to prove that for every $\varphi$ $\Pi^1_{n-1}$-formula, and any $B\subseteq V_\lambda$, 
			\begin{equation*}
			 V_\lambda\vDash\exists X\varphi(X,J(B))\rightarrow V_\alpha\vDash\exists X \varphi(X,B).
			\end{equation*}
			Suppose $V_\lambda\vDash\exists X\varphi(X,J(B))$, and fix $A$ a witness. Using Lemma \ref{ChangeInverse}, we find $K$ inverse limit of $\Sigma^1_{n-2}$ elementary embeddings such that $K(A')=A$ and $K(B)=J(B)$ for some $A'\subseteq V_\alpha$. So $V_\lambda\vDash\varphi(K(A'),K(B))$, and by elementarity $V_\alpha\vDash\varphi(A',B)$, that is $V_\alpha\vDash\exists X\ \varphi(X,B)$.
		\end{proof}

		Finally, we can prove that the existence of a $\Sigma^1_{n+2}$ elementary embedding strongly implies the existence of a $\Sigma^1_n$ elementary embedding.

		\begin{teo}[Laver, \cite{Laver3}]
		\label{strongimplEn}
		 Let $j:V_\lambda\prec V_\lambda$ be $\Sigma^1_{n+2}$. Then
		 \begin{itemize}
		  \item for every $B\subseteq V_\lambda$, there exist an $\alpha<\lambda$ and a $k_\alpha:V_\alpha\prec V_\lambda$ such that $k_\alpha(B_\alpha)=B$ for some $B_\alpha\subseteq V_\alpha$. In fact, we can find an $\omega$-club $C\subseteq\lambda$ of such $\alpha$'s.
		  \item there exist an $\alpha<\lambda$ and a $j_\alpha:V_\alpha\prec V_\alpha$ that is $\Sigma^1_n$. Moreover, we can find an $\omega$-club $C\subseteq\lambda$ of such $\alpha$'s. Therefore ``there exists $j:V_\lambda\prec V_\lambda$ $\Sigma^1_{n+2}$'' strongly implies ``$j:V_\lambda\prec V_\lambda$ $\Sigma^1_{n+2}$''. In particular I2 strongly implies I3.
		 \end{itemize}
		\end{teo}
		\begin{proof}
		 \begin{itemize}
		  \item Let 
		  	\begin{multline*}
					G=\{\langle l_0,\dots,l_m\rangle:l_i:V_\lambda\prec V_\lambda\text{ is }\Sigma^1_n,\ \crt l_0<\crt l_1<\dots<\crt l_m<\kappa_0,\\
					\exists B_0,\dots,B_m\ l_0(B_0)=B,\ \forall i\ l_{i+1}(B_{i+1})=B_i\}. 
				\end{multline*}
				By Lemma \ref{change} the set 
				\begin{equation*}
					\{\theta<\kappa_0:\exists l\ \langle l_0,\dots,l_m,l\rangle\in G,\crt(l)=\theta\} 
				\end{equation*}
				is unbounded in $\kappa_0$. Pick an infinite branch $\langle l_0,l_1,\dots\rangle$ of $G$, and let $\alpha=\sup_{i\in\omega}\crt(l_i)$. Let $k_\alpha$ be the inverse limit of the $l_i$'s. Then by Theorem \ref{StrengthInvLimit} $k_\alpha$ is $\Sigma^1_n$. Define $B'=\bigcup_{m\in\omega}(B_{m+1}\cap V_{\crt(l_m)})$ as in Lemma \ref{ChangeInverse} to have $k_\alpha(B')=B$. To prove the existence of the $\omega$-club $C$, note that we could have used \emph{any} infinite branch of $T$, and the set of the ordinals that are the supremum of the critical points of the elementary embeddings appearing in an infinite branch of $T$ (like $\alpha$) contains an $\omega$-club.

		 \item	Let the $B$ above be $j$. Then there exists $\alpha$ (again, any $\alpha\in C$ works), $k_\alpha:V_\alpha\prec V_\lambda$, and $j_\alpha\subseteq V_\alpha$ such that $k_\alpha(j_\alpha)=j$. Suppose, by Theorem \ref{sameconsistency}, that $n$ is odd. Then by Lemma \ref{Sigmavalue} ''$j$ is $\Sigma^1_n$'' is $\Pi^1_{n+1}$. Again by Theorem \ref{sameconsistency} and Lemma \ref{StrengthInvLimit} $k_\alpha$ is $\Sigma^1_{n+1}$, so by elementarity $j_\alpha:V_\alpha\prec V_\alpha$ is $\Sigma^1_n$.	
		\end{itemize}
		\end{proof}

 The situation for I1 has some peculiarities. First of all, if $j:V_{\lambda+1}\prec V_{\lambda+1}$, then $j$ is definable from $j\upharpoonright V_\lambda$, so it is actually a definable class of $V_{\lambda+1}$ This is not possible in the I3 case, as $V_\lambda\vDash\ZFC$, so it would contradict the naive Kunen's Theorem. We can reformulate Lemma \ref{Sigmavalue} in this way:

\begin{rem}
 \label{I1}
 Suppose that $j:V_{\lambda+1}\to V_{\lambda+1}$. Then there are $\Pi_{n+1}$ formulas $\varphi_n(x)$ such that for any $n\in\omega$, $V_{\lambda+1}\vDash\varphi_n(j\upharpoonright V_\lambda)$ iff $j$ is an elementary embedding.  
 \end{rem}

So let $j:V_{\lambda+1}\prec V_{\lambda+1}$. Consider $j^-=j\upharpoonright V_\lambda$. We can define $j^-(j^-):V_\lambda\prec V_\lambda$. By the remark above $[j^-(j^-)]^+$ is an elementary embedding, and we write it simply as $j(j)$. So we can iterate also the I1 embeddings.

As $j:V_{\lambda+1}\prec V_{\lambda+1}$ is $\Sigma^1_n$ for every $n$, by Lemma \ref{change} for any $n$ it is possible to find a square root of $j$ that is $\Sigma^1_n$. Therefore for any $n$ is it possible to find $\alpha<\lambda$ and $K:V_\alpha\prec V_\lambda$ inverse limit that is $\Sigma^1_n$, so I1 strongly implies any ``There exists $j:V_\lambda\prec V_\lambda$ that is $\Sigma^1_n$''. We can do more:

\begin{cor}
 I1($\lambda$) strongly implies $\forall n\ \exists j:V_\lambda\prec V_\lambda$ that is $\Sigma^1_n$
\end{cor} 
\begin{proof}
 Run the proof of Theorem \ref{strongimplEn} $n$ times, each time considering that $j$ is $\Sigma^1_n$. So there are $C_n$ $\omega$-clubs in $\lambda$ such that for any $\alpha\in C_n$ there exists $k:V_\alpha\prec V_\alpha$ that is $\Sigma^1_n$. The intersection of $\omega$ $\omega$-clubs is an $\omega$-club, so $\alpha\in\bigcap_{n\in\omega}C_n$ is as desired. 
\end{proof}

It is possible to push this kind of proof even more:

\begin{lem} [Square root Lemma for $L_(V_{\lambda+1})$, Laver, \cite{Laver3}]
   \label{ReflectionI0}
   Let $j:L_1(V_{\lambda+1})\prec L_1(V_{\lambda+1})$. For every $A,B\subseteq V_\lambda$, $\beta<\crt(j)$ there exists $k:V_{\lambda+1}\prec V_{\lambda+1}$, with $\beta<\crt(k)<\crt(j)$, such that $k(B)=j(B)$ and there exists $A'\subseteq V_\lambda$ such that $k(A')=A$.
  \end{lem}
  \begin{proof}
   The only detail we should care of is the fact that ``$j:V_{\lambda+1}\prec V_{\lambda+1}$'' must be definable in $L_1(V_{\lambda+1})$. But this is true, because the satisfaction relation in $V_{\lambda+1}$ is definable in $L_1(V_{\lambda+1})$, so the proof is, \emph{mutatis mutandis}, the same as Lemma \ref{change}. 
  \end{proof}
	
	\begin{teo}[Laver, \cite{Laver3}]
    \label{FirstStepI0}
    ``There exists $j:L_1(V_{\lambda+1})\prec L_1(V_{\lambda+1})$'' strongly implies ``there exists $j:V_{\lambda+1}\prec V_{\lambda+1}$''.   
  \end{teo}
	\begin{proof}
   The proof follows the same method of Theorem \ref{strongimplEn}.
	\end{proof}
	
	Note that in this case there is a fundamental difference between being $\Sigma^1_n$ for every $n$ and being an elementary embedding from $V_{\lambda+1}$ to itself. Suppose that we have $j_i:V_{\lambda+1}\prec V_{\lambda+1}$. Then the inverse limit $J$ of $\langle j_i:i\in\omega\rangle$ is \emph{not} an elementary embedding from $V_{\lambda'+1}$ to $V_{\lambda+1}$, because the domain of $J$ is just not $V_{\lambda'+1}$. Consider again the definition of the domain of $J$: 
	\begin{equation*}
	 H=\{x\in V_{\lambda+1}:\exists n\in\omega\ \forall m\geq n\ j_m(x)=x\}. 
	\end{equation*}
	Following this definition, for example, $\lambda'\notin H$, because it is moved by all $j_i$, while on the other hand $\lambda\in H$, because it is never moved. What we can prove with the methods provided, is that the unique extension of the inverse limits of the $j_i\upharpoonright V_\lambda$ to $V_{\lambda'+1}$ is an elementary embedding. 
	
	As I0 is equivalent to $j:L_\Theta(V_{\lambda+1})\prec L_\Theta(V_{\lambda+1})$, we consider now hypotheses like ``there exists $j:L_\alpha(V_{\lambda+1})\prec L_\alpha(V_{\lambda+1})$'', $\alpha<\Theta$, with the aim of finding strong implications among them. 
	
	Note that with such aim not any $\alpha$ will be appropriate. We defined strong implications only between properties of $\lambda$, while ``there exists $j:L_\alpha(V_{\lambda+1})\prec L_\alpha(V_{\lambda+1})$'' is a property of $\alpha$ and $\lambda$. We should consider therefore only $\alpha$'s that depend on $\lambda$, for example $\alpha$'s that are definable from $\lambda$ (e.g., $\lambda^+$, $\lambda^++\omega$, \dots) or just absolutely definable (e.g., $\omega$, $\omega^\omega$, \dots). This kind of ordinals have also the characteristic of being a fixed point for any large enough embedding, so we already have that I0 implies ``there exists $j:L_\alpha(V_{\lambda+1})\prec L_\alpha(V_{\lambda+1})$''.
	
	Our objective is now, given $j:L_\beta(V_{\lambda+1})\prec L_\beta(V_{\lambda+1})$ to find a $\alpha<\beta$ and $\bar{\alpha}$, $\bar{\lambda}<\lambda$ such that $L_{\bar{\alpha}}(V_{\bar{\lambda}+1})\equiv L_\alpha(V_{\lambda+1})$ and there exists $k:L_{\bar{\alpha}}(V_{\bar{\lambda}+1})\prec L_{\bar{\alpha}}(V_{\bar{\lambda}+1})$. With $\beta$ and $\alpha$ as above, this will imply that ``exists $j:L_\beta(V_{\lambda+1})\prec L_\beta(V_{\lambda+1})$'' strongly implies ``exists $j:L_\alpha(V_{\lambda+1})\prec L_\alpha(V_{\lambda+1})$''.
	
	The strategy is to follow Laver's proof for strong implications as much as possible. The first step is therefore to prove some square root principle like Lemma \ref{change}. The best possible outcome would be to find for any $j:L_{\alpha+1}(V_{\lambda+1})\prec L_{\alpha+1}(V_{\lambda+1})$ a square root $k$ that is $k:L_\alpha(V_{\lambda+1})\prec L_\alpha(V_{\lambda+1})$. For this, it is needed that 
	\begin{equation*}
	L_{\alpha+1}(V_{\lambda+1})\vDash \exists k:L_\alpha(V_{\lambda+1})\prec L_\alpha(V_{\lambda+1}), 
	\end{equation*}
	i.e., that $j\upharpoonright L_\alpha(V_{\lambda+1})\in L_{\alpha+1}(V_{\lambda+1})$. Can we assume that?
	
	The next step would be to define an inverse limit of embeddings at the $L_\alpha(V_{\lambda+1})$ level. Here another problem arises following the remarks after Theorem \ref{FirstStepI0}: the domain inverse limit of the $k_i:L_\alpha(V_{\lambda+1})\prec L_\alpha(V_{\lambda+1})$ is not some $L_{\bar{\alpha}}(V_{\bar{\lambda}+1})$, in fact it does not even contain $V_{\bar{\lambda}+1}$, therefore we should do the inverse limit of $k_i\upharpoonright V_\lambda$, and then extend it to some $L_{\bar{\alpha}}(V_{\bar{\lambda}+1})$. How to do it?
	
	Both problems were solved with the study of the structure of the sets $L_\alpha(V_{\lambda+1})$ for $\alpha<\Theta$ made by Laver in \cite{Laver4}.

\begin{defin}[Laver, \cite{Laver4}]
 Let $\lambda$ be a cardinal and let $\alpha<\Theta$. Then $\alpha$ is \emph{good} iff every element of $L_\alpha(V_{\lambda+1})$ is definable in $L_\alpha(V_{\lambda+1})$ from an element in $V_{\lambda+1}$.
\end{defin}

The successor of a good ordinal is a good ordinal: let $\alpha$ be good; the largest ordinal in $L_{\alpha+1}(V_{\lambda+1})$ is $\lambda+1+\alpha$, therefore $\alpha$ is definable in $L_{\alpha+1}(V_{\lambda+1})$ with $\lambda$ as a parameter, so $L_\alpha(V_{\lambda+1})$ is definable in $L_{\alpha+1}(V_{\lambda+1})$ with $\lambda$ as a parameter. But every element in $L_{\alpha+1}(V_{\lambda+1})$ is definable using $L_\alpha(V_{\lambda+1})$ and elements of $L_\alpha(V_{\lambda+1})$ (which in turn are definable with parameters from $V_{\lambda+1}$), and therefore $\alpha+1$ is good. This proves that the natural numbers are good. But also $\omega$ is good: every element of $L_\omega(V_{\lambda+1})$ is in some $L_n(V_{\lambda+1})$, $n\in V_{\lambda+1}$ and $n$ is good for every $n$. Following the same line of reasoning, every ordinal up to $\lambda$ is good, and considering that all ordinals less than $\lambda^+$ are coded as subsets of $\lambda$, and therefore in $V_{\lambda+1}$, every ordinal up to $\lambda^+$ is good.

On the other side, non-good ordinals exist. The definition of good ordinal is restricted to ordinals strictly less than $\Theta$ because larger ones are trivially not good: if $x$ is definable with a parameter, then it is uniquely determined by its definition, therefore for any $\alpha$ there exists in $L(V_{\lambda+1})$, 
\begin{multline*}
 \pi:V_{\lambda+1}\twoheadrightarrow G_\alpha=\{x\in L_\alpha(V_{\lambda+1}):\\
 L_\alpha(V_{\lambda+1})\vDash x\text{ is definable from an element in }V_{\lambda+1}\}; 
\end{multline*}
 then if $\Theta\subseteq L_\alpha(V_{\lambda+1})$, $G_\alpha$ must be strictly contained in $L_\alpha(V_{\lambda+1})$, by definition of $\Theta$. But non-good ordinals exist also below $\Theta$: define $L_\gamma(V_{\lambda+1})$ as the collapse of the Skolem closure of $V_{\lambda+1}$ in $L_{\Theta}(V_{\lambda+1})$; as $L_{\Theta}(V_{\lambda+1})\vDash\exists x\ x\text{ not definable from an element in }V_{\lambda+1}$, by elementarity the same must be true in $L_\gamma(V_{\lambda+1})$, as the collapse does not collapse $V_{\lambda+1}$, therefore $\gamma$ is not good. 

One can ask how many good ordinals there are.

\begin{lem}[Laver, \cite{Laver4}]
 Let $\lambda$ be a strong limit cardinal. Then the good ordinals are unbounded in $\Theta$.
\end{lem}

Therefore assuming $\alpha$ good is in most cases a reasonable choice.

By the usual condensation argument, if $\alpha<\Theta$ and $i:L_\alpha(V_{\lambda+1})\to L_\alpha(V_{\lambda+1})$ then $i\in L_{\Theta}(V_{\lambda+1})$. If $\alpha$ is good, however, it is possible to be much more precise:

\begin{lem}[Laver, Woodin, \cite{Kafkoulis},\cite{Laver4}]
\label{SimplyGood}
 Let $\lambda$ and $\alpha$ be such that $\alpha$ is good and there exists $i:L_\alpha(V_{\lambda+1})\prec L_\alpha(V_{\lambda+1})$ with $\crt(i)<\lambda$. Then $i$ is induced by $i\upharpoonright V_\lambda$, and therefore $i\in L_{\alpha+1}(V_{\lambda+1})$.
\end{lem}

Finally:
	
	\begin{lem}[Square root Lemma for $L_\alpha(V_{\lambda+1})$]
	\label{FinalChange}
	 Let $\alpha$ be good, and let $j:L_{\alpha+1}(V_{\lambda+1})\prec L_{\alpha+1}(V_{\lambda+1})$. Then for any $A,B\in L_\alpha(V_{\lambda+1})$, for any $\beta<\crt(j)$ there exists $k:L_\alpha(V_{\lambda+1})\prec L_\alpha(V_{\lambda+1})$ such that $\beta<\crt(k)<\crt(j)$, $k(k\upharpoonright V_\lambda)=j\upharpoonright V_\lambda$, $k(B)=j(B)$ and there exists $A\in\ran(k)$.
	\end{lem}
	\begin{proof}
	Exactly as Lemma \ref{change}. The key point is that as $\alpha$ is good, ``$\exists k\ k:L_\alpha(V_{\lambda+1})\prec L_\alpha(V_{\lambda+1})$'' is expressible in $L_{\alpha+1}(V_{\lambda+1})$ and it is satisfied by $j\upharpoonright L_\alpha(V_{\lambda+1})$, so it can be put inside the formula instead of ``$j$ is $\Sigma^1_{n-2}$''. 
	\end{proof}
	
	Note that if $\alpha$ is not good, then we can have a similar result: for $j:L_\beta(V_{\lambda+1})\prec L_\beta(V_{\lambda+1})$, let $\alpha$ be maximum such that $j\upharpoonright L_\alpha(V_{\lambda+1})\in L_\beta(V_{\lambda+1})$. Then we can find a square root that extend to $L_\alpha(V_{\lambda+1})$. The difference is that in this case the gap can be large, and it is somewhat arbitrary, while for the good case is just one.
	
	Moreover, I0 implies ``$\exists j:L_\alpha(V_{\lambda+1})\prec L_\alpha(V_{\lambda+1})$'':

\begin{lem}[Woodin, \cite{Kafkoulis}]
\label{Fixed}
 Let $\lambda$ be a cardinal. If there exists $j:L_{\Theta}(V_{\lambda+1})\prec L_{\Theta}(V_{\lambda+1})$, then for any $\alpha<\Theta$ there exists an $i:L_\alpha(V_{\lambda+1})\prec L_\alpha(V_{\lambda+1})$.
\end{lem}
\begin{proof}
 Suppose it is false. Then there is a counterexample $\alpha$ such that every $i:L_\alpha(V_{\lambda+1})\to L_\alpha(V_{\lambda+1})$ is not an elementary embedding. All such $i$'s are in $L_{\Theta}(V_{\lambda+1})$ (see remark before Lemma \ref{SimplyGood}), therefore $L_{\Theta}(V_{\lambda+1})\vDash\exists\alpha\ (\alpha\text{ is a counterexample})$. Let $\alpha_0$ be the least counterexample. Then $\alpha_0$ is definable in $L_{\Theta}(V_{\lambda+1})$ and $j(\alpha_0)=\alpha_0$. Then $j\upharpoonright L_{\alpha_0}(V_{\lambda+1})$ is as in the lemma, contradiction.
\end{proof}

Compare this lemma with the fact that $U_j\notin L(V_{\lambda+1})$, and therefore $j\upharpoonright L_\Theta(V_{\lambda+1})\notin L(V_{\lambda+1})$. Then if $j(\alpha)=\alpha$ and 
\begin{equation*}
 j\upharpoonright L_\alpha(V_{\lambda+1}):L_\alpha(V_{\lambda+1})\prec L_\alpha(V_{\lambda+1}), 
\end{equation*}
for any $\beta>\alpha$ there must be a $j_1$ elementary embedding with domain $L_\beta(V_{\lambda+1})$ and such that $j_1\neq j\upharpoonright L_\beta(V_{\lambda+1})$, otherwise one could build $j\upharpoonright L_\Theta(V_{\lambda+1})$ inside $L(V_{\lambda+1})$. In other words, the tree of embeddings with domain of the type $L_\alpha(V_{\lambda+1})$ has levels of ``cardinality'' $<\Theta$ (by \ref{inaccessible}), has height $\Theta$ and there is no cofinal branch in $L(V_{\lambda+1})$. In a certain sense, it is a $\Theta$-Aronszajn tree in $L(V_{\lambda+1})$. 

	Now we can do the inverse limit like in Theorem \ref{ChangeInverse}, and so we have $J:V_{\bar{\lambda}}\prec V_{\lambda}$, inverse limit of embeddings that can be extended to $L_\alpha(V_{\lambda+1})$. With $\bar{\alpha}$ fixed, we can extend $J$ to $L_{\bar{\alpha}}(V_{\bar{\lambda}+1})$. But how to find a $\bar{\alpha}$ so that $J$ is an elementary embedding?
	
	This point needs more work, and has been solved by Scott Cramer in his PhD Thesis. The idea is always to use \ref{FinalChange} repeatedly, to build inverse limits that have in the range witnesses for existential formulas, so that in the end the inverse limit is elementary. The difference is that we need $\lambda$ different witnesses to be considered every time, so the plan is to build a forcing that at the same time constructs a limit of inverse limits and collapses $\lambda$ to $\omega$. We follow with a sketch of the ideas behind such a construction, for the details we refer the reader to \cite{Cramer}.
	
	The structure of the proof of Theorem \ref{ChangeInverse} is like this:
		\newline
\newline
	\begin{tabular}{c c c c c c}
		& $j\upharpoonright V_\lambda$ & $j\upharpoonright V_\lambda$ & $j\upharpoonright V_\lambda$ & $j\upharpoonright V_\lambda$ & \dots\\
		& & & & & \\
		$K:$ & $k_0$ & $k_1$ & $k_2$ & $k_3$ & \dots\\
	\end{tabular}
		\newline
\newline
	
	where each embedding below is a square root of the embedding above. We have some freedom for the critical points in the embedding belows, in that for each one we can say that it must be bigger than a certain fixed ordinal, and (optionally) less than $\crt(j)$. In this way, we can push the limit of the critical points of the $k_i$'s to be $\lambda$, so that the inverse limit will be from $V_\lambda$ to itself, or less than $\crt(j)$, as in Lemma \ref{ChangeInverse}, so to have $J:V_{\bar{\lambda}}\prec V_\lambda$. So, by Lemma \ref{FinalChange}, the existence of $j:L_{\alpha+1}(V_{\lambda+1})\prec L_{\alpha+1}(V_{\lambda+1})$ implies the existence of an inverse limit $J$ of square roots of $j\upharpoonright V_\lambda$ such that $J:V_{\bar{\lambda}}\prec V_\lambda$ (possibly with $\bar{\lambda}=\lambda$), and all the square roots can be extended to $L_\alpha(V_{\lambda+1})$. 
	
	Instead of $j$, we can start already with an inverse limit:
	\newline
\newline
			\begin{tabular}{c c c c c c}
				$J:$ & $j_0$ & $j_1$ & $j_2$ & $j_3$ & \dots\\
				& & & & & \\
				$K:$ & $k_0$ & $k_1$ & $k_2$ & $k_3$ & \dots\\
			\end{tabular}
	\newline
\newline
	We say that $K$ (the inverse limit of the $k_i$'s) is a \emph{limit root} of $J$ (the inverse limit of the $j_i$'s) iff there exists an $n\in\omega$ such that $k_i=j_i$ for $i<n$ and $k_i$ is a square root of $j_i$ for $i\geq n$. So for any inverse limit $J$ of embeddings that extend to $L_{\alpha+1}(V_{\lambda+1})$ there exist a limit root of $J$ of embeddings that extend to $L_\alpha(V_{\lambda+1})$. Moreover, by \ref{FinalChange}, if $J:V_{\bar{\lambda}}\prec V_{\lambda}$, for any $B\in V_{\bar{\lambda}+1}$ and any $A\in V_{\lambda+1}$ we can have $K$ as before such that $K:V_{\bar{\lambda}}\prec V_\lambda$, $K(B)=J(B)$ and $A\in\ran(K)$.
	
	The next step is to build a limit of inverse limits. Let $j$ be an embedding that extends to $L_{\alpha+\omega}(V_{\lambda+1})$. Then consider something like this:
	\newline
\newline
	 \begin{tabular}{c c c c c c}
	         &  $j$  & $j$   & $j$   & $j$   & \dots\\
					& & & & & \\
			$J:$ & $j_0$ & $j_1$ & $j_2$ & $j_3$ & \dots\\
								& & & & & \\
			$K_1$ &	$j_0$ & $k^1_1$ & $k^1_2$ & $k^1_3$ & \dots\\
								& & & & & \\
			$K_2$ &	$j_0$ & $k^1_1$ & $k^1_2$ & $k^2_3$ & \dots\\
								& & & & & \\
			$\vdots$ & $\vdots$ & $\vdots$ & $\vdots$ & $\vdots$ & \dots\\
	 \end{tabular}
	\newline
\newline
	 where $J$ is an inverse limit of square roots of $j$, and every $K_{i+1}$ is a limit root of $K_i$. We can choose the $k^n_m$'s so that they all extend to $L_\alpha(V_{\lambda+1})$: Let $n_i$ be an increasing sequence of integers. Then choose $j_i$ to be a square root of $j$ extendible to $L_{\alpha+n_i}(V_{\lambda+1})$. Once defined $k^m_i$, if it is extendible to some $L_{\alpha+n}(V_{\lambda+1})$ with $n>0$, then choose $k^{m+1}_i$ to be a square root of $k^m_i$ that is extendible to $L_{\alpha+n-1}(V_{\lambda+1})$, otherwise $k^{m+1}_i=k^m_i$. As $n_i$ was increasing, the last case can happen only a finite amount of times, so $K_{m+1}$ will always be a limit square root of $K_m$. Also, at each stage, for any $A_m\subseteq V_{\bar{\lambda}}$ (where $V_{\bar{\lambda}})$ is the domain of $J$) and any $B_m\subseteq V_\lambda$, we can have $K^{m+1}(A)=K^m(A)$ and $B\in\ran(k^{m+1}_i\circ k^{m+1}_{i+1}\circ\dots)$, where $i$ is the minimum such that $k^m_i\neq k^{m+1}_i$. Note that this is less than a Square root Lemma for inverse limits for $L_\alpha(V_{\lambda+1})$, as we ask only $B$ to be in the range of a tail of the limit, and not of $K$. But this will be enough.
	
	\begin{teo}[Cramer, 2015 \cite{Cramer}]
	\label{therightalpha}
	 Let $\alpha$ be good, and let $j:L_{\alpha+\omega}(V_{\lambda+1})\prec L_{\alpha+\omega}(V_{\lambda+1})$. Let $J$ be as the first step of the construction above, with $V_{\bar{\lambda}}$ its domain. Then there exist $\bar{\alpha},\ \bar{\lambda}<\lambda$ such that $L_{\bar{\alpha}}(V_{\bar{\lambda}+1})\equiv L_\alpha(V_{\lambda+1})$.
	\end{teo}
	\begin{proof}
	 First of all, via a definable surjection $\rho:V_{\lambda+1}\twoheadrightarrow L_\alpha(V_{\lambda+1})$, we can think of any element of $L_\alpha(V_{\lambda+1})$ as a subset of $V_{\lambda}$. We work in $V[G]$, where $G$ is the generic of the collapse of $\lambda$ on $\omega$. Therefore $V_{\bar{\lambda}+1}$ is countable: let $\langle a_n:n\in\omega\rangle$ be an enumeration of it, and let $\langle\varphi_i:i\in\omega\rangle$ be an enumeration of all formulas.
	
	 We continue the construction of the limit of limit roots as above. The idea is that we gradually add to $A_m$ all the elements of $V_{\bar{\lambda}}$, each time adding just a finite amount, while $B_m$ will be the set of witnesses of existential formulas with parameters in $\rho''K_m''A_m$. 
	
	So suppose that we have defined $K_m$ and that there exists an $n_m$ such that $A_m=\langle a_i:i<n_m\rangle$. If $\varphi^m$ is not a formula with an existential quantifier, then $K_{m+1}=K_m$ and $A_{m+1}=A_m$. Otherwise suppose that 
	\begin{equation*}
	 L_\alpha(V_{\lambda+1})\vDash \exists x\ \psi(x,\rho(K_m(a_{i_0})),\dots\rho(K_m(a_{i_l}))), 
	\end{equation*}
	with $i_0,\dots, i_l<n_m$. We are following the construction of the limit of limit roots, so at this step there exists a $p$ such that $K^{m+1}$ should be a limit square root of $K_m$ such that $k^{m+1}_0=k^m_0$, $\dots$, $k^{m+1}_p=k^m_p$, $k^{m+1}_{p+1}\neq k^m_{p+1}$, etc... Call $(K_m)_p=k^m_p\circ k^m_{p+1}\circ\dots$. Then by elementarity (and the definability of $\rho$) 
	\begin{equation*}
	 L_\alpha(V_{\lambda+1})\vDash \exists x\ \psi(x,\rho((K_m)_p(a_{i_0})),\dots\rho((K_m)_p(a_{i_l}))). 
	\end{equation*}
	
	Let $b$ that witnesses this, and let $B_m$ be such that $\rho(B_m)=b$. Finally let $K_{m+1}$ be the limit root of $K_m$ such that $K_{m+1}(A_m)=K_m(A_m)$ and $B_m\in\ran((K_{m+1})_p)$. By the remarks above $K_{m+1}$ exists, and there exists $a\in V_{\bar{\lambda}+1}$ such that $\rho((K_{m+1})_p(a))=b$, so 
	\begin{equation*}
	 L_\alpha(V_{\lambda+1})\vDash\psi(\rho((K_{m+1})_p(a)),\rho((K_{m+1})_p(a_{i_0})),\dots\rho((K_{m+1})_p(a_{i_l})))
	\end{equation*}
	(note that $(K_{m+1})_p(a_i)=(K_m)_p(a_i)$). By elementarity 
	\begin{equation*}
	 L_\alpha(V_{\lambda+1})\vDash\psi(\rho((K_{m+1})(a)),\rho((K_{m+1})(a_{i_0})),\dots\rho((K_{m+1})(a_{i_l}))). 
	\end{equation*}

	Let $a=a_t$, and fix $A_{m+1}=\langle a_n:n\leq t\rangle$. 

 If $a=a_n\in V_{\bar{\lambda}}$, then for any $m$ such that $n_m>n$ $K_m(a)=K_{m+1}(a)$. Define $J_*:V_{\bar{\lambda}+1}\prec V_{\lambda+1}$ as $J_*(a_n)=K^m(a_n)$ with $n_m>n$. It is called the common part of the sequence $\langle K_i:i<\omega\rangle$. Let $M=\rho''J_*''V_{\bar{\lambda+1}}$. Then $M\prec L_\alpha(V_{\lambda+1})$: suppose that 
 \begin{equation*}
   L_\alpha(V_{\lambda+1})\vDash \exists x\ \psi(x,\rho(J_*(a_{i_0})),\dots\rho(J_*(a_{i_l}))), 
 \end{equation*}
 then there exists $m$ such that $n_m>n_{i,0},\dots,n_{i_l}$ and 
 \begin{equation*}
  L_\alpha(V_{\lambda+1})\vDash \exists x\ \psi(x,\rho(K_m(a_{i_0})),\dots\rho(K_m(a_{i_l}))). 
 \end{equation*}
  But we built $K_{m+1}$ so that 
	\begin{equation*}
	 L_\alpha(V_{\lambda+1})\vDash\psi(\rho((K_{m+1})(a)),\rho((K_{m+1})(a_{i_0})),\dots\rho((K_{m+1})(a_{i_l}))), 
	\end{equation*}
	and $K_{m+1}(a)=J_*(a)$. By condensation, the collapse of $M$ is a $L_{\bar{\alpha}}(V_{\bar{\lambda}+1})$ for some $\bar{\alpha}$. By absoluteness, this is true also in $V$.
	\end{proof}

Theorem \ref{therightalpha} gives us finally the right $\bar{\alpha}$ where to extend $J$, the inverse limit of embeddings that can be extended to $L_\alpha(V_{\lambda+1})$ in the most natural way: as $\bar{\alpha}$ is $(\bar{\lambda}-)$ good, then every element $b\in L_{\bar{\alpha}}(V_{\bar{\lambda}+1})$ is defined from some parameter $a\in V_{\bar{\lambda}+1}$, and $J(b)$ is the element of $L_\alpha(V_{\lambda+1})$ that is defined with the same formula and parameter $J(a)$. It remains to see that $J$ is actually an elementary embedding. Theorem \ref{therightalpha} gives us an elementary embedding, $J_*$, but it is in $V[G]$. But if we consider the partial order of all possible sequences of $K_i$'s, that start with $J$, coupled with the collapse, then $J_*$ will be the common part of the generic of such forcing. Considering $J$ as a sequence of length 1, for any $b\in L_{\bar{\alpha}}(V_{\bar{\lambda}+1})$ we can always extend $J$ so that $J^*(b)=J(b)$, and then elementarity comes from absoluteness. All the details are in \cite{Cramer}.

\begin{teo}[Cramer, 2015, \cite{Cramer}]
 If $j:L_{\alpha+\omega}(V_{\lambda+1})\prec L_{\alpha+\omega}(V_{\lambda+1})$, then there exist $\bar{\alpha}$, $\bar{\lambda}<\lambda$ and $k:L_{\bar{\alpha}}(V_{\bar{\lambda}})\prec L_{\bar{\alpha}}(V_{\bar{\lambda}})$. In particular, if $\alpha$ depends only on $\lambda$, then ``$\exists j:L_{\alpha+\omega}(V_{\lambda+1})\prec L_{\alpha+\omega}(V_{\lambda+1})$'' strongly implies ``$\exists j:L_\alpha(V_{\lambda+1})\prec L_\alpha(V_{\lambda+1})$''.
\end{teo}  

So all the following are strong implications:

\begin{multline*}
 I0(\lambda) \Rightarrow \dots \Rightarrow \exists j:L_{\alpha+\omega}(V_{\lambda+1})\prec L_{\alpha+\omega}(V_{\lambda+1}) \Rightarrow \exists j:L_\alpha(V_{\lambda+1})\prec L_\alpha(V_{\lambda+1}) \Rightarrow \\
\Rightarrow\dots \Rightarrow \exists j:L_\omega(V_{\lambda+1})\prec L_\omega(V_{\lambda+1}) \Rightarrow \exists j:L_1(V_{\lambda+1})\prec L_1(V_{\lambda+1} \Rightarrow\\
\Rightarrow I1(\lambda)\equiv\exists j:V_\lambda\prec V_\lambda\ \Sigma^1_\omega \Rightarrow \forall n\ \exists j:V_\lambda\prec V_\lambda \Sigma^1_n \Rightarrow \dots\\
\dots\Rightarrow \exists j:V_\lambda\prec V_\lambda\ \Sigma^1_n \equiv \exists j:V_\lambda\prec V_\lambda\ \Sigma^1_{n+1}\ \text{($n$ is odd)} \Rightarrow \dots \\
\dots\Rightarrow \exists j:V_\lambda\prec V_\lambda\ \Sigma^1_2 \equiv j:V_\lambda\prec V_\lambda \Sigma^1_1 \equiv I2(\lambda) \Rightarrow \text{iterable}\ I3(\lambda) \Rightarrow \dots\\
\dots\Rightarrow \omega\cdot\alpha\text{-iterable}\ I3(\lambda) \Rightarrow \alpha\text{-iterable}\ I3(\lambda) \Rightarrow I3(\lambda).
\end{multline*}

\section{Similarities with $\AD^{L(\mathbb{R})}$}
\label{simil}

We already noticed that for $I0(\lambda)$ to hold it is necessary that $L(V_{\lambda+1})\nvDash\AC$, just like for \AD{} to hold in $L(\mathbb{R})$ it is necessary that $L(\mathbb{R})\nvDash\AC$. This is the first sign of something bigger. The research is still ongoing on this respect, but there are many results that indicate a strong similarity between $L(V_{\lambda+1})$ under I0 and $L(\mathbb{R})$ under \AD. On the other hand, there are also striking differences, so there need to be further work on how much the two models are similar and why (and why they are not too similar).

The first similarity will be the following:
\begin{teo}[Solovay]
 \label{measinR}
 Suppose $L(\mathbb{R})\vDash\AD$. Then $L(\mathbb{R})\vDash\omega_1$ is measurable.
\end{teo}

\begin{lem}[Woodin, \cite{Kafkoulis}]
\label{measurable}
			Suppose that there exists $j:L(V_{\lambda+1})\prec L(V_{\lambda+1})$ with $\crt(j)<\lambda$. Then for every $\delta<\lambda$ regular, define
			\begin{equation*}
				S_\delta^{\lambda^+}=\{\eta<\lambda^+:\cof(\eta)=\delta\},
			\end{equation*}
			the stationary set of ordinals with the same cofinality and let ${\cal F}$ be the club filter on $\lambda^+$. Then there exists $\eta<\lambda$ and $\langle S_\alpha:\alpha<\eta\rangle\in L(V_{\lambda+1})$ a partition of $S_\delta^{\lambda^+}$ such that for every $\alpha<\eta$ ${\cal F}\upharpoonright S_\alpha$ is a $L(V_{\lambda+1})$-ultrafilter in $\lambda^+$. In particular $\lambda^+$ is measurable.
		\end{lem}
		\begin{proof}
			Define $\kappa_0=\crt(j)$. Suppose that there is a $\delta$ such that this is false, and pick the minimum one. Then $\delta$ is definable (using only $\lambda$ as a parameter, that is a fixed point), so $j(\delta)=\delta$, and this means that $\delta<\kappa_0$. Moreover, note that $j\upharpoonright{\lambda^+}$ is in $L(V_{\lambda+1})$, because the elements of $\lambda^+$ are well-orders of $\lambda$, so $j\upharpoonright{\lambda^+}$ depends on $j\upharpoonright V_{\lambda+1}$, that in turn is defined by $j\upharpoonright V_\lambda$, that is in $V_{\lambda+1}$.
			
			As in the proof of \ref{kunens}, there is no partition of $S_\delta^{\lambda^+}$ in $\kappa_0$ stationary sets. The difference is that here we need that $C=\{\alpha<\lambda^+:j(\alpha)=\alpha\}$ is a $\delta$-club. But it is: let $\vec{T}$ be a $\delta$-sequence of elements of $D$. Then $\vec{T}$ can be coded in $V_{\lambda+1}$, and since $j(\eta)=\eta$ for every $\eta\in\vec{t}$ and $j(\delta)=\delta$ we have $j(\vec{t})=\vec{t}$, so $j(\sup\vec{t})=\sup j(\vec{t})=\sup\vec{t}$.
			
			Let $I$ be the non-stationary ideal, i.e. the ideal of the non-stationary sets, and consider $\mathbb{B}={\cal P}(S_\delta^{\lambda^+})\slash I$. We prove that $\mathbb{B}$ is atomic, and its atoms are dense. Looking for a contradiction suppose that there exists a set $S\in I^+$ that contains no atom, so that 
			\begin{equation*}
				\forall T\subset S\ (T\in I^+\rightarrow(\exists T_0,T_1\in I^+\ (T_0\cup T_1=T\wedge T_0\cap T_1=\emptyset))).
			\end{equation*}
			We construct by induction a tree ${\cal T}$ of subsets of $S$ that refines reverse inclusion. Indicating $Lev_\gamma({\cal T})$ as the $\gamma$-th level of ${\cal T}$ we define $Lev_0({\cal T})=\{S\}$ and every $T\in{\cal T}$ has exactly two successors $T_0,T_1\in I^+$ such that $T_0\cup T_1=T$ and $T_0\cap T_1=\emptyset$, chosen with $\DC_\lambda$. At a limit $\alpha$, we put the intersections $T=\bigcap_{\beta<\alpha}T_\beta$, with $T_\beta\in Lev_\beta({\cal T})$, such that $T\in I^+$. We call $Res(\alpha)=S\setminus Lev_\alpha({\cal T})$ the set of the points in $S$ that don't belong to any $T$ in $Lev_\alpha({\cal T})$. Now, let $\alpha\leq\kappa_0$ be a limit ordinal, we want to prove by induction that $Res(\alpha)\in I$ and $Lev_\alpha({\cal T})\neq\emptyset$. Let $x\in S\setminus\bigcup_{\beta<\alpha}Res(\beta)$, then for every $\beta<\alpha$ there exists $T\in Lev_\beta({\cal T})$ such that $x\in T$, so we can define $b_x=\{Z\in\bigcup_{\beta<\alpha}Lev_\beta({\cal T}):x\in Z\}$ a branch of ${\cal T}$ of length $\alpha$ and $x\in\bigcap b_x$. So $S$ is the union of $\bigcup_{\beta<\alpha} Res(\beta)$, the points already skimmed by the process, and all of the $\bigcap b_x$. By induction every $Res(\beta)\in I$, and since $I$ is the dual of the club filter, $I$ is $\lambda^+$-complete by $\DC_\lambda$, so $\bigcup_{\beta<\alpha}Res(\beta)\in I$. We consider $\{\bigcap b_x:x\in S\setminus\bigcup_{\beta<\alpha}Res(\beta)\}$: it is a partition of $S\setminus \bigcup_{\beta<\alpha} Res(\beta)\in I^+$ and since $\lambda$ is a strong limit we have that there cannot be more than $2^{|\alpha|}<\lambda$ branches, so by $\lambda^+$-completeness there must be at least one set $T\in I^+$ such that $T=\bigcap b_x$ for some branch $b_x$, i.e. there must exist $T\in Lev_\alpha({\cal T})$. As for the branches such that $\bigcap b_x\notin I^+$, their union, again by completeness, must be in $I$, so 
			\begin{equation*}
				Res(\alpha)=\bigcup_{\beta<\alpha}Res(\beta)\cup\bigcup\{\bigcap b_x:\bigcap b_x\notin I^+\}\in I, 
			\end{equation*}
			and we can carry on the induction. If we consider a branch in ${\cal T}$ of length $\kappa_0$, say $\langle T_\alpha:\alpha<\kappa_0\rangle$, then 
			\begin{equation*}
				\{Res(\kappa_0)\cup(T_0\setminus T_1)\}\cup\{T_\beta\setminus T_{\beta+1}:\beta<\kappa_0\} 
			\end{equation*}
			is a partition of $S$ in $\kappa_0$ stationary sets, and we've just seen that this cannot be, so we've reached a contradiction.
			
			We've proved that the atoms of $\mathbb{B}={\cal P}(S_\delta^{\lambda^+})/I$ are dense, so for every $S\subseteq S_\delta^{\lambda^+}$ there exists $T\subseteq S$ atom for $\mathbb{B}$, in other words there exists $T\subseteq S$ in $I^+$ such that ${\cal F}\upharpoonright T$ is an ultrafilter. Since by $\DC_\lambda$ we have that ${\cal F}$ is $\lambda^+$-complete, ${\cal F}\upharpoonright T$ is a measure. By induction (using $\DC_\lambda$) we define $\langle S_\alpha:\alpha<\eta\rangle$:
			\begin{itemize}
				\item Let $S=S_\delta^{\lambda^+}$. Then there exists a $T\in I^+$ such that ${\cal F}\upharpoonright T$ is a measure, and we choose $S_0=T$.
				\item Let $\alpha$ ordinal, and suppose that for every $\beta<\alpha$, $S_\beta$ is defined. If $S_\delta^{\lambda^+}\setminus\bigcup_{\beta<\alpha}S_\beta\in I$, then we stop the sequence and $\eta=\alpha$. Otherwise there exists $S_\alpha\subseteq S_\delta^{\lambda^+}\setminus\bigcup_{\beta<\alpha}S_\beta$ such that $S_\alpha\in I^+$ is stationary and ${\cal F}\upharpoonright S_\alpha$ is a measure.
			\end{itemize}
   		In particular, ${\cal F}\upharpoonright S_0$ is a measure for $\lambda^+$. 		
		\end{proof}
		
		It is remarkable that the proof is completely different, yet in both theorems the measure is not just any measure, but it comes from the club filter. Woodin calls the cardinals that are measurable for this reason ``$\omega$-strongly measurable cardinals'', and they have a key role in the development of the inner model descriptive set theory.
		
		The fact that $\lambda^+$ is measurable has a lot of $\AD^{\mathbb{R}}$-like consequences:
		
		\begin{teo}[Shi, Trang, 2017 \cite{ShiTrang}]
		\label{firstShiTrang}
		 Suppose I0($\lambda$) holds. Then in $L(V_{\lambda+1})$:
		 \begin{itemize}
		  \item there is no $\lambda^+$-Aronszajn tree;
			\item $\Box_\lambda$ fails;
			\item there is no scale at $\lambda$ in $V_{\lambda+1}$;
			\item $\Diamond_{\lambda^+}$ fails;
			\item there is no $\lambda^+$-sequence of distinct members of ${\cal P}(\lambda)$.
		 \end{itemize}
		\end{teo}
		
		In fact Theorem \ref{measinR} was just the first $\omega$-strongly measurable cardinal found in $L(\mathbb{R})$, when there is a rich structure of them under \AD. The same is true in $L(V_{\lambda+1})$ under I0. Moschovakis' Coding Lemma is the way to find them:
		
		\begin{teo}[Coding Lemma, Moschovakis]
			Suppose $L(\mathbb{R})\vDash\AD$. Let $\eta<\Theta$ and $\rho:\mathbb{R}\twoheadrightarrow\eta$, $\rho\in L(\mathbb{R})$. Then there exists $\gamma_\rho<\Theta$ such that for every $A\subseteq \mathbb{R}\times \mathbb{R}$, $A\in L(\mathbb{R})$, there exists $B\subseteq \mathbb{R}\times\mathbb{R}$ such that:
			\begin{itemize}
				\item $B\subseteq A$;
				\item $B\in L_{\gamma_\rho}(\mathbb{R})$;
				\item for every $\alpha<\eta$ if $\exists(a,b)\in A\ \rho(a)=\alpha$ then $\exists(a,b)\in B\ \rho(a)=\alpha$.
			\end{itemize}
		\end{teo}
		
		It is a sort of very weak choice principle for subsets of $\mathbb{R}$: Consider $A_\xi=\{b:(a,b)\in A,\ \rho(a)=\xi\}$. Thus $A$ can be seen as an $\eta$-sequence of subsets of $\mathbb{R}$. If we had the Axiom of Choice, we could have a selector for such a sequence. Instead the Coding Lemma gives us a sequence $B_\xi\subseteq A_\xi$ that is in some $L_{\gamma_\rho}(\mathbb{R})$. But however ``complex'' (i.e., high in the constructive hierarchy) we choose $A$, $\gamma_\rho$ is the same, so the ``selector'' is actually simpler than $A$.
		
		We want to prove the Coding Lemma in $L(V_{\lambda+1})$. First we prove a weakening of the Coding Lemma, that will be used in the proof of the Coding Lemma itself.
		
		\begin{lem}[Weak Coding Lemma, Woodin, \cite{Kafkoulis}]
			Suppose that there exists $j:L(V_{\lambda+1})\prec L(V_{\lambda+1})$ with $\crt(j)<\lambda$. Let $\eta<\Theta$ and $\rho:V_{\lambda+1}\twoheadrightarrow\eta$, $\rho\in L(V_{\lambda+1})$. Then there exists $\gamma_\rho<\Theta$ such that for every $A\subseteq V_{\lambda+1}\times V_{\lambda+1}$, $A\in L(V_{\lambda+1})$, there exists $B\subseteq V_{\lambda+1}\times V_{\lambda+1}$ such that:
			\begin{itemize}
				\item $B\subseteq A$;
				\item $B\in L_{\gamma_\rho}(V_{\lambda+1})$;
				\item for cofinally $\alpha<\eta$ if $\exists(a,b)\in A\ \rho(a)=\alpha$ then $\exists(a,b)\in B\ \rho(a)=\alpha$.
			\end{itemize}
		\end{lem}
		\begin{proof}
			Since this lemma is quite rich in terms of quantifiers, we will use some abbreviation, calling $WCL$ the sentence of the Lemma, $WCL_\eta$ the same, but with fixed $\eta$, $WCL_{\eta,\rho}$ the same, but with fixed $\eta$ and $\rho$, and so on\dots
			
			Looking for a contradiction, suppose that $WCL$ is false and let $\eta$ be the least such that $\neg WCL_\eta$. Then $\eta$ is definable, so $j(\eta)=\eta$, and $\eta$ is a limit: in fact, $WCL_\eta\rightarrow WCL_{\eta+1}$, since for cofinally $\alpha<\eta+1$ means for $\alpha=\eta$, and in that case we can just choose an element in $\{(a,b)\in A:\rho(a)=\eta\}$.
			
			Let $\rho$ be such that $\neg WCL_{\eta,\rho}$. We define by induction, using $\DC_\lambda$, $\langle(\gamma_\xi,Z_\xi):\xi<\lambda\rangle$:
			\begin{itemize}
				\item Suppose that we have defined $(\gamma_\xi,Z_\xi)$. Since $\neg WCL_{\eta,\rho}$, for every $\gamma<\Theta$ $\neg WCL_{\eta,\rho,\gamma}$; so by $\neg WCL_{\eta,\rho,\gamma_\xi}$ there exists $Z_{\xi+1}\subseteq V_{\lambda+1}\times V_{\lambda+1}$ such that 
					\begin{itemize}
						\item $\{\rho(a):\exists b\ (a,b)\in Z_{\xi+1}\}$ is unbounded in $\eta$;
						\item for every $B\subseteq Z_{\xi+1}$, $B\in L_{\gamma_\xi}(V_{\lambda+1})$, $\{\rho(a):\exists b\ (a,b)\in B\}$ is bounded in $B$.
					\end{itemize}
					Let $\gamma_{\xi+1}$ be such that $Z_{\xi+1}\in L_{\gamma_{\xi+1}}(V_{\lambda+1})$ and $\gamma_\xi+\omega<\gamma_{\xi+1}$. By Lemma \ref{Thetasubsets} we can suppose that $\gamma_{\xi+1}<\Theta$.
				\item In the limit case, we choose $\gamma_\xi$ such that $\gamma_\xi>\gamma_\zeta+\omega$ for every $\zeta<\xi$ and a $Z_\xi\subseteq V_{\lambda+1}\times V_{\lambda+1}$ such that $Z_\xi\in L_{\gamma_\xi}(V_{\lambda+1})$.
			\end{itemize}
			
			So, for every $\xi_1<\xi_2<\lambda$ we have:
			\begin{itemize}
				\item $\gamma_{\xi_1}+\omega<\gamma_{\xi_2}$;
				\item $Z_{\xi_1}\in L_{\gamma_{\xi_1}}(V_{\lambda+1})$;
				\item $\{\rho(a):\exists b\ (a,b)\in Z_{\xi+1}\}$ is unbounded in $\eta$ and for every $B\subseteq Z_{\xi+1}$, $B\in L_{\gamma_\xi}(V_{\lambda+1})$, $\{\rho(a):\exists b\ (a,b)\in B\}$ is bounded in $B$.
			\end{itemize}
			
			Let $\rho_0=\rho$, $\rho_{n+1}=j(\rho_n)$,
			\begin{equation*}
				\langle(\gamma_\xi^0,Z_\xi^0):\xi<\lambda\rangle=\langle(\gamma_\xi,Z_\xi):\xi<\lambda\rangle
			\end{equation*}
			and 
			\begin{equation*}
				\langle(\gamma_\xi^{n+1},Z_\xi^{n+1}):\xi<\lambda\rangle=j(\langle(\gamma_\xi^n,Z_\xi^n):\xi<\lambda\rangle). 
			\end{equation*}
			Let $n$ be the minimum such that $\gamma_{\kappa_0+1}^n\leq\gamma_{\kappa_0+1}^{n+1}$ (there must exist, otherwise $\langle\gamma_{\kappa_0}^n:n\in\omega\rangle$ would be a descendent chain of ordinals). By elementarity we have that:
			\begin{itemize}
				\item $Z_{\kappa_0+1}^n\in L_{\gamma_{\kappa_0+1}^n}(V_{\lambda+1})$;
				\item $\{\rho_{n+1}(a):\exists b\ (a,b)\in Z_{\kappa_1+1}^{n+1}\}$ is unbounded in $\eta$ and for every $B\subseteq Z_{\kappa_1+1}^{n+1}$, $B\in L_{\gamma_{\kappa_1}^{n+1}}(V_{\lambda+1})$ we have that $\{\rho_{n+1}(a):\exists b\ (a,b)\in B\}$ is bounded in $\eta$;
				\item $j(Z_{\kappa_0+1}^n)=Z_{\kappa_1+1}^{n+1}$.
			\end{itemize}
			
			Let $B=\{(j(a),j(b)):(a,b)\in Z_{\kappa_0+1}^n\}$. The parameters used in the definition of $B$ are $j\upharpoonright V_{\lambda+1}$ (that in turn is defined by $j\upharpoonright V_\lambda$) and $Z_{\kappa_0+1}^n$, so $B\in L_{\gamma_{\kappa_0+1}^n+1}(V_{\lambda+1})\subseteq L_{\gamma_{\kappa_1}^n}(V_{\lambda+1})$. If $(a,b)\in Z_{\kappa_0+1}^n$, then 
			\begin{equation*}
				j((a,b))=(j(a),j(b))\in j(Z_{\kappa_0+1}^n)=Z_{\kappa_1+1}^{n+1}, 
			\end{equation*}
			so $B\subseteq Z_{\kappa_1+1}^{n+1}$. Finally, for every $\alpha<\eta$ there exists $(a,b)\in Z_{\kappa_0+1}^n$ such that $\alpha<\rho_n(a)$, so $\alpha\leq j(\alpha)<\rho_{n+1}(j(a))$, and $\{\rho_{n+1}(a):\exists b\ (a,b)\in B\}$ is unbounded in $\eta$. Contradiction.			
		\end{proof}
		
		\begin{teo}[Coding Lemma, Woodin, \cite{Kafkoulis}]
			Suppose that there exists $j:L(V_{\lambda+1})\prec L(V_{\lambda+1})$ with $\crt(j)<\lambda$. Let $\eta<\Theta$ and $\rho:V_{\lambda+1}\twoheadrightarrow\eta$, $\rho\in L(V_{\lambda+1})$. Then there exists $\gamma_\rho<\Theta$ such that for every $A\subseteq V_{\lambda+1}\times V_{\lambda+1}$, $A\in L(V_{\lambda+1})$, there exists $B\subseteq V_{\lambda+1}\times V_{\lambda+1}$ such that:
			\begin{itemize}
				\item $B\subseteq A$;
				\item $B\in L_{\gamma_\rho}(V_{\lambda+1})$;
				\item for every $\alpha<\eta$ if $\exists(a,b)\in A\ \rho(a)=\alpha$ then $\exists(a,b)\in B\ \rho(a)=\alpha$.
			\end{itemize}
		\end{teo}
		\begin{proof}
			The proof is by induction on $\eta$. Let $\rho:V_{\lambda+1}\twoheadrightarrow\eta$. For every $\alpha<\eta$ we define $\rho_\alpha:V_{\lambda+1}\twoheadrightarrow\alpha$:
			\begin{equation*}
				\rho_\alpha(a)=\begin{cases}
					\rho(a)   & \text{if $\rho(a)<\alpha$};\\
					0         & \text{otherwise}.
					\end{cases}
			\end{equation*}
			By induction, for every $\alpha<\eta$ there exists $\gamma_{\rho_\alpha}<\Theta$ that satisfies the Coding Lemma for $\eta, \rho_\alpha$. Let $\beta_0=\sup\{\gamma_{\rho_\alpha}:\alpha<\eta\}$. Since $\Theta$ is regular, we have that $\beta_0<\Theta$. Let $\gamma_\rho$ be the ordinal that witnesses the Weak Coding Lemma for $\eta,\rho$, and let $\beta_1=\sup\{\beta_0,\gamma_\rho\}$. Since $\beta_1<\Theta$, there exists $\pi:V_{\lambda+1}\twoheadrightarrow L_{\beta_1}(V_{\lambda+1})$, and this $\pi$ can be codified as a subset of $V_{\lambda+1}$. We call $\beta$ the ordinal $<\Theta$ such that the code of $\pi$ is in $L_\beta(V_{\lambda+1})$. We prove that $\beta+1$ witnesses the Coding Lemma for $\eta,\rho$.
			
			Fix $A$ and define $A_{<\alpha}=\{(a,b)\in A:\rho(a)<\alpha\}$. We can suppose that for every $\alpha<\eta$ there exists $(a,b)\in A$ such that $\rho(a)=\alpha$. We want to code the set $\{(A_{<\alpha},B): \alpha<\eta,\ B\text{ witnesses the Coding Lemma for }\eta,\rho_\alpha,A_{<\alpha}\}$ as a subset of $V_{\lambda+1}\times V_{\lambda+1}$:
			\begin{multline*}
				A^*=\{(a,b)\in V_{\lambda+1}\times V_{\lambda+1}:\rho(a)>0,\pi(b)\subseteq A_{<\rho(a)},\\
				\forall\xi<\rho(a)\ \exists(x,y)\in\pi(b)\ \rho(x)=\xi\}.
			\end{multline*}
			Since $\beta$ witnesses the Coding Lemma for all the $\rho_\alpha$, we have that $\{\rho(a):\exists b\ (a,b)\in A^*\}=\eta$, and since the Weak Coding Lemma holds for $\eta,\rho,\beta$, there exists $B^*\subseteq A^*$ such that $B^*\in L_\beta(V_{\lambda+1})$ and $\{\rho(a):\exists b\ (a,b)\in B^*\}$ is cofinal in $\eta$. 
			
			Let $B=\bigcup\{\pi(b):(a,b)\in B^*\}$. Then $B\subseteq A$ and $B\in L_{\beta+1}(V_{\lambda+1})$. Moreover, for every $\alpha<\eta$, there exists $(a,b)\in B^*$ such that $\rho(a)>\alpha$, $\pi(b)\in A_{<\rho(a)}$ and there exists $(x,y)\in\pi(b)$ such that $\rho(x)=\alpha$. But then $(x,y)\in B$, and $B$ witnesses the Coding Lemma for $\eta,\rho,\beta+1,A$. 
		\end{proof}
		
		The Coding Lemma is used to prove a sort of inaccessibility of $\Theta$ in $L(V_{\lambda+1})$: as $L(V_{\lambda+1})\nvDash\AC$, we cannot say $2^\gamma<\Theta$ for all $\gamma<\Theta$; but we have the following:
		
		\begin{lem}[Woodin, \cite{Kafkoulis}]
		\label{inaccessible}
			Suppose that there exists $j:L(V_{\lambda+1})\prec L(V_{\lambda+1})$ with $\crt(j)<\lambda$. Then in $L(V_{\lambda+1})$ for every $\alpha<\Theta$ there exists a surjection $\pi:V_{\lambda+1}\twoheadrightarrow{\cal P}(\alpha)$.
		\end{lem}
		\begin{proof}
			Fix $\alpha<\Theta$, $\rho:V_{\lambda+1}\twoheadrightarrow\alpha$, and let $\gamma$ be the maximum between the least $\beta<\Theta$ such that $\rho\in L_\beta(V_{\lambda+1})$ and the witness for the Coding Lemma for $\alpha,\rho$. For every $A\subseteq\alpha$ define $A^*=\{(a,0)\in V_{\lambda+1}:\rho(a)\in A\}$. Then $A^*\in L(V_{\lambda+1})$. So there exists $B^*\subseteq A^*$, $B^*\in L_\gamma(V_{\lambda+1})$ such that $\{\rho(a):\exists b\ (a,b)\in B^*\}=A$. But then $A\in L_{\gamma+1}(V_{\lambda+1})$. This means that ${\cal P}(\alpha)\subseteq L_{\gamma+1}(V_{\lambda+1})$, and using $\pi:V_{\lambda+1}\twoheadrightarrow\gamma+1$ we're done.
		\end{proof}
		
		The following is the strongest theorem related to the Coding Lemma that has been proved until the writing of this paper, and it is akin to a similar one in $\AD^{L(\mathbb{R})}$ (cfr. with \cite{Moschovakis}):
		
		\begin{teo}[Woodin, \cite{Kafkoulis}]
		\label{manymeas}
			Suppose that there exists $j:L(V_{\lambda+1})\prec L(V_{\lambda+1})$ with $\crt(j)=\kappa_0<\lambda$. Then $\Theta$ is a limit of $\gamma$ such that:
			\begin{enumerate}
				\item $\gamma$ is weakly inaccessible in $L(V_{\lambda+1})$;
				\item $\gamma=\Theta^{L_\gamma(V_{\lambda+1})}$ and $j(\gamma)=\gamma$;
				\item for all $\beta<\gamma$, ${\cal P}(\beta)\cap L(V_{\lambda+1})\in L_\gamma(V_{\lambda+1})$;
				\item for cofinally $\kappa<\gamma$, $\kappa$ is an $\omega$-strongly measurable cardinal in $L(V_{\lambda+1})$;
				\item $L_\gamma(V_{\lambda+1})\prec L_\Theta(V_{\lambda+1})$.
			\end{enumerate}
		\end{teo}
		\begin{proof}
 Fix $\alpha<\Theta$, $j(\alpha)=\alpha$ and let $\gamma<\Theta$ be the minimum such that $L_\gamma(V_{\lambda+1})\prec_1 L_\Theta(V_{\lambda+1})$. As also $j(\Theta)=\Theta$, we have $j(\gamma)=\gamma$. This will be the $\gamma$ as in the Theorem. By Corollary \ref{FixedpointsbelowTheta} we can have $\gamma$ as large as we want below $\Theta$.

It is tempting to stop here, and say that $\gamma$ has the same properties of $\Theta$ by elementarity, but we do not have this directly, as $\Theta\notin L_\Theta(V_{\lambda+1})$.

Note that, by \ref{finestructure} relativized to $\Sigma_1(\alpha)$ formulas, $\mathbf{\Delta}^2_1(\alpha)$ sets, etc\dots, $\Theta^{L_\gamma(V_{\lambda+1})}=\gamma$, and therefore (2). Also, $\cof(\delta)>\lambda$: Otherwise let $\langle\beta_\xi:\xi<\mu\rangle$ be a sequence cofinal in $\delta$ with $\mu<\lambda$. Then any $\beta_\xi$ is the length of a $\mathbf{\Delta}^2_1(\alpha)$ pwo. With $\DC_\lambda$ choose one pwo for any $\xi$, and then gluing this pwos together it is possible to find a $\mathbf{\Delta}^2_1(\alpha)$ pwo of $\gamma$, and this contradicts \ref{finestructure}. Now the Weak Coding Lemma can be re-proved in $L_\gamma(V_{\lambda+1})$: the key points are that since $\Theta^{L_{\gamma}(V_{\lambda+1})}=\gamma$, $\gamma_{\xi+1}<\gamma$ and since $\cof(\gamma)>\lambda$ the sequence of the $Z_\xi$ is longer than $\lambda$, and we can find the contradiction.

With the Weak Coding Lemma we prove that $\gamma$ is regular. It is similar to the proof that its cofinality is larger than $\lambda$, using the Weak Coding Lemma instead of $\DC_\lambda$: Let $\beta<\gamma$ and $f:\beta\to\gamma$ cofinal in $\gamma$. Let $\pi$ be a $\Sigma_1(\alpha)$ partial bounded surjection from $V_{\lambda+1}$ to $L_\gamma(V_{\lambda+1})$, as in Lemma \ref{finestructure}. As $\Theta^{L_\gamma(V_{\lambda+1})}=\gamma$, there exists $\rho:V_{\lambda+1}\twoheadrightarrow\beta$ with $\rho\in L_\gamma(V_{\lambda+1})$. Let $\delta$ witness the Weak Coding Lemma for $\rho$ in $L_\gamma(V_{\lambda+1})$, and consider 
\begin{equation*}
  A=\{(a,b):b\in\dom(\pi)\wedge\pi(b)=f(\rho(a))\}. 
\end{equation*}

Then there exists $B\subseteq A$, $B\in L_\delta(V_{\lambda+1})$ such that for cofinally $\xi<\beta$ if there exists $(a,b)\in A$ $\rho(a)=\xi$ then there exists $(a,b)\in B$ such that $\rho(a)=\xi$. As for every $\xi$, $A_\xi\neq\emptyset$, this means that for cofinally $\xi$, $B_\xi\neq\emptyset$. Let $B^0=\{a:\exists b\ (a,b)\in B\}$. Then $\pi''B^0$ is cofinal in $\beta$, so $(f\circ\pi)''B^0$ is cofinal in $\gamma$. Let $Z=\{b:\exists a\in V_{\lambda+1}\ (a,b)\in B\}$. As $B\subseteq A$, $\rho''Z=(f\circ\pi)''B^0$, but that contradicts the boundedness of $\rho$.

Since $\gamma$ is regular, then the Coding Lemma can be re-proved relativized to $L_\gamma(V_{\lambda+1})$, so also Lemma \ref{inaccessible} holds in $L_\gamma(V_{\lambda+1})$, therefore (3) is proved. So $\gamma$ is weakly inaccessible because of regularity and (3). 

Let $S_\omega^\gamma=\{\alpha<\gamma:\cof(\alpha)=\omega\}$. Let ${\cal F}$ be the club filter on $S_\omega^\gamma$. Instead of proving that ${\cal F}$ is an ultrafilter on a stationary set, we define another filter ${\cal F}_0\subseteq {\cal F}$, that is an ultrafilter on a stationary set, and that coincides with ${\cal F}$ there.

Let ${\cal E}=\{k:L_\gamma(V_{\lambda+1})\prec L_\gamma(V_{\lambda+1})\}$. For each $k\in{\cal E}$, let $C_k=\{\eta\in S_\omega^\gamma:k(\eta)=\eta\}$. As $\gamma$ is regular, $C_k$ is an $\omega$-club. For each $\sigma\subseteq{\cal E}$ such that $|\sigma|\leq\lambda$, let $C_\sigma=\cap\{C_k:k\in\sigma\}$. Let ${\cal F}_0$ be the filter on $S_\omega^\gamma$ generated by the sets $C_\sigma\cap S_\omega^\gamma$. It is $\lambda^+$-complete by $\DC_\lambda$. All its elements are clubs on $S_\omega^\gamma$, so ${\cal F}_0\subseteq{\cal F}$. Also, $j({\cal F}_0)={\cal F}_0$. 

As in \ref{measurable} $S$ cannot be partitioned in $\kappa_0$ many ${\cal F}_0$-positive sets, since $j\upharpoonright L_\gamma(V_{\lambda+1})\in{\cal E}$ and therefore $C_j\in{\cal F}_0$. So there must exists a $T\subseteq S$ stationary such that ${\cal F}_0\upharpoonright T$ is an ultrafilter. As it is a maximal filter on $T$, it must be that ${\cal F}\upharpoonright T={\cal F}_0\upharpoonright T$. It remains to show that ${\cal F}$ on $T$ is $\gamma$-complete.

Let $\beta<\gamma$ and $\langle A_\xi:\xi<\beta\rangle$ a sequence of subsets of $T$ of ${\cal F}$-measure one. Let $\rho:V_{\lambda+1}\twoheadrightarrow\beta$, with $\rho\in L_\gamma(V_{\lambda+1})$. Let $\eta$ to witness the Coding Lemma in $L_\gamma(V_{\lambda+1})$ for $\rho$. Ideally, we would like to use the Coding Lemma on 
\begin{equation*}
 A=\{(A_\xi,C_\sigma):\sigma\subseteq{\cal E},\ |\sigma|\leq\lambda,\ \xi<\beta,\ C_\sigma\subseteq A_\xi\}, 
\end{equation*}
to ``choose'' some $C_\sigma$ inside every $A_\xi$, so that we can intersect them in the following way: Suppose we can find $B\subseteq A$ ``small'' enough, so that, if 
\begin{equation*}
 {\cal E}_B=\{k\in{\cal E}:\exists\xi\ \exists\sigma\ \xi<\beta,\ \sigma\subseteq{\cal E},\ |\sigma|\leq\lambda,\ (A_\xi,C_\sigma)\in B\, k\in\sigma\}, 
\end{equation*}
for any $\xi<\gamma$, $\sup\{k(\xi):k\in{\cal E}_B\}<\gamma$. Then pick $\xi_0>\alpha$, $\xi_{n+1}=\sup\{k(\xi_n):k\in{\cal E}_B\}$. Then $\xi_\omega=\sup_{n\in\omega}\xi_n\in T$ and for any $k\in{\cal E}_B$ $k(\xi_\omega)=\xi_\omega$. This means that for any $A_\xi$, $\xi_\omega\in A_\xi$, so we found an intersection. 

Therefore we have to code $A$ as a subset of $V_{\lambda+1}\times V_{\lambda+1}$: Note that $\gamma$ is good, therefore by \ref{SimplyGood} for any $k_1,k_2\in{\cal E}$, $k_1=k_2$ iff $k_1\upharpoonright V_\lambda=k_2\upharpoonright V_\lambda$, so there is a bijection between ${\cal E}$ and ${\cal E}_0=\{k\upharpoonright V_\lambda:k\in{\cal E}\}\subseteq V_{\lambda+1}$. Moreover let $F$ be any bijection between $V_{\lambda+1}$ and $^\lambda V_{\lambda+1}$: for $a\in V_{\lambda+1}$ we can define $C_a=C_{\sigma(a)}$ with $\sigma(a)=\{k:k\upharpoonright V_\lambda\in F(a)\}$.  

Let 
\begin{equation*}
 A=\{(a,b)\in V_{\lambda+1}\times V_{\lambda+1}:F(b)\subseteq{\cal E}_0,\ C_b\subseteq A_{\rho(a)}\}. 
\end{equation*}
Let $B\subseteq A$, $B\in L_\eta(V_{\lambda+1})$ given by the Coding Lemma. Let 
\begin{equation*}
 {\cal E}_B=\cup\{F(b):\exists a\ (a,b)\in B\}\in L_{\eta+1}(V_{\lambda+1}), 
\end{equation*}
a set of codes for embeddings $k$ such that $C_k\subseteq A_\xi$ for some $\xi$. Let $a\in\dom(\pi)$. Note that, as $\pi$ is $\Sigma_1(\alpha)$-definable, if $a\in\dom(\pi)$ then $k(a)\in\dom(\pi)$ for any $k$ embedding, so 
\begin{equation*}
 W_a=\{k(a):k\in{\cal E},\ k\upharpoonright V_\lambda\in{\cal E}_B\}\subseteq\dom(\pi), 
\end{equation*}
and $W_a\in L_{\eta+2}(V_{\lambda+1})$. So 
\begin{equation*}
 \pi''W_a=\{\pi(k(a)):k\in{\cal E},\ k\upharpoonright V_\lambda\in{\cal E}_B\}=\{k(\pi(a)):k\in{\cal E},\ k\upharpoonright V_\lambda\in{\cal E}_B\}. 
\end{equation*}
Let $\xi<\gamma$, and $a\in V_{\lambda+1}$ such that $\pi(a)=\xi$. So $\pi''W_a=\{k(\xi):k\in{\cal E},\ k\upharpoonright V_\lambda\in{\cal E}_B\}$. As $W_a$ is bounded, then $\pi''W_a$ is bounded, so $\sup\{k(\xi):k\in{\cal E},\ k\upharpoonright V_\lambda\in{\cal E}_B\}<\gamma$. Now we can find a fixed point for all the extension of embeddings in ${\cal E}_B$, and we are done.

For the point (5) we just sketch the proof. The idea is to consider $P$, the theory of $L_\Theta(V_{\lambda+1})$, and relativize the previous proof for $L_\gamma(V_{\lambda+1})[P]$. Then 
\begin{equation*}
 (L_\gamma(V_{\lambda+1})[P],P\cap L_\gamma(V_{\lambda+1}))\prec_1 (L_\Theta(V_{\lambda+1}),P)
\end{equation*}
implies full elementarity. 
		\end{proof}
		
There are still open questions about the combinatorial nature of cardinals between $\lambda$ and $\Theta$, following the similarities with $\AD^{L(\mathbb{R})}$:

\begin{Q}
 Suppose I0($\lambda$), are $\lambda^{++}$, $\lambda^{+++}$ regular or singular in $L(V_{\lambda+1})$?
\end{Q}

\begin{Q}
 Suppose I0($\lambda$). Is it true that in $L(V_{\lambda+1})$ every regular cardinal below $\Theta$ is measurable?
\end{Q}
		
Another point in favour of \AD{} is the regularity properties that it entails, and this is partially true also for I0.

We consider on $V_{\lambda+2}$ the topology with basic open sets $O_{(\alpha,a)}=\{x\subseteq V_\lambda:x\cap V_\alpha=a\}$, with $\alpha<\lambda$ and $a\subseteq V_\alpha$. Then a set $X\subseteq V_{\lambda+1}$ is perfect iff $\Pi_{n\in\omega}\kappa_n$ can be continuously embedded into $X$. A set has the perfect set property iff it has cardinality $\leq\lambda$ or it contains a perfect subset.

\begin{teo}[Cramer, 2015 \cite{Cramer}]
 \label{allperfect}
 Suppose there exists $j:L(V_{\lambda+1})\prec L(V_{\lambda+1})$, with $\crt(j)<\lambda$. Then every set in $L(V_{\lambda+1})\cap V_{\lambda+2}$ has the perfect set property in $L(V_{\lambda+1})$.
\end{teo} 		

The proof uses heavily all the inverse limit tools developed in Section \ref{strimp}. It exploits the fact that there is a multiplicity of possible square roots: Consider $J:L_{\bar{\alpha}+1}(V_{\bar{\lambda}+1})\prec L_{\alpha+1}(V_{\lambda+1})$, inverse limit, and let $X\in L_{\alpha+1}(V_{\lambda+1})\cap V_{\lambda+2}$ and $\bar{X}\in L_{\bar{\alpha}+1}(V_{\bar{\lambda}+1})$ such that $J(\bar{X})=X$. Then if $|X|>\lambda$ we can find for any $a\in\bar{X}$ many ($>\lambda$) limit roots of $J$, all with different images on $a$. Repeated uses of this kind of pigeonhole principle will be used to construct a perfect set inside $X$.

\begin{Q}
 Is there an equivalent in $L(V_{\lambda+1})$ for Baire property or Lebesgue measurability?
\end{Q}

Some key properties of subsets of $\mathbb{R}$ are proven thanks to the fact that under $\AD$ all such sets have a nice tree-structure, namely they are weakly homogeneously Suslin. Woodin introduced a definition that wants to develop a similarity with this concept: $\mathbb{U}(j)$-representability.

If $x\in V_{\lambda+1}$ and $\langle\lambda_i:i\in\omega\rangle$ is cofinal in $\lambda$, then we can see $x$ as the $\omega$-sequence $\langle x\cap V_{\lambda_i}:i\in\omega\rangle$, just like any real can be seen as an $\omega$-sequence of natural numbers. So any set $X\subseteq V_{\lambda+1}$ can be seen as a set of $\omega$-sequences. In $\mathbb{R}$, homogeneous Suslin-ness consists in labeling any finite sequence of natural numbers with an ultrafilter, so that if a sequence labeled with $U_2$ extends another labeled with $U_1$, then $U_2$ projects to $U_1$. If $U_2$ projects to $U_1$, then there is a related elementary embedding from $\Ult(V,U_1)$, to $\Ult(V,U_2)$, therefore for any infinite sequence one can define the direct limit of all the ultrapowers of the ultrafilters that label its finite initial segments. Then a set $Z$ is homogeneously Suslin iff there is such labeling so that a branch is in $Z$ iff such direct limit is well-founded. 
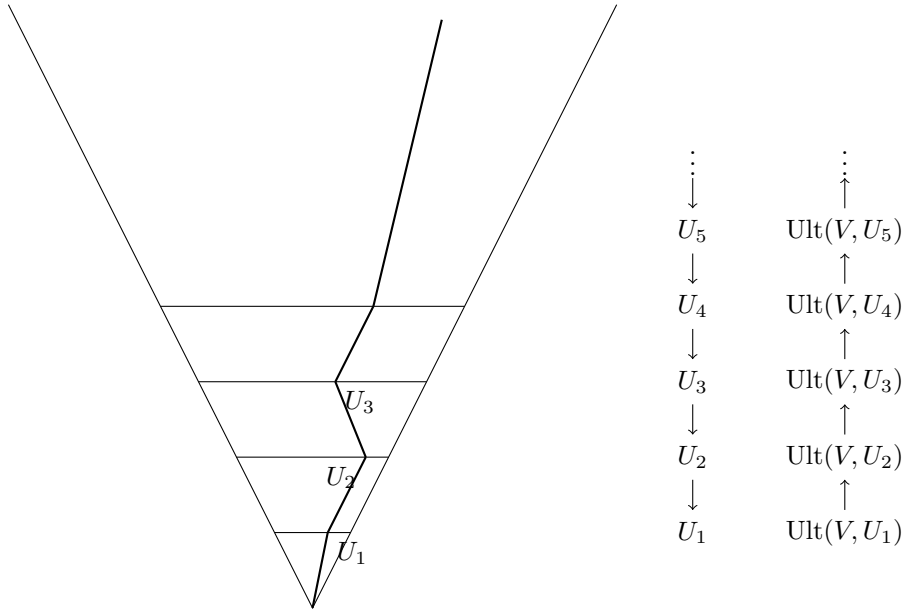
\newline
\newline
\begin{tikzpicture}
 \draw (0,8) -- (4,0) -- (8,8);
 \draw (3.5,1) -- (4.5,1);
 \draw (3,2) -- (5,2);
 \draw (2.5,3) -- (5.5,3);
 \draw (2,4) -- (6,4);
 \draw [thick] (4,0) -- (4.2,1);
 \draw [thick] (4.2,1) -- (4.7,2);
 \draw [thick] (4.7,2) -- (4.3,3);
 \draw [thick] (4.3,3) -- (4.8,4);
 \draw [thick] (4.8,4) -- (5.7,7.8);
 \draw (4.2,1) node[anchor=north west] {$U_1$};
 \draw (4.7,2) node[anchor=north east] {$U_2$};
 \draw (4.3,3) node[anchor=north west] {$U_3$};
 \draw (9,1) node {$U_1$};
 \draw (9,2) node {$U_2$};
 \draw (9,3) node {$U_3$};
 \draw (9,4) node {$U_4$};
 \draw (9,5) node {$U_5$};
 \draw (9,6) node {$\vdots$};
 \draw [->] (9,1.7) -- (9,1.3);
 \draw [->] (9,2.7) -- (9,2.3);
 \draw [->] (9,3.7) -- (9,3.3);
 \draw [->] (9,4.7) -- (9,4.3);
 \draw [->] (9,5.7) -- (9,5.3);
 \draw (11,1) node {$\Ult(V,U_1)$};
 \draw (11,2) node {$\Ult(V,U_2)$};
 \draw (11,3) node {$\Ult(V,U_3)$};
 \draw (11,4) node {$\Ult(V,U_4)$};
 \draw (11,5) node {$\Ult(V,U_5)$};
 \draw (11,6) node {$\vdots$};
 \draw [<-] (11,1.7) -- (11,1.3);
 \draw [<-] (11,2.7) -- (11,2.3);
 \draw [<-] (11,3.7) -- (11,3.3);
 \draw [<-] (11,4.7) -- (11,4.3);
 \draw [<-] (11,5.7) -- (11,5.3);
\end{tikzpicture}
\newline
\newline
Weakly homogeneous Suslin-ness adds another dimension, considering instead the couples of finite sequences.

So the first step is to decide which ultrafilters are going to label our sequences:  

\begin{defin}
 Let $\mathbb{U}(j)$ be the set of $U\in L(V_{\lambda+1})$ such that in $L(V_{\lambda+1})$ the following hold:
\begin{itemize}
 \item $U$ is a $\lambda^+$-complete ultrafilter;
 \item for some $\gamma<\Theta$, $U\subseteq{\cal P}(L_\gamma(V_{\lambda+1}))$;
 \item for some $A\in U$ and all sufficiently large $n\in\omega$, $j^n(U)=U$ and $\{a\in A:j^n(a)=a\}\in U$.
\end{itemize}

Let $\kappa<\Theta$, $\kappa\leq\Theta^{L_\kappa(V_{\lambda+1})}$. Let $\langle a_i:i<\omega\rangle$ be such that for all $i\in\omega$ there exists $n\in\omega$ such that $j^n(a_i)=a_i$. Then $\mathbb{U}(j,\kappa,\langle a_i:i\in\omega\rangle)$ is the set of $U\in\mathbb{U}(j)$ such that there exists an $n\in\omega$ such that for all $k:L_\kappa(V_{\lambda+1})\prec L_\kappa(V_{\lambda+1})$ with the property that there are $m,l\in\omega$, $k^m=j^l\upharpoonright L_\kappa(V_{\lambda+1})$, if for all $i\leq n$ $k(a_i)=a_i$, then $\{a\in L_\kappa(V_{\lambda+1}):k(a)=a\}\in U$.
\end{defin}

Now we can introduce $\mathbb{U}(j)$ representability:

\begin{defin}[Woodin, 2011 \cite{Woodin}]
 Let $Z\in L(V_{\lambda+1})\cap V_{\lambda+2}$. Then $Z$ is $\mathbb{U}(j)$-representable if there are $\kappa$, $\langle a_i:i\in\omega\rangle$, an increasing sequence $\langle\lambda_i:i\in\omega\rangle$, and a function $\pi:\bigcup\{V_{\lambda_i+1}\times V_{\lambda_i+1}\times\{i\}:i\in\omega\}\to\mathbb{U}(j,\kappa,\langle a_i:i\in\omega\rangle)$ such that the following hold:
\begin{itemize}
 \item for all $i\in\omega$ and $(a,b,i)\in\dom(\pi)$ there exists $A\subseteq (L(V_{\lambda+1}))^i$ such that $A\in\pi(a,b,i)$;
 \item for all $i\in\omega$ and $(a,b,i)\in\dom(\pi)$, if $m<i$ then $(a\cap V_{\lambda_m},b\cap V_{\lambda_m},m)\in\dom(\pi)$ and $\pi(a,b,i)$ projects to $\pi(a\cap V_{\lambda_m},b\cap V_{\lambda_m},m)\in\dom(\pi)$, i.e., for every $A\in\pi(a\cap V_{\lambda_m},b\cap V_{\lambda_m},m)$ 
 \begin{equation*}
  \{x\in (L(V_{\lambda+1}))^i:x\upharpoonright m\in A\}\in\pi(a,b,i);
 \end{equation*}
 \item for all $x\subseteq V_\lambda$, $x\in Z$ iff there exists $y\subseteq V_\lambda$ such that the tower $\langle\pi(x\cap V_{\lambda_m},y\cap V_{\lambda_m},m:m\in\omega\rangle$ is well-founded.
\end{itemize}
\end{defin} 

The question is now whether there are sets in $L(V_{\lambda+1})\cap V_{\lambda+2}$ that are $\mathbb{U}(j)$-representable, and particularly which ones are. The desired result would be, in similarity with \AD, that they are all $\mathbb{U}(j)$-representable. There has been some twists and turns on the still ongoing quest towards this result, and it is worth to see them step by step, with some advice for the reader that wants to know all the details.

The main source for $\mathbb{U}(j)$-representability is \cite{Woodin}. There the first closure properties are introduced, reminiscent of the closure properties for weakly homogeneously Suslin sets:

\begin{teo}[Woodin, Lemma 114, Lemma 115 and Theorem 134 in \cite{Woodin}]
 Suppose $j:L(V_{\lambda+1})\prec L(V_{\lambda+1})$ with $\crt(j)<\lambda$. Then 
 \begin{itemize}
  \item the $\lambda$-union of $\mathbb{U}(j)$-representable sets is $\mathbb{U}(j)$-representable 
	\item the projection of a $\mathbb{U}(j)$-representable set is $\mathbb{U}(j)$-representable
	\item if $Z\in L_\lambda(V_{\lambda+1})$ is $\mathbb{U}(j)$-representable, then its complement is $\mathbb{U}(j)$-representable.
 \end{itemize}
 Therefore all sets in $L_\lambda(V_{\lambda+1})\cap V_{\lambda+2}$ are $\mathbb{U}(j)$-representable.
\end{teo}

This was the best result for $\mathbb{U}(j)$-representability in \cite{Woodin}. Further research was blocked by the Tower Condition. Remember by \ref{gaifpow} that a tower of ultrafilters $\langle U_n:n\in\omega\rangle$ is well-founded iff there exists an $X$ such that $X\cap\dom(U_n)\in U_n$ for all $n\in\omega$. The Tower Condition holds for $\mathbb{U}(j)$ iff for any $\mathbb{U}(j)$-representation, it is possible to find such an $X$ for any well-founded tower uniformly:

\begin{defin}[Woodin, 2011 \cite{Woodin}]
 The Tower Condition holds for $\mathbb{U}(j)$ iff for any $A\subseteq\mathbb{U}(j)$, $A\in L(V_{\lambda+1})$, $|A|\leq\lambda$, there is a function $F:A\to L(V_{\lambda+1})$ such that, for all $\langle U_i:i\in\omega\rangle\in L(V_{\lambda+1})$ that satisfies:
 \begin{itemize}
  \item for all $i\in\omega$ there exists $Z\in U_i$ such that $Z\subseteq L(V_{\lambda+1})^i$;
	\item for all $i\in\omega$ $U_i\in A$ and $U_{i+1}$ projects to $U_i$.
 \end{itemize}
 Then the tower is wellfounded iff there exists a function $f:\omega\to L(V_{\lambda+1})$ such that for all $i\in\omega$, $f\upharpoonright i\in F(U_i)$.
\end{defin}

Note that if $\pi$ is a $\mathbb{U}(j)$ representation and $A=\ran(\pi)$, then all towers in $A$ satisfy the conditions. Woodin immediately noticed that the Tower Condition was central for the study of $\mathbb{U}(j)$-representable sets:

\begin{lem}[Woodin, Lemma 127 in \cite{Woodin}]
 \label{lambdaplusrepr}
 If the Tower Condition holds for $\mathbb{U}(j)$, then the complement of a $\mathbb{U}(j)$-representable set is $\mathbb{U}(j)$-representable, therefore all the sets in $L_{\lambda^+}(V_{\lambda+1})$ are $\mathbb{U}(j)$-representable.
\end{lem}

For much of Section 7 of \cite{Woodin}, the Tower Condition is a constant and rather cumbersome presence, as at the time of the publication it was still not proven. In \cite{Cramer} Cramer actually managed to prove the Tower Condition, therefore unlocking much of the content of Section 7 of \cite{Woodin}. The suggestion for the completist reader is therefore first to read \cite{Cramer}, and reading Section 7 of \cite{Woodin} after, so that the Tower Condition provides the most results.

This approach would also make some proofs in \cite{Woodin} more direct and easier to read. For example:

\begin{lem}[Woodin, Lemma 121 of \cite{Woodin}]
 Suppose $Z\in L(V_{\lambda+1})$ is $\mathbb{U}(j)$-representable and $|Z|>\lambda$. Then $Z$ contains a $\mathbb{R}$-perfect set\footnote{With $\mathbb{R}$-perfect set we intend that $2^\omega$ can be continuously embedded in $Z$, a weaker property than to be a perfect set. Therefore this Lemma is now incorporated in \ref{allperfect}}.
\end{lem}

The proof of such Lemma starts by fixing an elementary embedding $j_0:V\to M_0$ such that $(M_0)^{\lambda^+}\subseteq M_0$, $A_0=\{j_0(U):U\in\ran(\pi)\}$. All of this is to prove the Tower Condition for $j_0''A_0$. Now that we know that the Tower Condition holds, we can ignore this trick and just consider $M_0=V$, $j_0=\id$ and $A_0=\ran(\pi)$.

It is also possible to improve \ref{lambdaplusrepr}:

\begin{lem}[Woodin, Corollary 148 of \cite{Woodin}]
 Suppose there is $j:L(V_{\lambda+1})\prec L(V_{\lambda+1})$ with $\crt(j)<\lambda$. Let $\kappa=\lambda^+$ and $\eta=\sup\{(\kappa^+)^{L[A]}:A\subseteq\lambda\}$. Then every set in $L_\eta(V_{\lambda+1})$ is $\mathbb{U}(j)$-representable.
\end{lem}

So it is still open:

\begin{Q}
 \label{representable}
 Are all sets in $L(V_{\lambda+1})\cap V_{\lambda+2}$ $\mathbb{U}(j)$-representable?
\end{Q}

\section{Rank-into-rank and forcing}
\label{forcing}

A recurrent image for describing the difference of the extension of the universe of sets via large cardinals or via forcing is that large cardinals extend the universe vertically, while forcing extends the universe horizontally. Such image has definitively its shortcomings, first because stronger large cardinals do not always mean larger large cardinals (we will see an example of this), second because it implies that large cardinals and forcing are orthogonal, while there are cases of interdependence. It is therefore natural, facing new axioms, to ask what is their relationship with forcing. The ideal results are two:
\begin{itemize}
 \item A forcing always destructs the large cardinal axiom. This would be a strong indication that the propositions forced could be inconsistent with the large cardinal.
 \item A forcing never destructs the large cardinal axiom. This means that the proposition forced is independent from the large cardinal, and we are actually in a situation of orthogonality.
\end{itemize}

The test case has always been Cohen forcing: any large cardinal cannot be destructed by Cohen forcing, therefore the Continuum Hypotheses is independent from any large cardinal. 

The main lemma for proving the indestructibility of large cardinals via forcing is the following:

\begin{lem}[Lifting Lemma]
 Let $j:M\prec N$. If $\mathbb{P}\in M$ is a forcing notion, and $G$ is $\mathbb{P}$-generic, then for any $H$ $j(\mathbb{P})$-generic such that $j''G\subseteq H$ there exists $j^*:M[G]\prec N[H]$ such that $j^*\upharpoonright M=j$.
\end{lem}
\begin{proof}
 The key point is that the forcing relation is definable. For any $\mathbb{P}$-name $\tau$, define $j^*(\tau_G)=(j(\tau))_H$. Then $M[G]\vDash\varphi(\tau_G)$ iff there exists $p\in G$ such that $p\vDash\varphi(\tau)$. So by elementarity $j(p)\vDash\varphi(j(\tau))$. But $j(p)\in H$, i.e., $N[H]\vDash\varphi((j(\tau))_H)$, i.e., $N[H]\vDash\varphi(j^*(\tau_G))$. The opposite direction holds considering $\neg\varphi$.
\end{proof}

This gives immediately a result of lifting: Let $j:V_\lambda\prec V_\lambda$, $\mathbb{P}\in V_{\crt(j)}$. Then $j(\mathbb{P})=\mathbb{P}$, so $G=H$ in the Lifting Lemma works, as $j''G=G$. Therefore $j$ extends to $j^*:V_\lambda[G]\prec V_\lambda[G]$. But this is still not $V[G]\vDash I3(\lambda)$, because that is $\exists k:V[G]_\lambda\prec V[G]_\lambda$, where $V[G]_\lambda=(V_\lambda)^{V[G]}$. The difference is subtle, but fundamental: $V_\lambda[G]$ is the set of $\tau_G$ such that $\tau\in V_\lambda$, $V[G]_\lambda$ is the set of $\tau_G$ with rank of $\tau_G$ less than $\lambda$ in $V[G]$. As the rank of the interpretation of a name is always less or equal than the rank of the name, clearly $V_\lambda[G]\subseteq V[G]_\lambda$, but it is not immediate to see whether the opposite direction is true. This is solved by the nice names lemma:

\begin{lem}[Name Rank Lemma]
 If $\tau$ is a $\mathbb{P}$-name, $\mathbb{P}\in V_{\gamma+1}$ and $\tau_G\in V[G]_\beta$ for any $G$ $\mathbb{P}$-generic, then there is a name $\sigma\in V_{\gamma+3\cdot\beta}$ such that $\sigma_G=\tau_G$ for any $G$ $\mathbb{P}$-generic.
\end{lem}  

In particular, if $\mathbb{P}\in V_{\crt(j)}$, any element of $V[G]_\lambda$ will be in some $V[G]_{\kappa_n}$, and so it will have a name in $V_{3\cdot\kappa_n}\subseteq V_\lambda$, therefore $V_\lambda[G]=V[G]_\lambda$. 

This is peculiar to rank-into-rank axioms: the indestructibility proof must pass two phases: first we have to extend the embedding, then we have to make sure that the domain of the embedding is exactly what we want.

By the discussion before Proposition \ref{verylarge}, $\crt(j)$ can be as large as we want below $\lambda$, therefore we have the following:

\begin{lem}
 Suppose I3($\lambda$), let $\mathbb{P}\in V_\lambda$ be a forcing notion and $G$ be $\mathbb{P}$-generic. Then $V[G]\vDash I3(\lambda)$, So I3 is not destructed by ``small'' forcings.
\end{lem}

The Lifting Lemma holds easily also for $j:V_{\lambda+1}\prec V_{\lambda+1}$ and $\mathbb{P}\in V_{\crt(j)}$, so there exists $j^*:V_{\lambda+1}[G]\prec V_{\lambda+1}[G]$. We have to prove that $V_{\lambda+1}[G]=V[G]_{\lambda+1}$. If $\tau_G\in V[G]_{\lambda+1}$, then $\tau_G\subseteq V[G]_\lambda$, so $\tau_G=\bigcup_{n\in\omega}(\tau_G\cap V[G]_{\kappa_n})$. As $V_\lambda[G]=V[G]_\lambda$, any $\tau_G\cap V[G]_{\kappa_n}$ has a name in $V_\lambda$. The name of a union of sets is the union of the names of the sets, therefore there exists a name for $\tau_G$ that is the $\omega$-union of sets in $V_\lambda$, so it must be in $V_{\lambda+1}$. 

Again, $\crt(j)$ can be unbounded below $\lambda$, so:

\begin{lem}
 Suppose I1($\lambda$) (or anything slightly weaker, like I2 or $\Sigma^1_n$), let $\mathbb{P}\in V_\lambda$ be a forcing notion and $G$ be $\mathbb{P}$-generic. Then $V[G]\vDash I1(\lambda)$ (or the corresponding weaker axiom). So I1 is not destructed by ``small'' forcings.
\end{lem}

Finally we consider I0 and $\mathbb{P}\in V_{\crt(j)}$. Again, the Lifting Lemma holds, so there exists $j^*:L(V_{\lambda+1})[G]\prec L(V_{\lambda+1})[G]$ with $\crt(j^*)<\lambda$. This time we cannot always prove that $L(V_{\lambda+1})[G]=L(V[G]_{\lambda+1})$, but it actually does not matter: as a constructible set will have a constructible name, we have $L(V[G]_{\lambda+1})\subseteq L(V_{\lambda+1})[G]$. Also, $L(V[G]_{\lambda+1})$ is a definable (with parameter $\lambda$) subclass of $L(V_{\lambda+1})[G]$, therefore 
\begin{equation*}
 j^*\upharpoonright L(V[G]_{\lambda+1}):L(V[G]_{\lambda+1})\prec L(V[G]_{\lambda+1})
\end{equation*}
and we have I0($\lambda$) in the extension. 

Can we say that $\crt(j)$ can be unbounded below $\lambda$? For this, we have to prove that a $j$ that witnesses I0 is at least finitely iterable. If $j$ is not weakly proper, then it is not clear whether this is possible. But if $j$ is weakly proper, then it is:

\begin{lem}[Woodin, Lemma 16 of \cite{Woodin}]
\label{FinIterable}
 Let $j:L(V_{\lambda+1})\prec L(V_{\lambda+1})$ with $\crt(j)<\lambda$ and weakly proper. Then $j$ is finitely iterable.
\end{lem}
\begin{proof}
 Let $I$ be the class of fixed points of $\xi$ (it is a class by the proof of \ref{onlyfromVlambda}). For each $s\in[I]^{<\omega}$, let $Z_s=Hull^{L(V_{\lambda+1})}(V_{\lambda+1}\cup\{V_{\lambda+1}\}\cup s)$. Then $\langle Z_s:s\in[I]^{<\omega}\rangle$ is a directed system. Let $Z$ be its limit. Then $Z\prec L(V_{\lambda+1})$. 

For each $s\in [I]^{<\omega}$, $j\upharpoonright Z_s$ is definable from $\{j\upharpoonright V_\lambda,s\}\in L(V_{\lambda+1})$, therefore $k_s=j(j\upharpoonright Z_s)$ is well-defined. As $j(Z_s)=Z_s$, we have that $k_s:Z_s\prec Z_s$. Then we can define $k^*:Z\prec Z$, its direct limit. Composing with the collapse of $Z$, we have $k:L(V_{\lambda+1})\prec L(V_{\lambda+1})$, and $k\upharpoonright V_\lambda=j(j\upharpoonright V_\lambda)$. If $k$ is not weakly proper, then we consider $k_U$ its weakly proper part, and we call it still $k$. 

This proves the existence of $j^2$. As $j^2$ is weakly proper, we can prove by induction that $j^n$ exists and is weakly proper for every $n\in\omega$.
\end{proof}

In particular $j^n\upharpoonright V_\lambda=(j\upharpoonright V_\lambda)^n$, therefore $\crt(j^n)=\kappa_n$, and so:

\begin{lem}
\label{nosmallforcing}
 Suppose I0($\lambda$), let $\mathbb{P}\in V_\lambda$ be a forcing notion and $G$ be $\mathbb{P}$-generic. Then $V[G]\vDash I0(\lambda)$, So I0 is not destructed by ``small'' forcings.
\end{lem}

The reason why we do not know whether $L(V_{\lambda+1})[G]=L(V[G]_{\lambda+1})$ is quite deep:

\begin{teo}[Woodin, Theorem 175 in \cite{Woodin}]
 Suppose $j:L(V_{\lambda+1})\prec L(V_{\lambda+1})$ with $\crt(j)<\lambda$. Let $\mathbb{P}\in V_\lambda$, $G$ be $\mathbb{P}$-generic and $\lambda^\omega\neq(\lambda^\omega)^{V[G]}$. Then $V_{\lambda+1}\notin L(V[G]_{\lambda+1})$, so $L(V_{\lambda+1})[G]\neq L(V[G]_{\lambda+1})$\footnote{The proof in \cite{Woodin} is more convoluted, as it proves that $V_{\lambda+1}$, if in $L(V[G]_{\lambda+1})$ and $\mathbb{U}(j)$-representable, is a perfect set, and this brings a contradiction. By \ref{allperfect} we immediately have that $V_{\lambda+1}$ is perfect.}.
\end{teo}
\begin{proof}

Let $\langle\kappa_i:i\in\omega\rangle$ be the critical sequence of $j$. As $\mathbb{P}\in V_\lambda$, we can suppose that $\mathbb{P}\in V_{\kappa_0}$, and then there exists $j^*:L(V[G]_{\lambda+1})\prec L(V[G]_{\lambda+1})$ that extends $j$. Suppose that $V_{\lambda+1}\in L(V[G]_{\lambda+1})$. Then by Theorem \ref{allperfect} in $L(V[G]_{\lambda+1})$ it is possible to embed continuously $(\lambda^\omega)^{V[G]}$ in $V_{\lambda+1}$. 

We do some cosmetic changes so that the proof is easier to read: we fix $\rho_n$ bijections between $V_{\kappa_n+1}$ and $|V_{\kappa_n+1}|$, rename $|V_{\kappa_n+1}|$ as $\kappa_n$, and modify the continuous embedding so that its domain is $(\Pi_{n\in\omega}\kappa_i)^{V[G]}$. So there exists $\pi$ that continuously embed $(\Pi_{n\in\omega}\kappa_i)^{V[G]}$ to $(\Pi_{n\in\omega}\kappa_i)^{V}$. We prove that $(\Pi_{n\in\omega}\kappa_i)^{V[G]}\subseteq V$, and therefore $(\lambda^\omega)^{V[G]}=\lambda^\omega$, contradiction.

Let $\pi$ be such embedding. We define for every $s\in\bigcup_{n\in\omega}\Pi_{i<n}\kappa_n$ a set $S_s\subseteq \kappa_{\lh(s)}$, $S_s\in V$ in such a way:
	\begin{itemize}
		\item $0\in S_s$;
		\item Fix $n=\lh(s)$. Let $\alpha\in S_s$. Then let $P_\alpha=\{\beta<\kappa_n:\exists p\in\mathbb{p}\ p\vDash\beta=\min\{\dot{\pi}(a)(n):a\in(\Pi_{n\in\omega}\kappa_n)^v,\ \pi(a)\upharpoonright n=s,\ \dot{\pi}(a)(n)>\alpha\}\}$. Then $|P_\alpha|<\kappa_0$, therefore $P_\alpha$ is bounded in $\kappa_n$. As for any $s\in\bigcup_{n\in\omega}\Pi_{i<n}\kappa_n$ $\{\pi(a)(n):a\in(\Pi_{n\in\omega}\kappa_n)^{V[G]},\ \pi(a)\upharpoonright n=s\}$ is unbounded in $\kappa_n$, $P_\alpha\neq\emptyset$, so let $\alpha'=\max P_\alpha +1\in S_s$. 
	\end{itemize}
	
	Note that not only $S_s\in V$, but $\{S_s:s\in \bigcup_{n\in\omega}\Pi_{i<n}\kappa_n\}\in V$. 
	
	By the construction, (*) for every $a\in (\Pi_{n\in\omega}\kappa_n)^{V[G]}$ and every $\alpha\in S_{\pi(a)\upharpoonright n}$, there exists $b\in (\Pi_{n\in\omega}\kappa_n)^{V[G]}$ such that $\pi(a)\upharpoonright n=\pi(b)\upharpoonright n=s$, and $\alpha<\pi(b)(n)<\min S_s\setminus\alpha$. 
	
	We define now another continuous embedding, $\pi'$. For any $S_s$, let $S_s(\alpha)$ be the $\alpha$-th element of $S_s$. 
	
	Let $a\in(\Pi_{n\in\omega}\kappa_n)^{V[G]}$. Then we define by induction $b_n$:
	\begin{itemize}
		\item Applying (*) with $n=0$, $a$ anything and $\alpha=S_{\emptyset}(a(0))$, there exists $b_0$ such that $S_{\emptyset}(a(0))<\pi(b_0)(0)<S_{\emptyset}(a(0)+1)$;
		\item Suppose that $b_n$ is already defined. Applying (*) with $a=b_n$ and $\alpha=S_{\pi(b_n)\upharpoonright n}(a(n+1))$, there exists $b_{n+1}$ such that $\pi(b_{n+1})\upharpoonright n=\pi(b_n)\upharpoonright n$ and $S_{\pi(b_n)\upharpoonright n}(a(n+1))<\pi(b_{n+1})(n+1)<S_{\pi(b_n)\upharpoonright n}(a(n+1)+1)$.
	\end{itemize}
	
	As $\pi$ is continuous, the $b_n$'s are more and more closer. Let $b$ be their limit. Then we define $\pi'(a)=\pi(b)$.
	
	Note that, if $\alpha$ and $\alpha'$ are two consecutive members of $S_s$, then $|\{\pi(b)(n):\exists a\in(\Pi_{n\in\omega}\kappa_n)^{V[G]},\ \pi'(a)=\pi(b),\ \pi(b)\upharpoonright n=s\}\cap [\alpha,\alpha']|=1$. Therefore $\pi'$ is well-defined and injective. But then the inverse of $\pi'$ is definable in $V$, as for any $a\in(\Pi_{n\in\omega}\kappa_n)^{V[G]}$, $\pi'(a)(n)=|S_{\pi'(a)\upharpoonright n}\cap a(n)|$.
\end{proof}

On the other hand, if the forcing is $\omega$-closed in $V_\lambda$ then we have equality:

\begin{teo}
\label{omegaclosed}
 Suppose $j:L(V_{\lambda+1})\prec L(V_{\lambda+1})$ with $\crt(j)<\lambda$. Let $\mathbb{P}\in V_\lambda$ be $\omega$-closed in $V_\lambda$, i.e., $(V_\lambda)^\omega=((V_\lambda)^\omega)^{V[G]}$, $G$ be $\mathbb{P}$-generic. Then $L(V_{\lambda+1})[G]=L(V[G]_{\lambda+1})$.
\end{teo}
\begin{proof}
 Fix a bijection in $V$ between $V_\lambda$ and $\lambda$, this extends naturally to a bijection from $V_{\lambda+1}$ and ${\cal P}(\lambda)$, so it suffices to define ${\cal P}(\lambda)$ inside $L(V[G]_{\lambda+1})$: 
\begin{equation*}
 {\cal P}(\lambda)=\{a\in {\cal P}^{V[G]}(\lambda):\forall n\in\omega\ a\cap\kappa_n\in V\}. 
\end{equation*}
This is because if $a\cap\kappa_n=a_n\in V$, then $\langle a_n:n\in\omega\rangle\in V$ and $a=\bigcup_{n\in\omega}a_n\in V$. 
\end{proof}

\begin{cor}
  Suppose $j:L(V_{\lambda+1})\prec L(V_{\lambda+1})$ with $\crt(j)<\lambda$. Let $\mathbb{P}\in V_\lambda$ be a forcing notion, $G$ be $\mathbb{P}$-generic. Then $L(V_{\lambda+1})[G]=L(V[G]_{\lambda+1})$ iff $(V_\lambda)^\omega=((V_\lambda)^\omega)^{V[G]}$.
\end{cor}

So I3, I2, I1 and I0 are not changed by forcing notions bounded in $\lambda$. Also, trivially, if a forcing notion is $\lambda$-closed, then $V[G]_{\lambda+1}=V_{\lambda+1}$, therefore I3, I2, I1 and I0 are witnessed by the exact same embedding, and are not destructed. In other words, they are not changed by forcings that are above $\lambda$. 

This case can give some very nice results:

\begin{teo}[Shi, Trang, 2017 \cite{ShiTrang}]
 \label{secondShiTrang}
 Suppose I0($\lambda$) holds. Then all the following statements are true in a generic extension where I0($\lambda$) still holds:
 \begin{itemize}
  \item special $\lambda^+$-Aronszajn trees exist;
	\item $\lambda^+$-Suslin tree exist;
	\item there is a very good scale at $\lambda^+$;
	\item Stationary Reflections fails at $\lambda^+$;
	\item $\Diamond_{\lambda^+}$:
 \end{itemize}
\end{teo} 

This is particularly striking if compared with Theorem \ref{firstShiTrang}.

It remains to see what happens when the forcing has size $\lambda$. The first case is when the forcing is unbounded in $\lambda$ (the second case will be when the forcing is exactly at $\lambda$). The typical case is when the forcing is a $\lambda$-iteration of small forcings. In such case, the reverse Easton iteration has been proven in the past to behave very well with large cardinals:

\begin{defin}
 Let $\mathbb{P}_\alpha$ be a forcing iteration of length $\alpha$, where $\alpha$ is either a strong limit cardinal or is equal to $\infty$, the class of all ordinals. Then $\mathbb{P}_\alpha$ is a reverse Easton iteration if nontrivial forcing is done only at infinite cardinal stages, direct limits are taken at all inaccessible cardinal limit stages, and inverse limits are taken at all other limit stages; moreover, $\mathbb{P}_\alpha$ is the direct limit of $\langle\mathbb{P}_\delta:\delta<\alpha\rangle$ if $\alpha$ is regular or $\infty$, the inverse limit otherwise.
\end{defin}

Such iteration was introduced by Easton to force different behaviours of the function $\kappa\mapsto 2^\kappa$, the power function. There are two known rules for the power function: if $\kappa<\eta$ then $2^\kappa\leq 2^\eta$, and $\cof(2^\kappa)>\kappa$ for $\kappa$ regular. Easton proved that any function that satisfies this two rules can be the power function on the regular cardinals:

\begin{defin}
 Let $E:\Reg\to\Card$ be a class function. Then $E$ is an Easton function iff:
 \begin{itemize}
   \item $\alpha<\beta\rightarrow E(\alpha)\leq E(\beta)$;
	 \item $\cof(E(\alpha))>\alpha$ for all $\alpha\in\Reg$.
 \end{itemize}
\end{defin}

Then Easton proved:

\begin{teo}[Easton, 1970 \cite{Easton}]
 Let $E$ be an Easton function. Then there exists a generic extension $V[G]$ of $V$ such that $V[G]\vDash\forall\kappa(\kappa\in\Reg\rightarrow 2^\kappa=E(\kappa))$.
\end{teo}

The forcing used was a reverse Easton iteration. The easiest application is to force \GCH: in that case for any $\gamma<\alpha$ $\mathbb{Q}_\gamma$ forces that $2^\gamma=\gamma^+$ if $\gamma$ is a cardinal, it is trivial otherwise, $\mathbb{P}_{\gamma+1}=\mathbb{P}_\gamma*\dot{Q}_\gamma$ and $G$ will the generic filter of the reverse Easton iteration $\mathbb{P}_\infty$. The problem is when one, for example, starts with a model of \GCH{} and wants to increase everywhere on the regulars the power function: If we force with $\mathbb{Q}_\omega$ $2^\omega=\omega_2$ and after that we force $2^{\omega_1}=\omega_3$ via Cohen forcing, then in the final model $2^\omega$ will collapse to $\omega_1$ again. The idea is then that $\mathbb{Q}_\delta$ is defined only on $\delta$ closed under $E$, and $\mathbb{Q}_\delta$ is the Easton product\footnote{It is the product with support the Easton ideal, i.e., the sets that are bounded below every inaccessible cardinal.} of the forcings that make $2^\gamma=E(\gamma)$ for all $\gamma\in[\delta,\delta^*)$, where $\delta^*$ is the smallest cardinal closed under $E$ bigger than $\delta$ (this is still in \cite{Easton}).

So let $E$ be an Easton function, and $\mathbb{P}_\lambda$ that forces $2^\kappa=E(\kappa)$ for any regular $\kappa$, with the stages indicated with $\mathbb{Q}_\delta$ for $\delta<\lambda$. We want to know for which $E$'s the forcing $\mathbb{P}_\lambda$ does not destroy I3, $\dots$, I0. 

If for some $\delta<\lambda$ $|\mathbb{Q}_\delta|\geq\lambda$, then there is a $\alpha<\delta^*$ such that in $V[G]$ $2^\alpha\geq\lambda$, therefore $\lambda$ is not strong limit anymore in $V[G]$ and the embedding is destroyed. Therefore we must have $|\mathbb{Q}_\delta|<\lambda$ for any $\delta<\lambda$.

If for some $\delta<\lambda$ we have that $j(\mathbb{P}_\delta)\neq\mathbb{P}_{j(\delta)}$, then possibly $j(2^\alpha)\neq 2^{j(\alpha)}$ for some $\alpha<\delta$, therefore we should ask that $j(\mathbb{P}_\delta)=\mathbb{P}_{j(\delta)}$. If $E$ is definable then this is trivial.

Also note that $\mathbb{Q}_\delta$ is $\delta$-directed closed. This is enough to lift I3, $\dots$, I0 in the generic extension:

\begin{teo}[\cite{DimFrie}]
\label{forcingupto}
Suppose $j$ witnesses I3($\lambda$) (or I1($\lambda$) $\dots$). Let $\mathbb{P}_\lambda$ be a forcing iteration of length $\lambda$ that is reverse Easton, directed closed (i.e., for all $\delta<\lambda$, $\mathbb{Q}_\delta$ is $\delta$-directed closed), $\lambda$-bounded (i.e., for all $\delta<\lambda$, $|\mathbb{Q}_\delta|<\lambda$) and $j$-coherent (i.e., for all $\delta<\lambda$, $j(\mathbb{P}_\delta)=\mathbb{P}_{j(\delta)}$). Then $j$ lifts to any generic extension via $\mathbb{P}_\lambda$.
\end{teo}
\begin{proof}[Sketch of proof]
 The proof is in two stages: lifting and right domain.

 First notice that if $\gamma<\kappa_n$, then $|\mathbb{P}_\gamma|<\kappa_n$, by elementarity, $\lambda$-boundedness and $n$ applications of $j$. Fixing an $n<\omega$, consider the image of $\mathbb{Q}_{\kappa_n}$ under $j$: it is a subset of $\mathbb{Q}_{\kappa_{n+1}}$ of cardinality less than $\kappa_{n+1}$, therefore by directed closeness there is an element $\mathbb{Q}_{\kappa_{n+1}}$ that extends all the $j(p)$ with $p\in\mathbb{Q}_{\kappa_n}$. Doing the same for all $n$ and also for $\mathbb{P}_{\kappa_n,\kappa_{n+1}}$, one gets a condition of $\mathbb{P}_\lambda$ that, if in a generic $G$, implies that $j''G\subseteq G$. Therefore we can apply the Lifting Lemma.

As for the right domain, since the forcing is directed closed every element in the generic extension is actually in a generic extension for some $\mathbb{P}_{\kappa_n}$, therefore the previous proofs still hold.
\end{proof}

Theorem \ref{forcingupto} is just the final step of a series of theorems that goes in that direction: in \cite{Hamkins} Hamkins has proved such theorem for I1 and reverse Easton directed closed $j$-coherent iterations such that $|\mathbb{Q}_\delta|<2^\delta$; in \cite{Corazza2} Corazza has proved it for I3 and reverse Easton directed closed $j$-coherent iterations such that $|\mathbb{Q}_\delta|<i(\delta)$, where $i(\delta)$ is the smallest inaccessible larger than $\delta$; in \cite{Friedman} Friedman has proved it for I2 (called there $\omega$-superstrong) and GCH.

The hypotheses in Theorem \ref{forcingupto} are very reasonable, and most of the reverse Easton forcing usually employed satisfy them. So, for example:

\begin{cor}
 Suppose I3 (or I1, $\dots$). Then in a generic extension can hold:
 \begin{itemize}
  \item I3 + \GCH;
	\item I3 + $2^\kappa=\kappa^{+++}$ for any $\kappa$ regular;
	\item I3 + $V=\HOD$;
	\item I3 + $\Diamond$ everywhere;
	\item I3 + every $\kappa$ supercompact is Laver indestructible (see \cite{Shi});
	\item $\dots$
 \end{itemize}
\end{cor}

Finally the last case is when the forcing is changing $\lambda^+$. This is the hardest case, and at the moment the indestructibility results are very scarce. For example, Gitik short extenders Prikry forcing construction in \cite{Gitik} can be carried on under I3 (it needs $o(\kappa_n)\geq(\kappa_n)^{+m}$ for any $n,m\in\omega$): such construction push $2^\lambda$ to $\lambda^{+\delta+1}$ for any desired $\delta<\aleph_1$, but does not add bounded subsets of $\lambda$, therefore $V[G]_\lambda=V_\lambda$ and I3 is not destructed. Yet, this construction for I1 does not work, as $V_{\lambda+1}\neq V[G]_{\lambda+1}$.

In \cite{CummFore} Cummings and Foreman, using something more than I2, manage to prove that I2($\lambda$) is consistent with $2^\lambda=\lambda^{++}$.

The most successful strategy has been instead to start with I0($\lambda$), and using some sophisticated tool to prove that, adding a Prikry sequence to the critical point $\kappa_0$, in the extension I1($\kappa_0$) holds. 

\begin{defin}
 Let $\kappa$ be a measurable cardinal, and $U$ a normal measure on $\kappa$. Then $\mathbb{P}$, the Prikry forcing on $\kappa$ via $U$, is the set of $(s,A)$ such that $s\in[\kappa]^{<\omega}$, $A\in U$ and $\min A>\max s$.

 We say that $(s,A)<(t,B)$ if $s\sqsupseteq t$, $A\subseteq B$ and for any $n\in\lh(s)\setminus\lh(t)$, $s(n)\in B$.
\end{defin}

If $(s,A)$ and $(t,B)$ are in the generic set $G$, then, $s$ and $t$ must be compatible. Therefore by density $\bigcup\{s:\exists A\ (s,A)\in G\}$ is an $\omega$-sequence cofinal in $\kappa$. So $(\cof(\kappa)=\omega)^{V[G]}$, but it is a very delicate forcing, as it does not add any bounded subset of $\kappa$. 

There is a convenient condition for an $\omega$-sequence cofinal in $\kappa$ to be generic:

\begin{teo}[Mathias Condition, or geometric condition]
 Let $\kappa$ be a measurable cardinal, $\mathbb{P}$ the Prikry forcing on $\kappa$ via the normal ultrafilter $U$ and let $\langle\alpha_n:n\in\omega\rangle$ be a cofinal sequence in $\kappa$. Then $\langle\alpha_n:n\in\omega\rangle$ is generic for $\mathbb{P}$ iff for any $A\in U$ the set $\langle\alpha_n:n\in\omega\rangle\setminus A$ is finite.
\end{teo}

In Theorem \ref{FinIterable} we noticed that if $j$ witnesses I0($\lambda$) is weakly proper, then it is finitely iterable, and its iterates are still weakly proper embeddings from $L(V_{\lambda+1})$ to itself. Woodin proved that iterability can go further:

\begin{teo}[Woodin, Lemma 21 in \cite{Woodin}]
 Suppose that $j:L(V_{\lambda+1})\prec L(V_{\lambda+1})$ with $\crt(j)<\lambda$ is weakly proper. Then $j$ is iterable.
\end{teo}

From the proof of Lemma \ref{iterable} $j_{0,\omega}(\kappa_0)=\lambda$, therefore $M_\omega$ is different than $L(V_{\lambda+1})$: in it, by elementarity, $\lambda$ is regular. Let $U$ be the ultrafilter from $j$. Then $j_{0,\omega}(U)$ is a normal ultrafilter on $\lambda$, therefore $\lambda$ is measurable in $M_\omega$. If $\mathbb{P}$ is the Prikry forcing on $\kappa_0$ via $U$, then $j_{0,\omega}(\mathbb{P})$ is the Prikry forcing on $\lambda$ via $j_{0,\omega}(U)$.

\begin{rem}
 If $\langle\kappa_n:n\in\omega\rangle$ is the critical sequence of $j$, then $\langle\kappa_n:n\in\omega\rangle$ is $j_{0,\omega}(\mathbb{P})$-generic in $M_\omega$.
\end{rem}
\begin{proof}
 Note that Mathias' characterization needs the Axiom of Choice, but by elementarity $(V_\lambda)^{M_\omega}\vDash\AC$, so we can use it. If $A\in j_{0,n}(U)$, then by the definition of $U$ $\kappa_n\in j_{n,\omega}(A)$, as $A\in j_{0,n}(U)$ iff $\kappa_n\in j^n(A)$ and 
\begin{equation*}
 \kappa_n=j_{n+1,\omega}(\kappa_n)\in j_{n+1,\omega}(j^n(A))=j_{n,\omega}(A). 
\end{equation*}
So if $A\in j_{0,\omega}(U)$, there exist $n\in\omega$ and $\bar{A}\in L(V_{\lambda+1})$ such that $A=j_{n,\omega}(\bar{A})$. By elementarity $\bar{A}\in j_{0,n}(U)$ and $j_{n,n+i}(\bar{A})\in j_{0,n+i}(U)$. So for any $i\in\omega$, $\kappa_{n+i}\in j_{n+i,\omega}(j_{n,n+i}(\bar{A}))=A$.
\end{proof}

So it make sense to consider $M[\langle\kappa_i:i\in\omega\rangle]$.

\begin{teo}[Generic Absoluteness, Woodin, \cite{Woodin}]
 \label{GenAbsoluteness}
 Let $j:L(V_{\lambda+1})\prec L(V_{\lambda+1})$ with $\crt(j)$ be a weakly proper elementary embedding. Let $\langle\kappa_n:n\in\omega\rangle$ be its critical sequence, and $M_\omega$ its $\omega$-iterated model. 

Let $\delta$ be such that all the sets in $L_\delta(V_{\lambda+1})\cap V_{\lambda+2}$ are $\mathbb{U}(j)$-representable. Then there exists an elementary embedding 
\begin{equation*}
 \pi:L_\delta(M_\omega[\langle\kappa_n:n\in\omega\rangle]\cap V_{\lambda+1})\prec L_\delta(V_{\lambda+1}) 
\end{equation*}
such that $\pi\upharpoonright V_{\lambda+1}$ is the identity.
\end{teo}

Two examples to understand better the theorem:
\begin{itemize}
 \item As $\langle\kappa_n:n\in\omega\rangle\in V_{\lambda+1}$, clearly 
  \begin{equation*}
	 M_\omega[\langle\kappa_n:n\in\omega\rangle]\cap V_{\lambda+1}=(V_{\lambda+1})^{M_\omega[\langle\kappa_n:n\in\omega\rangle]}\subseteq V_{\lambda+1}. 
	\end{equation*}
 Since $\pi\upharpoonright V_{\lambda+1}=\id$, the Theorem says that $M_\omega[\langle\kappa_n:n\in\omega\rangle]\cap V_{\lambda+1}\prec V_{\lambda+1}$. 
 \item If $A$ is definable on $M_\omega[\langle\kappa_n:n\in\omega\rangle]$, then $\pi(A)$ will be the set defined in the same way on $V_{\lambda+1}$, and $\pi(A)\cap M_\omega[\langle\kappa_n:n\in\omega\rangle]=A$.
\end{itemize}

The proof exploits the fact that in the definition of $\mathbb{U}(j)$-representability we had $j^n(U)=U$ for the ultrafilters used in a representation, therefore given a set in $L_\delta(V_{\lambda+1})$, its representation is also in $M_\omega$, in a certain sense is the set encrypted. When we add the critical sequence to $\lambda$, then we can rebuild something similar to the set (if not the set itself) using such representation. The Theorem uses the theory of $L_\delta(V_{\lambda+1})$, as a set, to build an embedding.

This essential tool has an immediate corollary:

\begin{cor}
 Suppose that there exists $j:L(V_{\lambda+1})\prec L(V_{\lambda+1})$ with $\crt(j)=\kappa_0<\lambda$. Let $\mathbb{P}$ be the Prikry forcing on $\kappa_0$. Then there exists $G$ $\mathbb{P}$-generic such that there exists $k:V[G]_{\kappa_0+1}\prec V[G]_{\kappa_0+1}$.
\end{cor}
\begin{proof}
 We have seen that $j\upharpoonright V_{\lambda+1}$ is definable from $j\upharpoonright V_\lambda$, therefore it is in $L_1(V_{\lambda+1})$. Then 
  \begin{equation*}
	 \pi^{-1}(j)=j\cap M_\omega[\langle\kappa_n:n\in\omega\rangle]:(V_{\lambda+1})^{M_\omega[\langle\kappa_n:n\in\omega\rangle]}\prec (V_{\lambda+1})^{M_\omega[\langle\kappa_n:n\in\omega\rangle]}. 
	\end{equation*}
 So in $M_\omega$ there is a condition in $j_{0,\omega}(\mathbb{P})$ that forces the existence of a I1($\lambda$) embedding. By elementarity, in $V$ there is a condition on $\mathbb{P}$ that forces the existence of a I1($\kappa_0$) embedding.
\end{proof}

This corollary, coupled with Theorem \ref{forcingupto}, can expand the possibilities of consistency results: for example one can force $2^{\kappa_0}>\kappa_0^+$ with the reverse Easton forcing, and then force with Prikry to have I1($\kappa_0$). Prikry will not change $2^{\kappa_0}>\kappa_0^+$, and therefore we have a model for $I1(\lambda)+2^\lambda>\lambda^+$.

\begin{teo}
 \label{atlambda1}
 Suppose there exists $j:L(V_{\lambda+1})\prec L(V_{\lambda+1})$ with $\crt(j)<\lambda$. Then for every Easton function $E$ such that $E\upharpoonright\lambda$ is definable over $V_\lambda$, there is a generic extension $V[G]$ of $V$ that satisfies $\exists\gamma<\lambda\ \exists k:V_{\gamma+1}\prec V_{\gamma+1}\wedge 2^\gamma=E(\gamma)$. 
\end{teo}

In particular this proves that it is consistent with I1($\lambda$) that there are no strongly compact or supercompact cardinals below $\lambda$, since $\lambda$ is singular and Solovay's Theorem says that above a strongly compact or a supercompact cardinal we have $2^\lambda=\lambda^+$ on singular cardinals. Therefore this is one of the cases where ``stronger'' does not mean ``larger''. 

The first way to improve such results is to note that Theorem \ref{GenAbsoluteness} can express more than just I1, therefore the same proof works for stronger hypotheses (for example $\exists j:L_\lambda(V_{\lambda+1})\prec L_\lambda(V_{\lambda+1})$), with the only bound given by which sets are $\mathbb{U}(j)$-representable, making Question \ref{representable} even more relevant.

Another way is to notice that in the proof of Theorem \ref{GenAbsoluteness} it is not necessary to limit ourselves to Prikry forcing and to $\langle\kappa_n:n\in\omega\rangle$: the important thing is that there is a $G\in V$ $j_{0,\omega}(\mathbb{P})$-generic and there is an $\omega$-sequence cofinal in $\lambda$ in $M_\omega[G]$. If $G\notin V$ then there are easy counterexamples: if \CH{} holds in $V$, then it holds also in $M_\omega$, and a forcing that adds a Cohen real to $M_\omega$ would make $V_{\lambda+1}^{M_\omega[G]}$ not elementary in $V_{\lambda+1}$ (not even contained, in fact).

There are numerous variations on Prikry forcing that add $\omega$-sequences to a large cardinal, and as $\lambda$ is a very large cardinal in $M_\omega$ (see \ref{verylarge}) they are usually applicable in this context. For example, it is clear in Theorem \ref{atlambda1} that the reverse Easton forcing forces $2^\eta=E(\eta)$ for all the regular cardinals $\eta<\gamma$, not only for $\gamma$, therefore with the construction of Theorem \ref{atlambda1}, based on classical Prikry forcing, $\gamma$ is not the first cardinal on which \GCH{} fails. Being the first one is a stronger property, that is forced usually with Gitik-Magidor long extender Prikry forcing (see \cite{GItik2}). If $\mathbb{P}$ is therefore the Gitik-Magidor long extender Prikry forcing on $\kappa_0$ and it could be possible to find a $G\in V$ $j_{0,\omega}(\mathbb{P})$-generic in $M_\omega$, then by Generic Absoluteness we would have a generic extension in which I1($\kappa_0$) + $2^{\kappa_0}>\kappa_0^+$ and \GCH{} below $\kappa_0$.

To check if this and other forcing notions work with Generic Absoluteness there is the need of a criterion that, if satisfied, implies that there is a $G\in V$, $j_{0,\omega}(\mathbb{P})$-generic in $M_\omega$.

\begin{defin}
 Let $\lambda$ be an infinite cardinal. A partially ordered set $\mathbb{P}$ is $\lambda$-good if it adds no bounded subsets of $\lambda$ and for every $\mathbb{P}$-generic $G$ and every $A\in V[G]$, $A\subset\Ord$, $|A|<\lambda$, there is a non-$\subset$-decreasing $\omega$-sequence $\langle A_i:i<\omega\rangle\subseteq V$ such that $A=\bigcup_{i\in\omega} A_i$.
\end{defin}

\begin{prop}[Shi, Proposition 3.20 in \cite{Shi}]
 Let $j:L(V_{\lambda+1})\prec L(V_{\lambda+1})$ with $\crt(j)=\kappa_0<\lambda$ and let $\mathbb{P}\in V_\lambda$ be a $\kappa_0$-good forcing. Then if $(M_\omega,j_{0,\omega})$ is the $\omega$-iterate of $j$, there is a $G\in V$ $j_{0,\omega}(\mathbb{P})$-generic on $M_\omega$.
\end{prop}

The original definition of $\lambda$-goodness involved families of dense sets, and this is a more convenient way to prove that various Prikry forcings are $\lambda$-good. The following is a sufficient condition:

\begin{defin}[\cite{DimonteWu}]
 Let $\lambda$ be an infinite cardinal. Let $\mathbb{P}$ be a forcing notion equipped with a length measure of the conditions $l:\mathbb{P}\to\omega$ such that $l(1_{\mathbb{P}})=0$ and for $p,q\in\mathbb{P}$, if $p\leq q$ then $l(p)\geq l(q)$ (for example if $\mathbb{P}$ is the Prikry forcing, then $l(s,A)=\lh(s)$). 

 We say that $\mathbb{P}$ is $\lambda$-geometric if it does not add bounded subsets of $\lambda$ and for any $\alpha<\lambda$, any $\langle D_\beta:\beta<\alpha\rangle$ collection of open dense sets and any $p\in\mathbb{P}$ there exists a condition $q\leq p$ such that whenever a filter contains $q$ and meets all the dense open sets $E_n=\{p:l(p)>n\}$, it also meets all the $D_\beta$'s.
\end{defin}

\begin{rem}
 Let $\lambda$ be an infinite cardinal. Let $\mathbb{P}$ be a $\lambda$-geometric forcing notion. Then $\mathbb{P}$ is $\lambda$-good.
\end{rem}
\begin{proof}
 Let $G$ be $\mathbb{P}$-generic. Let $A\subseteq\Ord$, $|A|=\mu<\lambda$, $A\in V[G]$. Let $f:\mu\to\Ord$ such that $\ran(f)=A$, $f\in V[G]$. For any $\alpha<\mu$, let $D_\alpha=\{p\in\mathbb{P}:\exists\beta\ p\vDash\dot{f}(\alpha)=\beta\}$, the set of the conditions that decide the $\alpha$-th element of $A$. By $\lambda$-geometricity, the set of $q$'s such that every filter that contains $q$ and intersects $E_n$ intersects also all $D_\alpha$'s, is dense in $\mathbb{P}$, therefore there exists $q_0\in G$ with such property. Define by induction $q_{n+1}$ that extends $q_n$ such that $q_{n+1}\in G\cap E_n$. Let $B_n=\{\alpha<\mu:q_n\in D_\alpha\}\in V$ be the set of $\alpha$'s such that $q_n$ decides $f(\alpha)$. Then since ${\cal F}=\bigcup_{n\in\omega}{\cal F}_{q_n}$, by $\lambda$-geometricity, intersects all $D_\alpha$'s (${\cal F}_{q_n}$ is the filter generated by $q_n$), for any $\alpha<\mu$ there exists $n\in\omega$ such that $q_n\in D_\alpha$, therefore $\mu=\bigcup_{n\in\omega}B_n$. Define 
\begin{equation*}
 A_n=\{\beta:\exists\alpha\in B_n\ q_n\vDash\dot{f}(\alpha)=\beta\}\in V. 
\end{equation*}
 If $\beta\in A$, then there exists $\alpha<\mu$ such that $f(\alpha)=\mu$, and there exists $n\in\omega$ such that $\alpha\in B_n$, i.e., $q_n$ decides $f(\alpha)$. But since $q_n\in G$, it must be that $q_n\vDash\dot{f}(\alpha)=\beta$, so $\beta\in A_n$, therefore $A=\bigcup_{n\in\omega}A_n$.
\end{proof}

The key point is that many variations of Prikry forcing satisfy a variation of the Prikry condition: Let $p\leq^* q$ be $p\leq q$ and $l(p)=l(q)$. Then such variation says that for any dense set $D$ and every $p\in\mathbb{P}$ there is a $q\leq^* p$, such that for any $r\leq q$ sith $l(r)=l(q)+n$, $r\in D$. In other words every condition can be ``changed'' so that any extension above a certain length is in the dense set. This condition appears frequently in literature, for example in \cite{Mathias} for Prikry forcing, or \cite{Merimovich} for Gitik-Magidor long extender Prikry forcing. In \cite{DimonteWu2} there are other examples, and in both \cite{ShiTrang} and \cite{DimonteWu2} it is exploited to prove the $\lambda$-goodness of Gitik-Magidor long extender Prikry forcing, in different ways. If $\leq^*$ is sufficiently closed, it is easy to see that with repeated uses of this variation of Prikry condition one can find a $q$ that witnesses $\lambda$-geometricity, and therefore $\lambda$-goodness.

It is possible also to prove it directly:

\begin{rem}[\cite{DimonteWu}]
  Let $j:L(V_{\lambda+1})\prec L(V_{\lambda+1})$ with $\crt(j)=\kappa_0<\lambda$ and let $\mathbb{P}\in V_\lambda$ be a $\kappa_0$-geometric forcing. Then if $(M_\omega,j_{0,\omega})$ is the $\omega$-iterate of $j$, there is a $G\in V$ $j_{0,\omega}(\mathbb{P})$-generic on $M_\omega$.
\end{rem}
\begin{proof}
 Note that in $V$ there are only $\lambda$ open dense sets of $j_{0,\omega}(\mathbb{P})$: as $\mathbb{P}\in V_\lambda$ there exists $n\in\omega$ such that $\mathbb{P}\in V_{\kappa_n}$, so $j_{0,\omega}(\mathbb{P})\in M_\omega\cap V_{j_{0,\omega}(\kappa_n)}$. But $M_\omega\cap V_{j_{0,\omega}(\lambda)}$ is the range of all $j_{n,\omega}\upharpoonright V_\lambda$, therefore $|j_{0,\omega}(\mathbb{P})|<|M_\omega\cap V_{j_{0,\omega}(\lambda)}|=|V_\lambda|=\lambda$.

 Let $\langle D_\alpha:\alpha<\lambda\rangle$ be an enumeration of the dense sets of $j_{0,\omega}(\mathbb{P})$ in $V$. For every $n\in\omega$, fix $q_n$ that witnesses $\lambda$-geometricity for $\langle D_\alpha:\alpha<\kappa_n\rangle$. Then for every $m$ there exists a $q'_{n,m}<q_n$ such that $q'_{n,m}\in E_m$. Let $H$ be the filter $\bigcup_{n,m\in\omega}{\cal F}_{q'_{n,m}}$, with ${\cal F}_q$ the filter generated by $q$. Then $H\in V$ is generic.
\end{proof}

Therefore, thanks to the work that in literature has already been done with the structure of Prikry-like forcing, for many variations of Prikry forcing it is possible to prove Generic Absoluteness. Interestingly, for many applications such variations need a forcing preparation that is actually a reverse Easton iteration, therefore Theorem \ref{atlambda1} is necessary. In \cite{DimonteWu} are collected some results:

\begin{teo}
 Suppose I0($\lambda$). Then the following are consistent:
 \begin{itemize}
  \item I1($\gamma$) + $2^\gamma>\gamma^+$ + $\forall\kappa<\gamma\ 2^\kappa=\kappa^+$ (see \cite{Gitik2}, preparation forcing to force \GCH, then Gitik-Magidor long extender Prikry forcing);
	\item I1($\gamma$) + Tree property at $\gamma^+$ (see \cite{Neeman}, preparation forcings that make supercompact cardinals Laver indestructible and add many subsets to all regular cardinals, then diagonal supercompact Prikry forcing);
	\item I1($\gamma$) + Tree property at $\gamma^{++}$ (see \cite{DobrinenFriedman}, preparation forcing is reverse Easton iteration of Sacks forcings, then Prikry forcing);
	\item $\dots$
 \end{itemize} 
\end{teo}

\section{Dissimilarities with $\AD^{L(\mathbb{R})}$}
\label{dissimil}

While there are many similarities between I0 and $\AD^{L(\mathbb{R})}$, not everything really matches. The first thing that stands out is that while $L(\mathbb{R})\vDash\AD^{L(\mathbb{R})}$, $L(V_{\lambda+1})\nvDash I0(\lambda)$. But the big problem is Theorem \ref{nosmallforcing}: the fact that many structural properties below $\lambda$ are independent from I0 in fact implies that I0 cannot possibly prove them, and these properties can be the similarities we are interested in. In the $\AD^{L(\mathbb{R})}$ case this situation is avoided as there are no forcings below $\omega$. For example, we have proven that $\lambda^+$ is measurable in $L(V_{\lambda+1})$ under I0, just like $\omega_1$ was measurable in $L(\mathbb{R})$ under \AD. But in the $L(\mathbb{R})$ case we have that the club filter itself is an ultrafilter, while Theorem \ref{measurable} stops short and proves only that the club filter on a stationary set is measurable. So the question is: how close can we get to the original case?

As $\lambda>\omega$, for any $\alpha<\lambda$ regular $S^{\lambda^+}_\alpha$ is stationary, and for $\alpha\neq\beta$ $S^{\lambda^+}_\alpha\cap S^{\lambda^+}_\beta=\emptyset$, therefore it is not possible to have the club filter to be an ultrafilter. The best outcome would be then to have the club filter to be an ultrafilter on every $S^{\lambda^+}_\alpha$; this is called Ultrafilter Axiom at $\lambda$. But this is consistently false because of Theorem \ref{nosmallforcing}:

\begin{prop}
\label{noultrafilter}
 Let $j:L(V_{\lambda+1})\prec L(V_{\lambda+1})$ with $\crt(j)<\lambda$, let $\omega<\gamma<\lambda$ regular. Then in a generic extension the following is true: I0($\lambda$) still holds and in $L(V_{\lambda+1})$ there are $S_1$ and $S_2$ disjoint stationary sets such that $S_1\cup S_2=S^{\lambda^+}_\gamma$.
\end{prop}
\begin{proof}
 First of all, notice that $\lambda^+$ is coded by subsets of $\lambda$, so it can be considered a subset of $V_{\lambda+1}$. As $V_{\lambda+1}$ is closed by $\lambda$-sequences, the cofinality of a cardinal less than $\lambda^+$ is the same in $V$ and in $L(V_{\lambda+1})$ (because is witnessed by an element of $V_{\lambda+1}$), therefore $(S^{\lambda^+}_\eta)^V=(S^{\lambda^+}_\eta)^{L(V_{\lambda+1})}$ for any $\eta$ regular. Let $\mathbb{P}\in L(V_{\lambda+1})$ be $Coll(\gamma^+,\gamma)$. Then by \ref{nosmallforcing} I0($\lambda$) holds in $V[G]$. As $\gamma>\omega$, $\mathbb{P}$ is $\omega$-closed. Let $G$ be $\mathbb{P}$ generic. Consider $(S^{\lambda^+}_{\gamma^+})^V$ and $(S^{\lambda^+}_\gamma)^V$, both in $L(V_{\lambda+1})$ and therefore in $L(V_{\lambda+1})[G]$. Since $\mathbb{P}$ does not add unbounded sets in $\lambda^+$, $(S^{\lambda^+}_{\gamma^+})^V$ and $(S^{\lambda^+}_\gamma)^V$ are still stationary in $L(V_{\lambda+1})[G]$, are both subsets of $(S^{\lambda^+}_\gamma)^{V[G]}=(S^{\lambda^+}_\gamma)^{L(V_{\lambda+1})[G]}$ and of course are disjoint. Since $\mathbb{P}$ is $\omega$-closed, by \ref{omegaclosed} $L(V_{\lambda+1})[G]=L(V[G]_{\lambda+1})$, so $(S^{\lambda^+}_{\gamma^+})^V,\ (S^{\lambda^+}_\gamma)^V\in L(V[G]_{\lambda+1})$ are $S_1$ and $S_2$. 
\end{proof}

Note that if $\gamma=\omega$ then the proof breaks down: then $L(V_{\lambda+1})[G]\neq L(V[G]_{\lambda+1})$ and it is not clear why $S_1$ and $S_2$ would be in $L(V[G]_{\lambda+1})$. The following problems are still open:

\begin{Q}
 Is the Ultrafilter Axiom consistent with I0?
\end{Q}

\begin{Q}
 Is it consistent with I0 to have the club filter an ultrafilter on cofinality $\omega$? Could it be a consequence of I0?
\end{Q}

Towards a solution of the second problem, there are the following results:

\begin{teo}[Woodin, Theorem 176 in \cite{Woodin}]
 Suppose I0($\lambda$). Then there is no partition of $S^{\lambda^+}_\omega$ into two stationary sets such that all the sets definable from them are $\mathbb{U}(j)$-representable\footnote{This rather clumsy statement comes from the fact that the proof uses the fact that some $Z_S$ defined from the stationary set $S$ is $\mathbb{U}(j)$-representable. If all the sets in some $L_\delta(V_{\lambda+1})$, with $\delta$ limit, are $\mathbb{U}(j)$-representable, this implies that we cannot find two disjoint stationary sets in $L_\delta(V_{\lambda+1})$. If all the sets are $\mathbb{U}(j)$-representable, this means that the club filter is an ultrafilter on cofinality $\omega$.}.
\end{teo}

\begin{teo}[Cramer, Corollary 4.6 in \cite {Cramer}]
 \label{partitionS}
 Suppose there exists $j:L_\omega(V_{\lambda+1}^\sharp, V_{\lambda+1})\prec L_\omega(V_{\lambda+1}^\sharp, V_{\lambda+1})$. Then there is no partition of $S^{\lambda^+}_\omega$ into two stationary (in $V$) sets such that $S_1,\ S_2\in L(V_{\lambda+1})$.
\end{teo}

Note that if $S$ is stationary in $V$, then it is also stationary in $L(V_{\lambda+1})$, but not necessarily viceversa, therefore it is still open whether there can be two disjoint subsets of $S^{\lambda^+}_\omega$ that are in $L(V_{\lambda+1})$, stationary there, but not stationary in $V$.

Hypotheses like in Theorem \ref{partitionS} are investigated further in Section \ref{Icarus}. It is proven using inverse limits.

Towards a solution for the first problem, we have a weakening of the Ultrafilter Axiom. Theorem \ref{measurable} says that we can split $S^{\lambda^+}_\gamma$ in $\langle S_\alpha:\alpha<\eta_\gamma\rangle$ disjoint stationary sets for any $\gamma<\lambda^+$ regular. The Ultrafilter Axiom says that $\eta_\gamma=1$ for any $\gamma<\lambda^+$ regular. The Weak Ultrafilter Axiom states that $\eta_\gamma\leq\gamma^+$ for any $\gamma<\lambda^+$ regular.

\begin{teo}[Woodin, Theorem 16 in \cite{Woodin2}]
 Suppose I0($\lambda$). Then in a generic extension I0($\lambda$) + the Weak Ultrafilter Axiom hold.
\end{teo}

Another problematic point is the Wadge hierarchy. Wadge's Lemma states that under \AD{} for any two subsets $A,B$ of $\mathbb{R}$ either there is a continuous function $f$ with $A=f^{-1}[B]$ or there is one with $B=f^{-1}[\mathbb{R}\setminus A]$. Therefore any two subsets of $\mathbb{R}$ are reducible one to another. Under I0 this is consistently false:

\begin{teo}[Woodin, Theorem 174 in \cite{Woodin}]
 Suppose I0($\lambda$). Suppose $c$ is a generic Cohen real of $V$. Then there exist $X,Y\in V[c]_{\lambda+2}$ such that $X\notin L_\omega(Y,V[c]_{\lambda+1})$ and $Y\notin L_\omega(X,V[c]_{\lambda+1})$. 
\end{teo}

As by \ref{nosmallforcing} in $L(V[c]_{\lambda+1})$ I0($\lambda$) holds, this proves the consistency of I0 with a negation of Wadge Lemma.

Another aspect one can consider is partition combinatorics. In $L(\mathbb{R})$, under \AD, for all $\alpha<\omega_1$, $\omega_1\rightarrow(\omega_1)^\alpha_\omega$\footnote{Recall that $\kappa\rightarrow(\kappa)^\alpha_\beta$ is that if $\pi:\{\sigma\subseteq\kappa:\ot(\sigma)=\alpha\}\to\beta$ then there is a set $H\subseteq\kappa$ of cardinality $\kappa$ homogeneous for $\pi$, i.e., $|\{\pi(\sigma):\sigma\subseteq H,\ \ot(\sigma)=\alpha\}|=1$.}. A direct generalization always fails:

\begin{rem}[\ZF+\DC]
 Suppose there exists an $\omega_1$-sequence of distinct reals. Then there is no $\kappa$ such that $\kappa\rightarrow(\kappa)^{\omega_1}_2$.
\end{rem}

So the best attempt could only be $\lambda^+\rightarrow(\lambda^+)^\omega_\lambda$. It is still an open question even whether $\kappa\rightarrow(\kappa)^\omega_2$ is consistent at all. In \cite{Woodin} Lemma 212 proves that a universal strong version of $\lambda^+\rightarrow(\lambda^+)^\omega_\lambda$ is consistently false with I0.

But there is a partial result:

\begin{prop}[Lemma 180 in \cite{Woodin}]
 Suppose I0($\lambda$). Then for all $\alpha<\lambda$, $L(V_{\lambda+1})\vDash\lambda^+\rightarrow (\lambda^+)^\alpha_{(\lambda,<\lambda)}$, i.e., for all $\pi:\{\sigma\subseteq\lambda^+:\ot(\sigma)=\alpha\}\to\lambda$ there exists $H\subseteq\lambda^+$ such that $|\{\pi(\sigma):\sigma\subseteq H,\ \ot(\sigma)=\alpha\}|<\lambda$.
\end{prop}

Finally, we consider the measure on Turing degrees. A standard result of \AD{} is that the filter $U=\{A:\ A$ contains a cone$\}$ on sets of Turing degrees is an ultrafilter (Martin), i.e., every set of degrees either contains or is disjoint from a cone. This result is usually called Turing Determinacy, and it is at the base of many results, like Silver Dichotomy for every set in $L(\mathbb{R})$ under \AD.  In the context of I0 Turing degrees can be substituted with Zermelo degrees: given $a\subseteq\lambda$, let $M(a)$ be the smallest model of Zermelo set theory that contains $a$. We say that $a$ is equivalent with $b$ if $M(a)=M(b)$, and $a$ is reducible to $b$ if $M(a)\subseteq M(b)$. Then a Zermelo degree is a class of such equivalence relation. Unfortunately Zermelo Degree Determinacy is consistently false under I0:

\begin{teo}[Shi, 2015, \cite{Shi}]
 Suppose I0($\lambda$) and that all the $\lambda$-supercompact cardinals under $\lambda$ are Laver indestructible. Then there is a set of Zermelo degrees in $L(V_{\lambda+1})$ that neither contains nor is disjoint from a cone.
\end{teo}

This result is significantly more problematic respect the previous results: in \ref{noultrafilter} and Lemma 212 in \cite{Woodin} we had that small changes would destroy that $\AD^{L(\mathbb{R})}$-like properties, but in this case, modulo a reasonable indestructibility preparation, it seems that it is intrinsic from I0, and for deep reasons. But there is one ambiguity still to settle: for now, there is no proof that there are $\lambda$ singular cardinals of cofinality $\omega$ such that in $L(V_{\lambda+1})$ Zermelo Degree Determinacy holds, so maybe its failure is not something peculiar to I0, but a general property of such $\lambda$'s, just like it is not possible to have a club filter an ultrafilter everywhere or the full partition property.

\begin{Q}
 Is there a $\lambda$ such that Zermelo Degree Determinacy holds in $L(V_{\lambda+1})$?
\end{Q}

\section{Icarus sets}
\label{Icarus}

Is it possible to go beyond I0? The question is very interesting: if I0 holds, and since Reinhardt is inconsistent in \ZFC, there should be something in between, or maybe even a hierarchy between the two, where at a certain point the inconsistency comes up, so it is worth to investigate exactly where it is. But the question is not only speculative: results like \ref{atlambda1} are reflection results, from I0 to I1, so to find combinatorial properties that are independent from I0 one should have a starting point in a higher ground.

It seems that there is not much space above I0: already an elementary embedding from $L(V_{\lambda+2})$ to itself is inconsistent. We consider then models that are strictly between the two, i.e., $L(V_{\lambda+1})\subset M\subset L(V_{\lambda+2})$. They should be models of \ZF, so something like $L(N)$, with $V_{\lambda+1}\subset N\subset V_{\lambda+2}$, i.e., models built from collections of subsets of $V_{\lambda+1}$. In similarity with $\AD^{L(\mathbb{R})}$, the most important of such models are of the form $L(X,V_{\lambda+1})$, with $X\subset V_{\lambda+1}$:

\begin{defin}
 A set $X\subseteq V_{\lambda+1}$ is an Icarus set iff there is an elementary embedding $j:L(X,V_{\lambda+1})\prec L(X,V_{\lambda+1})$ with $\crt(j)<\lambda$. 
\end{defin}

The name comes from the fact that we are trying to go as high as possible before reaching an inconsistency. Trivially, if I0($\lambda$) holds, $V_{\lambda+1}$ (and in fact every set in $L(V_{\lambda+1})\cap V_{\lambda+2}$) is Icarus. On the other hand, a well-ordering of $V_{\lambda+1}$ cannot be Icarus, because of Kunen's theorem.

In this kind of models many structural properties of $L(V_{\lambda+1})$ still hold:

\begin{defin}
 Let $X\subseteq V_{\lambda+1}$. Then in $L(X,V_{\lambda+1})$ $\DC_\lambda$ holds and $\Theta=\Theta^{L(X,V_{\lambda+1})}_{V_{\lambda+1}}$ is regular.
\end{defin}

The structure of such embeddings is also very similar to the I0 ones:

\begin{teo}[Woodin, Lemma 5 in \cite{Woodin}]
 	Let $X\subseteq V_{\lambda+1}$ be an Icarus set. For every $j:L(X,V_{\lambda+1})\prec L(X,V_{\lambda+1})$ there exist a $L(X,V_{\lambda+1})$-ultrafilter $U$ in $V_{\lambda+1}$ and $j_U,k_U:L(X,V_{\lambda+1})\prec L(X,V_{\lambda+1})$ such that $j_U$ is the elementary embedding from $U$, $j=j_U\circ k_U$ and $j\upharpoonright L_\Theta(X,V_{\lambda+1})=j_U\upharpoonright L_\Theta(X,V_{\lambda+1})$.
\end{teo}
\begin{proof}
 The proof is the same as in \ref{allproper}. The only delicate point is to prove that the collapse of the ultrapower is indeed $L(X,V_{\lambda+1})$, as by elementarity ${\cal Z}\equiv L(j(X),V_{\lambda+1})$. But $j(X)\in L(X,V_{\lambda+1})$, and $X=\{a\in V_{\lambda+1}:j(a)\in j(X)\}$, therefore $L(j(X),V_{\lambda+1})=L(X,V_{\lambda+1})$. 
\end{proof}

There is one difference: in the I0 case we had that a weakly proper embedding was nicely amenable, and for every $a\in L(V_{\lambda+1})$ if we defined $a_0=a$ and $a_{n+1}=j(a_n)$, then $\langle a_n:n\in\omega\rangle\in L(V_{\lambda+1})$. This is because everything is definable from $V_{\lambda+1}$ and fixed points, and $V_{\lambda+1}$ is morally closed under $\omega$-sequences. But in the Icarus case it is not clear whether it is necessary that $\langle X_n:n\in\omega\rangle\in L(X,V_{\lambda+1})$. 

\begin{defin}[Woodin, \cite{Woodin}]
 Let $X\subseteq V_{\lambda+1}$ be an Icarus set, and let $j:L(X,V_{\lambda+1})\prec L(X,V_{\lambda+1})$ with $\crt(j)<\lambda$. Then:
 \begin{itemize}
  \item $j$ is weakly proper iff $j=j_U$;
	\item $j$ is proper iff it is weakly proper and $\langle X_n:n\in\omega\rangle\in L(X,V_{\lambda+1})$, with $X_0=X$ and $X_{n+1}=j(X_n)$. 
 \end{itemize}
\end{defin}

In the particular case of $X=V_{\lambda+1}$, then, weakly proper and proper are the same concept.

One could ask whether it is a good definition, as there can be a different $Z\subseteq V_{\lambda+1}$ such that $L(X,V_{\lambda+1})=L(Z,V_{\lambda+1})$. For this, we have to study the fixed points of $j$ under $\Theta$:

\begin{prop}
 Let $X\subseteq V_{\lambda+1}$ be an Icarus set and let $j:L(X,V_{\lambda+1})\prec L(X,V_{\lambda+1})$ with $\crt(j)<\lambda$. Then $\langle X_n:n\in\omega\rangle\in L(X,V_{\lambda+1})$ iff the fixed points of $j$ are cofinal in $\Theta$. 
\end{prop}
\begin{proof}
 Suppose that $\langle X_n:n\in\omega\rangle=Y\in L(X,V_{\lambda+1})$. We can suppose that $j$ is weakly proper. Then the proof is similar to the proof of Corollary \ref{FixedpointsbelowTheta}: In this case every ordinal $\beta<\Theta$ is definable from some elements of $I\cup V_{\lambda+1}\cup\{X\}$. As $Y\in L(X,V_{\lambda+1})$, $L(Y,V_{\lambda+1})=L(X,V_{\lambda+1})$, and $j(L_\omega(Y,V_{\lambda+1}))=L_\omega(Y,V_{\lambda+1})$. Let $\alpha_0<\Theta$, and suppose that $\alpha_0$ is definable from $\{i_0,\dots,i_n\}\cup L_\omega(Y,V_{\lambda+1})$. Let $\alpha$ be the supremum of the ordinals that are definable from $\{i_0,\dots,i_n\}\cup L_\omega(Y,V_{\lambda+1})$. Since $\Theta$ is regular, $\alpha<\Theta$, and since $j(L_\omega(Y,V_{\lambda+1}))=L_\omega(Y,V_{\lambda+1})$, $j(\alpha)=\alpha>\alpha_0$.

 Now suppose that the fixed points of $j$ are cofinal in $\Theta$. Let $j(X)\in L_\gamma(X,V_{\lambda+1})$ and $\alpha>\gamma$ such that $j(\alpha)=\alpha$. Then $j(L_\alpha(X,V_{\lambda+1}))=L_\alpha(X,V_{\lambda+1})$ and $L_\alpha(X,V_{\lambda+1})$ is closed by $\omega$-sequences, therefore $\langle j(X), j^2(X),\dots\rangle\in L_\alpha(V_{\lambda+1})$.
\end{proof}

\begin{prop}[Woodin, Lemma 6 of \cite{Woodin}]
 Let $X\subseteq V_{\lambda+1}$ and $j:L(X,V_{\lambda+1})\prec L(X,V_{\lambda+1})$ with $\crt(j)<\lambda$. Then if $\langle X_n:n\in\omega\rangle\in L(X,V_{\lambda+1})$, for all $Y\in L(X,V_{\lambda+1})\cap V_{\lambda+2}$, $\langle Y_n:n\in\omega\rangle\in L(X,V_{\lambda+1})$, with $Y_0=Y$ and $Y_{n+1}=j(Y_n)$.
\end{prop}
\begin{proof}
 Again, let $Y\in L_\gamma(X,V_{\lambda+1})$ and $\alpha>\gamma$ such that $j(\alpha)=\alpha$. Then $j(L_\alpha(X,V_{\lambda+1}))=L_\alpha(X,V_{\lambda+1})$ and $L_\alpha(X,V_{\lambda+1})$ is closed by $\omega$-sequences, therefore $\langle Y_n:n\in\omega\rangle\in L_\alpha(V_{\lambda+1})$.
\end{proof}

Not surprisingly, given that being proper implies some degree of amenability, properness is bound to iterability:

\begin{teo}[Woodin, Lemma 16 and Lemma 21 in \cite{Woodin}]
 Let $X\subseteq V_{\lambda+1}$ and $j:L(X,V_{\lambda+1})\prec L(X,V_{\lambda+1})$ with $\crt(j)<\lambda$. If $j$ is weakly proper, then it is finitely iterable. If $j$ is proper, then it is iterable.
\end{teo}

The Coding Lemma holds for any embedding, but the equivalent of \ref{manymeas} is proven in Lemma 207 of \cite{Woodin} only for the case where the embedding is proper. It is not known whether properness is necessary. In fact it is not even known whether it is a concept that is distinct from weak properness:

\begin{Q}
 Is there an Icarus set $X\subseteq V_{\lambda+1}$ and an elementary embedding $j:L(X,V_{\lambda+1})\prec L(X,V_{\lambda+1})$ that is weakly proper but not proper?
\end{Q}

We have seen that if we limit ourselves to I0 the answer is negative. Towards a solution for this problem in \cite{Dimonte1} and \cite{Dimonte2} it is proven that there are $N_1$ and $N_2$, $V_{\lambda+1}\subseteq N_1,N_2\subseteq V_{\lambda+2}$, such that there are embeddings weakly proper but non proper\footnote{We say that $j:L(N)\prec L(N)$ is proper if for any $X\in N$, $\langle X_n:n\in\omega\rangle\in L(N)$.} from $L(N_1)$ to itself, and such that all weakly proper embeddings from $L(N_2)$ to itself are not proper. But it is also proven that there cannot be $X\subseteq V_{\lambda+1}$ such that $L(N_1)=L(X,V_{\lambda+1})$ or $L(N_2)=L(X,V_{\lambda+1})$.

The existence of an Icarus set outside $L(V_{\lambda+1})\cap V_{\lambda+2}$ is not automatically an axiom strictly stronger than I0. We end this section introducing canonical Icarus sets, whose existence is strictly stronger than I0, and that form a hierarchy above it. 

\begin{defin}
 Let $A$ be a transitive set. Define ${\cal L}^+_A$ as LST, the language of set theory, extended with the constants $\{C\}\cup\{c_a\}_{a\in A}\cup\{d_i\}_{i\in\omega}$.

 Let $X\subseteq V_{\lambda+1}$. Suppose that in $L(X,V_{\lambda+1})$ there is a class of indiscernibles for formulas with parameters in $\{X\}\cup V_{\lambda+1}$. Then $X^\sharp$ is the theory in ${\cal L}^+_A$ of
\begin{equation*}
  (L(X,V_{\lambda+1}),\{X\}\cup V_{\lambda+1},\{a\}_{a\in \{X\}\cup V_{\lambda+1}},\{d_i\}_{i\in\omega}), 
\end{equation*}
where $d_i$ is the $i$-th indiscernible.
\end{defin}

We can consider therefore $(V_{\lambda+1})^\sharp$. Contrary to the $0^\sharp$ case, we do not know whether an I0 embedding implies the existence of $(V_{\lambda+1})^\sharp$: it is an inner model question, see Remark 28 in \cite{Woodin}. On the other hand, if $(V_{\lambda+1})^\sharp$ exists, with a shift of indiscernibles there is surely an embedding $k:L(V_{\lambda+1})\prec L(V_{\lambda+1})$, but it is not an I0 embedding, as $\crt(k)>\Theta$.

So we suppose that $(V_{\lambda+1})^\sharp$ exists. 

\begin{lem}
 Suppose that $(V_{\lambda+1})^\sharp$ exists. Then $\cof(\Theta^{L(V_{\lambda+1})}_{V_{\lambda+1}})=\omega$. 
\end{lem}
\begin{proof}
 By the theory of the sharp, if $(V_{\lambda+1})^\sharp$ exists every element of $L(V_{\lambda+1})$ is defined from a finite subset of $V_{\lambda+1}$ and finite indiscernibles. Let $\alpha_n$ be the supremum of the ordinals definable from $V_{\lambda+1}$ and the first $n$ indiscernibles. Then the supremum of the $\alpha_n$'s is $\Theta^{L(V_{\lambda+1})}_{V_{\lambda+1}}$.
\end{proof}

Therefore, if $(V_{\lambda+1})^\sharp$ exists, $L(V_{\lambda+1})$ is very far from $V$ (cfr. with Theorem \ref{manymeas}). 

What if $(V_{\lambda+1})^\sharp$ is an Icarus set? Note then that I0 is the same as $\exists j:(V_{\lambda+1},(V_{\lambda+1})^\sharp)\prec (V_{\lambda+1},(V_{\lambda+1})^\sharp)$, therefore not only $(V_{\lambda+1})^\sharp$ is Icarus is strictly stronger then I0($\lambda$), but any I0-embedding must be definable in $L((V_{\lambda+1})^\sharp, V_{\lambda+1})$. By the same proof of \ref{FirstStepI0}, therefore, $(V_{\lambda+1})^\sharp$ is Icarus strongly implies I0($\lambda$) (in fact, a $j:L_\omega((V_{\lambda+1})^\sharp, V_{\lambda+1})\prec L_\omega((V_{\lambda+1})^\sharp, V_{\lambda+1})$ suffices to carry on the construction of the inverse limits). 

If such construction is carefully analyzed, one can realize that the $\bar{\lambda}$ such that I0($\bar{\lambda}$) holds constructed in this way has more properties, as it is limit of critical points of embeddings:

\begin{teo}[Shi, Trang, 2017 \cite{ShiTrang}]
 Suppose that $(V_{\lambda+1})^\sharp$ is Icarus and let $\bar{\lambda}<\lambda$ be such that I0($\bar{\lambda}$) holds because of the inverse limit construction. Then $\bar{\lambda}$ is limit of $\bar{\lambda}^+$-supercompact cardinals. Therefore there is no $\bar{\lambda}^+$-Aronszajn tree, there is no good scale at $\bar{\lambda}$ and Stationary Reflection at $\bar{\lambda}^+$ is true. 
\end{teo}

This is interesting if compared to Theorem \ref{secondShiTrang}.

The following is the key to the analysis of Icarus sets:

\begin{teo}[Woodin, Lemma 213 in \cite{Woodin}]
 Suppose that $X\subseteq V_{\lambda+1}$ is an Icarus set, and let $Y\in L(X,V_{\lambda+1})\cap V_{\lambda+2}$ such that $\Theta^{L(Y,V_{\lambda+1})}<\Theta^{L(X,V_{\lambda+1})}$. Then $Y^\sharp$ exists and $Y^\sharp\in L(X,V_{\lambda+1})$. 
\end{teo}

So let $X\subseteq V_{\lambda+1}$ be an Icarus set such that $\Theta^{L(X,V_{\lambda+1})}>\Theta^{L(V_{\lambda+1})}$. Then $(V_{\lambda+1})^\sharp$ exists and, since it is definable, it is an Icarus set. Moreover $\Theta^{L(X,V_{\lambda+1})}\geq\Theta^{L((V_{\lambda+1})^\sharp,V_{\lambda+1})}$. In this sense $(V_{\lambda+1})^\sharp$ is a ``canonical'' Icarus set. If $X\subseteq V_{\lambda+1}$ is any Icarus set such that $\Theta{L(X,V_{\lambda+1})}>\Theta^{L((V_{\lambda+1})^\sharp,V_{\lambda+1})}$, then $(V_{\lambda+1})^{\sharp\sharp}$ exists, it is an Icarus set and $\Theta^{L(X,V_{\lambda+1})}\geq\Theta^{L((V_{\lambda+1})^{\sharp\sharp},V_{\lambda+1})}$. 

In other words, if $X\subseteq V_{\lambda+1}$ is an Icarus set then the largeness of $\Theta^{L(X,V_{\lambda+1})}$ ``measures'' how strong it actually is, as an axiom, and the $\Theta$'s of the models built from the sharps (therefore $\Theta^{L((V_{\lambda+1})^\sharp,V_{\lambda+1}}$, $\Theta^{L((V_{\lambda+1})^{\sharp\sharp},V_{\lambda+1}}$, $\dots$) are the reference for such measure. Also, the construction for strong implication still holds. Therefore we have a new hierarchy: 

\begin{prop}
 Let $\lambda$ be a cardinal. Then $V_{\lambda+1}$ is Icarus is strongly implied by $(V_{\lambda+1})^\sharp$ is Icarus, that is strongly implied by $(V_{\lambda+1})^{\sharp\sharp}$ is Icarus, etc. In other words, $(V_{\lambda+1})^{(n+1)\sharp}$ is Icarus strongly implies that $(V_{\lambda+1})^{n\sharp}$ is Icarus. Moreover, if $X\subseteq V_{\lambda+1}$ is Icarus and there exists $n\in\omega$ such that $\Theta^{L(X,V_{\lambda+1})}<\Theta^{L((V_{\lambda+1})^{n\sharp},V_{\lambda+1})}$, then there exists an $m<n$ such that $\Theta^{L(X,V_{\lambda+1})}=\Theta^{L((V_{\lambda+1})^{m\sharp},V_{\lambda+1})}$ and for any $s\leq m$ $(V_{\lambda+1})^{m\sharp}$ is Icarus.
\end{prop}

Pushing this result above the first $\omega$ sharps is more difficult. One can try to define $Y_{\alpha+1}=Y_\alpha^\sharp$ and $Y_\gamma=\bigcup_{\alpha<\gamma}Y_\alpha$ for $\gamma$ limit, but at a certain point there would be a $\gamma$ limit such that $Y_\gamma$ is not a subset of $V_{\lambda+1}$, so then $(Y_\gamma)^\sharp$ is still not a subset of $V_{\lambda+1}$, and so on. The way to do it is to enlarge our analysis to $L(N)$ with $V_{\lambda+1}\subset N\subset V_{\lambda+2}$, and if necessary at the successor of a limit stage adding just less of $N^\sharp$ so that everything is codeable with a subset of $V_{\lambda+1}$. It is a complex endeavour, that takes all Section 4 of \cite{Woodin}.

\section{Further Developments}
\label{further}

The main line of research right now is finding an answer for Question \ref{representable}. The research is ongoing and very promising: if the answer is yes, then not only many already known results will get unlocked (for example the Ultrafilter Axiom at $\omega$ or the full power of Generic Absoluteness), but it would give an elegant tree representation to all subsets of $V_{\lambda+1}$. The potential would be huge, just like what happened for  $\AD^{L(\mathbb{R})}$. 

Another approach would be to investigate the inner model theory of I0. Here the main result is this:

\begin{teo}[Woodin, 2011, \cite{Woodin0}, \cite{WooDavRod}]
 Let $\delta$ be an extendible cardinal. Assume that $N$ is a weak extender model for $\delta$ supercompact and $\gamma>\delta$ is a cardinal in $N$. Let $j:H(\gamma^+)^N\to H(j(\gamma)^+)^N$ be an elementary embedding with $\delta\leq\crt(j)$ and $j\neq\id$. Then $j\in N$.
\end{teo}

In other words, if there is a (reasonable) canonical inner model for supercompactness and I0 holds, then I0 must hold in the inner model. There is no hope therefore to reach results of consistency equivalence to I0 via inner models, and for this reason Woodin conceived Ultimate-$L$, a canonical inner model for all large cardinals. If such hypothesis will prove solid, then we would have an actual canonical model, a preferred $V$, to work with. This would overcome all the problems in the Section \ref{dissimil}, and moreover it would be interesting to understand the role of I0 in this setting.

\begin{Q}
 Is it true that Ultimate-$L\vDash I0(\lambda)$ iff Ultimate-$L\vDash L(V_{\lambda+1})\nvDash\AC$?
\end{Q}

One can ask if there are other axioms above I0, still in \ZFC. The most natural approach would be to add to $L(V_{\lambda+1})$ the measures that define $j$, just like for $L[\mu]$, but this is still open:

\begin{Q}
 Let $j:L(V_{\lambda+1})\prec L(V_{\lambda+1})$ weakly proper, and $U_j$ that defines it. Is $V\neq L(V_{\lambda+1})[U_j]$?
\end{Q}

Finally, there is the Reinhardt cardinal approach, i.e., large cardinals in \ZF{} without \AC. There is a whole hierarchy above, like super-Reinhardt cardinals and Berkeley cardinals, and it goes frontally against the Ultimate-$L$ approach. Yet, the situation about Reinhardt cardinals is similar to the I0 one: very few published results (for example \cite{AptSarg}), and a lot of underground activity, via lecture notes, seminar slides, Wikipedia pages, MathOverflow questions, $\dots$. The possible intersections with inner model theory, though, promise to make this topic a catalyst of set theorists' attention in the near future, and the hope is that this will bring more publications, and maybe even a survey, of this exciting topic. 

\emph{Acknowledgments}. The paper was written under the Italian program ``Rita Levi Montalcini 2013'', but it is a collection of notes, thoughts, sketches and ideas that started in the Summer Semester of 2008 at Berkeley, propelled by a course on ``Beyond I0'' by Woodin. With nine years in the making it is therefore a hard task to thank all the people that helped me writing this paper, as almost all the people I interacted with in my academic career has been an inspiration for this, directly or indirectly. Knowing that this is just a small percentage of the whole picture, I would like to thank first of all Hugh Woodin, the demiurge of the I0 world, that welcomed me in Berkeley when I was still a student and continued to be a creative influence during the years; Alessandro Andretta, for his steady support in all this years, and because of his (decisive) insistence for this survey to be written. I would also like to thank Sy Friedman and Liuzhen Wu, that helped to push my research in an unexpected direction, and Scott Cramer and Xianghui Shi, for the many, many discussions on I0 and similars. Special thanks also for Luca Motto Ros, that helped me in shaping this paper.

\end{document}